\documentclass[11pt]{amsart}
\usepackage{amsmath, amssymb, latexsym}
\usepackage{mathrsfs}
\usepackage{enumerate}
\usepackage[all, knot]{xy}
\usepackage{color}

\newtheorem{thm}{Theorem}[section]
\newtheorem{prop}[thm]{Proposition}
\newtheorem{cor}[thm]{Corollary}
\newtheorem{lemma}[thm]{Lemma}

\theoremstyle{definition}
\newtheorem{defn}[thm]{Definition}
\newtheorem{Ex}[thm]{Example}

\theoremstyle{remark}
\newtheorem{Rmk}[thm]{Remark}

\newenvironment{red}
{\relax\color{red}}
{\hspace*{.5ex}\relax}

\newcommand{\ber}{\begin{red}}
\newcommand{\er}{\end{red}}

\newenvironment{verd}
{\relax\color{magenta}}
{\hspace*{.5ex}\relax}

\newcommand{\bg}{\begin{verd}}
\newcommand{\eg}{\end{verd}}

\numberwithin{equation}{subsection}
\newcommand{\rank}{\operatorname{rank}}

\newcommand{\Z}{\mathbb{Z}}

\newcommand{\g}{\mathfrak{g}}

\newcommand{\Hom}{\mathrm{Hom}}
\newcommand{\End}{\mathrm{End}}
\newcommand{\Ext}{\mathrm{Ext}}

\newcommand{\wt}{{\rm wt}}

\newcommand{\Span}{{\rm Span}}

\newcommand{\Top}{{\rm Top}}
\newcommand{\Soc}{{\rm Soc}}
\newcommand{\Rad}{{\rm Rad}}
\newcommand{\Mat}{{\rm Mat}}

\renewcommand{\Im}{\mathrm{Im}}
\newcommand{\Ker}{\mathrm{Ker}}
\newcommand{\Coker}{\mathrm{Coker}}
\newcommand{\id}{\mathrm{id}}
\newcommand{\Tr}{\mathrm{Tr}}

\newcommand{\rlQ}{\mathsf{Q}}   % root lattice
\newcommand{\wlP}{\mathsf{P}}   % weight lattice
\newcommand{\weyl}{\mathsf{W}}  % Weyl group
\newcommand{\cmA}{\mathsf{A}}  % Cartan matrix

\newcommand{\ST}{\mathsf{ST}}   % standard tablueax
\newcommand{\res}{\mathrm{res}}   % residue sequance
\newcommand{\bR}{\mathbf{k}}   % base ring
\newcommand{\fqH}{R^{\Lambda_0}}   % finite quiver Hecke algebra

\newcommand{\sBox}[1]
{
\xy
(-2,-2)*{};(2,-2)*{} **\dir{-};
(-2,2)*{};(2,2)*{} **\dir{-};
(-2,-2)*{};(-2,2)*{} **\dir{-};
(2,-2)*{};(2,2)*{} **\dir{-};
(0,0)*{#1};
\endxy
}

\newcommand{\ssBox}[1]
{
\xy
(-1.5,-1.5)*{};(1.5,-1.5)*{} **\dir{-};
(-1.5,1.5)*{};(1.5,1.5)*{} **\dir{-};
(-1.5,-1.5)*{};(-1.5,1.5)*{} **\dir{-};
(1.5,-1.5)*{};(1.5,1.5)*{} **\dir{-};
(0,0)*{ _{#1}};
\endxy
}

\begin{document}

\title[Representation type of finite quiver Hecke algebras of type $A^{(1)}_{\ell}$]
{Representation type of finite quiver \\
Hecke algebras of type $A^{(1)}_{\ell}$ \\
for arbitrary parameters}

\author[Susumu Ariki]{Susumu Ariki $^1$}
\thanks{$^1$ S.A. is supported in part by JSPS, Grant-in-Aid for Scientific Research (B) 23340006.}
\address{Department of Pure and Applied Mathematics, Graduate School of Information
Science and Technology, Osaka University, Toyonaka, Osaka 560-0043, Japan}
\email{ariki@ist.osaka-u.ac.jp}

\author[Kazuto Iijima]{Kazuto Iijima}
\address{Suzuka National College of Technology, Shiroko-cho, Suzuka, Mie 510-0294, Japan}
\email{iijima@genl.suzuka-ct.ac.jp}

\author[Euiyong Park]{Euiyong Park} 
\address{Department of Mathematics, University of Seoul, Seoul 130-743, Korea}
\email{epark@uos.ac.kr}

%\date{\today}
%\subjclass[2010]{14F05, 17B67, 81R10}
%\keywords{}

\begin{abstract}
We give Erdmann-Nakano type theorem for the finite quiver Hecke algebras $\fqH(\beta)$
of affine type $A^{(1)}_{\ell}$. Note that each finite quiver Hecke algebra lies in one parameter family, and
the original Erdmann-Nakano theorem studied the finite quiver Hecke algebra at a
special parameter value. We study the general case in our paper.
Our result shows in particular that their representation type does not depend on the parameter.
Moreover, when the parameter value is nonzero, we show that finite quiver Hecke algebras of tame representation type are biserial algebras.
\end{abstract}

\maketitle

%\tableofcontents

\vskip 2em

\section*{Introduction}

This paper is the third of our series of papers on the representation type of finite quiver Hecke algebras.
The affine type we treat in this paper is the affine type $A$. Thus, it includes the original case of
block algebras of the Hecke algebras associated with the symmetric group. The latter classical Hecke algebras appeared
in many branches of mathematics: knot theory, mathematical physics, number theory, geometric representation theory and so on,
but recent progress reveals new features of the algebras as we explain in the next paragraph.

Let $q\in\bR^\times$ and
$e=\min\{ k \mid 1+q+\cdots+q^{k-1}=0 \}$. Block algebras of the Hecke algebra $\mathcal{H}_n(q)$ associated
with the symmetric group of degree $n$ are labelled by $e$-cores $\kappa$, and we denote them by $B_\kappa(n,q)$.
As is well-known, the Lascoux-Leclerc-Thibon conjecture on the decomposition numbers of the Hecke algebras
inspired the first author and he introduced categorification scheme for integrable highest weight modules $V(\Lambda)$
over the Kac-Moody Lie algebra of affine type $A$. When $\Lambda=\Lambda_0$, the basic module $V(\Lambda_0)$ is categorified by the $B_\kappa(n,q)$'s.
It was vastly generalized and refined after the introduction of
Khovanov-Lauda-Rouquier algebras, which we may call affine quiver Hecke algebras. Cyclotomic quotients of
affine quiver Hecke algebras were introduced by Khovanov and Lauda, and the
integrable module $V(\Lambda)$ is categorified by the cyclotomic quiver Hecke algebras.
The categorification scheme itself was strengthened by Rouquier from our weak form, by generalizing the
Chuang-Rouquier $sl_2$-categorification. Indeed, these results together with Kang and Kashiwara's result on adjoint pairs of
induction and restriction functors have been essentially used in our series of papers. When $\Lambda=\Lambda_0$, we call cyclotomic quiver Hecke algebras
\emph{finite quiver Hecke algebras}. Our interest lies in the study of finite quiver Hecke algebras. 
New features arising from the above mentioned development are two fold:
\begin{itemize}
\item[(i)]
each $B_\kappa(n,q)$ lies in a one-parameter family of finite quiver Hecke algebras,
\item[(ii)]
the finite quiver Hecke algebras are graded algebras, i.e. they have the KLR grading.
\end{itemize}

The second point was already taken up by Brundan and Kleshchev \cite{BK}. They refined the first author's categorification theorem into graded version.
See also Mathas' proof in \cite{Ma13}.

In this paper, we take up the first point.
We consider finite quiver Hecke algebras $\fqH(\beta)$ with parameters $\lambda\in \bR$. They fall into various isomorphism classes. 
Nevertheless, we may show that the representation type of the finite quiver Hecke algebra $\fqH(\beta)$
does not depend on the parameter $\lambda$ (Theorem \ref{main theorem}). This is our first result.
The proof follows our strategy in \cite{AP12} and \cite{AP13},
but we need extra work for the case $\ell=1$. This generalized Erdmann-Nakano theorem tells that
$\fqH(\beta)$ has tame representation type only when $\ell=1$. Thus, we give a detailed study of tame finite quiver Hecke algebras.
Recall that tame block algebras $B_\kappa(n,-1)$ are Morita equivalent to either $B_{(0)}(4,-1)$ or $B_{(1)}(5,-1)$ by Scopes' equivalence. 
We may only show that tame finite quiver Hecke algebras for $\lambda\ne0$ are biserial algebras. However, 
if it is special biserial then it is Morita equivalent to either $B_{(0)}(4,-1)$ or $B_{(1)}(5,-1)$. This is our second result
(Theorem \ref{second classification}). We can also describe possible form of
tame finite quiver Hecke algebras by concrete quiver presentation when $\lambda=0$. 
For the proof of the second result, we first classify two-point symmetric special biserial algebras (Theorem \ref{first classification}). 
It is a bit surprise that this classification also seems to be new. 

The paper is organized as follows. \S1 and \S2 are for preliminaries. We construct irreducible $\fqH(\delta-\alpha_i)$-modules, for $1\le i\le \ell$,
and irreducible $\fqH(\delta)$-modules in \S3. The proof in \S3 follows the same strategy as \cite{AP12} and \cite{AP13}.
Then, we need extra work for $\fqH(2\delta)$ in $\ell=1$ case, which is the topic of \S4.
In \S5, we prove the generalized Erdmann-Nakano theorem, the first main result. In \S6, we classify two-point symmetric special biserial algebras and
then prove the second main result. The appendix is for explaining some interesting results from \cite{Po94}. Although main results in \cite{Po94} are incorrect, 
we may use his setup to prove, at the end of \S6, that tame finite quiver Hecke algebras in affine type A, for parameter values $\lambda\ne0$, are biserial algebras.

\vskip 1em

\section{Preliminaries}

In this section, we briefly recall necessary materials.

\subsection{Cartan datum} \label{Sec: Cartan datum}

Let $I = \{0,1, \ldots, \ell \}$ be an index set, and let $\cmA$ be the {\it affine Cartan matrix} of type $A_{\ell}^{(1)}$ $(\ell\ge2)$
$$ \cmA = (a_{ij})_{i,j\in I} = \left(
                                  \begin{array}{ccccccc}
                                    2  & -1 & 0  & \ldots & 0  & 0  & -1 \\
                                    -1 &  2 & -1 & \ldots & 0  & 0  & 0 \\
                                    0  & -1 & 2  & \ldots & 0  & 0  & 0 \\
                                    \vdots   &  \vdots  &  \vdots  & \ddots &  \vdots  &  \vdots  & \vdots \\
                                    0  & 0  & 0  & \ldots & 2  & -1 & 0 \\
                                    0  & 0  & 0  & \ldots & -1 & 2  & -1 \\
                                   -1  & 0  & 0  & \ldots & 0  & -1 & 2 \\
                                  \end{array}
                                \right).
  $$
When $\ell=1$, the affine Cartan matrix is
$$
\cmA = (a_{ij})_{i,j\in I} = \begin{pmatrix} 2  & -2 \\ -2 & 2 \end{pmatrix}.
$$
We have an affine Cartan datum $(\cmA, \wlP, \Pi, \Pi^{\vee})$, where
\begin{itemize}
\item[(1)] $\cmA$ is the affine Cartan matrix given above.
\item[(2)] $\wlP$ is the weight lattice, a free abelian group of rank $\ell+1$.
\item[(3)] $\Pi = \{ \alpha_i \mid i\in I \} \subset \wlP$, the set of simple roots.
\item[(4)] $\Pi^{\vee} = \{ h_i \mid i\in I\} \subset \wlP^{\vee} := \Hom( \wlP, \Z )$.
\end{itemize}
They satisfy the following properties:
\begin{itemize}
\item[(a)] $\langle h_i, \alpha_j \rangle  = a_{ij}$ for all $i,j\in I$,
\item[(b)] $\Pi$ and $\Pi^{\vee}$ are linearly independent sets.
\end{itemize}

We fix a {\it scaling element} $d$ which obeys the condition $\langle d, \alpha_i\rangle=\delta_{i0}$,
and assume that $\Pi^{\vee}$ and $d$ form a $\Z$-basis of $P^{\vee}$. Then, the fundamental weight $\Lambda_0$ is defined by
$$\langle h_i, \Lambda_0\rangle=\delta_{i0},\quad \langle d,\Lambda_0\rangle=0.$$

The free abelian group $\rlQ = \bigoplus_{i \in I} \Z \alpha_i$ is the root lattice,  and we denote $\rlQ^+ = \sum_{i\in I} \Z_{\ge 0} \alpha_i$.
If $\beta\in\rlQ^+$, we denote the sum of coefficients by $|\beta|$. 
The Weyl group $\weyl$ associated with $\cmA$ is the affine symmetric group, which is generated by $\{r_i\}_{i\in I}$, where
$r_i\Lambda=\Lambda-\langle h_i, \Lambda\rangle\alpha_i$, for $\Lambda\in P$. The null root of type $A^{(1)}_{\ell}$ is
$$ \delta = \alpha_0 + \alpha_1 + \cdots + \alpha_\ell. $$
Note that $ \langle h_i, \delta\rangle = 0$ and $w \delta = \delta$, for $i\in I$ and $w\in \weyl$.

\subsection{Young diagrams}
Let $\lambda= (\lambda_1 \ge \lambda_2 \ge \ldots\ge \lambda_l > 0)$ be a Young diagram of depth $l$.
We denote the depth $l$ by $l(\lambda)$. If $\lambda$ has $n$ nodes, we write $\lambda \vdash n$.
For each $\lambda$, $\ST(\lambda)$ is the set of standard tableaux of shape $\lambda$.
We consider the residue pattern which repeats
\begin{equation} \label{Eq: residue pattern}
0\ 1\ 2\ \dots \ \ell
\end{equation}
in the first row, and we shift the residue pattern to the right by one in the next row. Namely,
the residue of the $(i,j)$-node of $\lambda$ is defined to be
$$\res(i,j)\equiv j-i\mod(\ell+1).$$
For example, if $\ell=3$ and $\lambda = (10,7,3)$, the residues are given as follows:

$$
\xy
(0,12)*{};(60,12)*{} **\dir{-};
(0,6)*{};(60,6)*{} **\dir{-};
(0,0)*{};(42,0)*{} **\dir{-};
(0,-6)*{};(18,-6)*{} **\dir{-};
(0,12)*{};(0,-6)*{} **\dir{-};
(6,12)*{};(6,-6)*{} **\dir{-};
(12,12)*{};(12,-6)*{} **\dir{-};
(18,12)*{};(18,-6)*{} **\dir{-};
(24,12)*{};(24,0)*{} **\dir{-};
(30,12)*{};(30,0)*{} **\dir{-};
(36,12)*{};(36,0)*{} **\dir{-};
(42,12)*{};(42,0)*{} **\dir{-};
(48,12)*{};(48,6)*{} **\dir{-};
(54,12)*{};(54,6)*{} **\dir{-};
(60,12)*{};(60,6)*{} **\dir{-};
(3,9)*{0}; (9,9)*{1}; (15,9)*{2}; (21,9)*{3}; (27,9)*{0}; (33,9)*{1}; (39,9)*{2};(45,9)*{3};(51,9)*{0}; (57,9)*{1};
(3,3)*{3}; (9,3)*{0}; (15,3)*{1}; (21,3)*{2}; (27,3)*{3}; (33,3)*{0}; (39,3)*{1};
(3,-3)*{2}; (9,-3)*{3}; (15,-3)*{0};
\endxy
$$

\bigskip
\noindent
Then, for $T\in \ST(\lambda)$, we define the {\it residue sequence} of $T$ by
$$ \res(T) = (\res_1(T), \res_2(T), \ldots, \res_n(T) )\in I^n,$$
where $\res_k(T)$ is the residue of the box filled with $k$ in $T$, for $1\leq k\leq n$.

\subsection{The combinatorial Fock space} \label{Sec: Fock space}

Let $\mathcal{F}$ be the complex vector space whose basis is given by
$\{ |\lambda \rangle \mid \lambda : \text{Young diagrams} \}$. Then $\mathcal{F}$ is a $\g$-module, where
$\g$ is the Kac-Moody algebra associated with $\mathsf{A}$.
To describe the action, we use the following notation:
\begin{itemize}
\item if we may remove a box of residue $i$ from $\lambda$ and obtain a new Young diagram, then we write $\lambda \nearrow \sBox{i}$ for the resulting diagram,
\item if we may add a box of residue $i$ to $\lambda$ and obtain a new Young diagram, then we write $\lambda \swarrow \sBox{i}$ for the resulting diagram.
\end{itemize}
Then, the action of the Chevalley generators $f_i$ and $e_i$, for $i\in I$, is given as follows.
\begin{align} \label{Eq: Kashiwara operators}
e_i |\lambda\rangle = \sum_{\mu = \lambda \nearrow \ssBox{i} } |\mu\rangle, \qquad f_i |\lambda\rangle = \sum_{\mu = \lambda \swarrow \ssBox{i} } |\mu\rangle,
\end{align}

\vskip 1em

\section{Quiver Hecke algebras}

\subsection{Quiver Hecke algebras}  \label{Sec: quiver Hecke algs}
Let $\bR$ be an algebraically closed field, and let
$(\cmA, \wlP, \Pi, \Pi^{\vee})$ be a Cartan datum. We assume that $\cmA$ is symmetrizable.
We denote the symmetrized bilinear form on $\wlP$ by $(\hphantom{-}| \hphantom{-})$.

We take polynomials $\mathcal{Q}_{i,j}(u,v)\in\bR[u,v]$, for $i,j\in I$,
of the form
\begin{align*}
\mathcal{Q}_{i,j}(u,v) = \left\{
                 \begin{array}{ll}
                   \sum_{p(\alpha_i|\alpha_i)+q (\alpha_j|\alpha_j) + 2(\alpha_i|\alpha_j)=0} t_{i,j;p,q} u^pv^q & \hbox{if } i \ne j,\\
                   0 & \hbox{if } i=j,
                 \end{array}
               \right.
\end{align*}
where $t_{i,j;p,q} \in \bR$ are such that $t_{i,j;-a_{ij},0} \ne 0$ and $\mathcal{Q}_{i,j}(u,v) = \mathcal{Q}_{j,i}(v,u)$.
The symmetric group $\mathfrak{S}_n = \langle s_k \mid k=1, \ldots, n-1 \rangle$ acts on $I^n$ by place permutations.

\begin{defn} \
Let $\Lambda \in \wlP^+$. The {\it cyclotomic quiver Hecke algebra} $R^{\Lambda}(n)$ associated with
$(\mathcal{Q}_{i,j}(u,v))_{i,j\in I}$ is the $\Z$-graded  $\bR$-algebra defined by generators
$e(\nu)$, for $\nu = (\nu_1,\ldots, \nu_n) \in I^n$, $x_1,\dots,x_n$, $\psi_1,\dots, \psi_{n-1}$ and the following relations.

\begin{align*}
& e(\nu) e(\nu') = \delta_{\nu,\nu'} e(\nu),\ \sum_{\nu \in I^{n}} e(\nu)=1,\
x_k e(\nu) =  e(\nu) x_k, \  x_k x_l = x_l x_k,\\
& \psi_l e(\nu) = e(s_l(\nu)) \psi_l,\  \psi_k \psi_l = \psi_l \psi_k \text{ if } |k - l| > 1, \\[5pt]
&  \psi_k^2 e(\nu) = \mathcal{Q}_{\nu_k, \nu_{k+1}}(x_k, x_{k+1}) e(\nu), \\[5pt]
&  (\psi_k x_l - x_{s_k(l)} \psi_k ) e(\nu) = \left\{
                                                           \begin{array}{ll}
                                                             -  e(\nu) & \hbox{if } l=k \text{ and } \nu_k = \nu_{k+1}, \\
                                                               e(\nu) & \hbox{if } l = k+1 \text{ and } \nu_k = \nu_{k+1},  \\
                                                             0 & \hbox{otherwise,}
                                                           \end{array}
                                                         \right. \\[5pt]
&( \psi_{k+1} \psi_{k} \psi_{k+1} - \psi_{k} \psi_{k+1} \psi_{k} )  e(\nu) \\[4pt]
&\qquad \qquad \qquad = \left\{
                                                                                   \begin{array}{ll}
\displaystyle \frac{\mathcal{Q}_{\nu_k,\nu_{k+1}}(x_k,x_{k+1}) -
\mathcal{Q}_{\nu_k,\nu_{k+1}}(x_{k+2},x_{k+1})}{x_{k}-x_{k+2}} e(\nu) & \hbox{if } \nu_k = \nu_{k+2}, \\
0 & \hbox{otherwise}, \end{array}
\right.\\[5pt]
& x_1^{\langle h_{\nu_1}, \Lambda \rangle} e(\nu)=0.
\end{align*}
\end{defn}

\bigskip
The $\Z$-grading on $R^\Lambda(n)$ is given as follows:
\begin{align*}
\deg(e(\nu))=0, \quad \deg(x_k e(\nu))= ( \alpha_{\nu_k} |\alpha_{\nu_k}), \quad  \deg(\psi_l e(\nu))= -(\alpha_{\nu_{l}} | \alpha_{\nu_{l+1}}).
\end{align*}

\vskip 1em

Lemma \ref{Kashiwara isom} ( see \cite[p.25]{Ro08}, for example) below shows that each cyclotomic quiver Hecke algebra for the affine type $A^{(1)}_\ell$
lies in one parameter family. Namely, we may assume without loss of generality that
$$
\mathcal{Q}_{0,1}(u,v)=u^2+\lambda uv+v^2,
$$
for the affine type $A^{(1)}_1$, and
\begin{align*}
\mathcal{Q}_{i,i+1}(u,v)&=u+v\;\;(0\le i\le \ell-1),\\
\mathcal{Q}_{\ell,0}(u,v)&=u+\lambda v,\\
\mathcal{Q}_{i,j}(u,v)&=1\;\;(j\not\equiv i\pm1\mod(\ell+1)),
\end{align*}
for the affine type $A^{(1)}_\ell$ with $\ell\ge2$. Here $\lambda\in\bR$ is a parameter.
In the rest of the paper, we assume that $\mathcal{Q}_{i,j}(u,v)$ are of this form.

\begin{lemma}\label{Kashiwara isom}
Let $C=(c_{ij})_{i,j\in I}\in\Mat(I,I,\bR)$ be a symmetric matrix such that $c_{ij}\ne0$, for all $i$ and all $j$. If we define
$\mathcal{Q}_{i,j}'(u,v)=c_{ij}^2\mathcal{Q}_{i,j}(c_{ii}u, c_{jj}v)$, then
the cyclotomic quiver Hecke algebra associated with $\{\mathcal{Q}_{i,j}'(u,v)\}_{i,j\in I}$ is isomorphic to
the cyclotomic quiver Hecke algebra associated with $\{\mathcal{Q}_{i,j}(u,v)\}_{i,j\in I}$ by the algebra isomorphism
$$
e(\nu)\mapsto e(\nu),\quad x_ke(\nu)\mapsto c_{\nu_k\nu_k}^{-1}x_ke(\nu),\quad \psi_ke(\nu)\mapsto c_{\nu_k\nu_{k+1}}\psi_ke(\nu).
$$
\end{lemma}

\begin{Rmk}
We studied affine types $A^{(2)}_{2\ell}$ and $D^{(2)}_{\ell+1}$ in
\cite{AP12} and \cite{AP13}. Lemma \ref{Kashiwara isom} shows that $\fqH(n)$ does not depend on
the polynomials $\{\mathcal{Q}_{i,j}(u,v)\}_{i,j\in I}$ in those types. This is the reason why the choice of the polynomials did not matter 
in \cite{AP12} and \cite{AP13}.
\end{Rmk}

For $\beta\in \rlQ^+$ with $|\beta|=n$, we define the central idempotent $e(\beta)$ of $R^\Lambda(n)$ by
$$
e(\beta) = \sum_{\nu \in I^\beta} e(\nu),\quad \text{where}\;\;
I^\beta = \left\{ \nu=(\nu_1, \ldots, \nu_n ) \in I^n \mid \sum_{k=1}^n\alpha_{\nu_k} = \beta \right\}
$$
Then we denote $R^\Lambda(\beta) = R^\Lambda(n) e(\beta)$.
We will be interested in the case when $\Lambda=\Lambda_0$.
We call $\fqH(\beta)$ \emph{finite quiver Hecke algebras of type $A_{\ell}^{(1)}$}. 
The finite quiver Hecke algebras categorify the highest weight $\g$-module $V(\Lambda_0)$. 
Thus, Chevalley generators $E_i$ and $F_i$ are categorified to exact functors, and irreducible modules over finite quiver Hecke algebras of various rank 
are labelled by the Kashiwara crystal $B(\Lambda_0)$. We use standard notations $(\wt, \tilde{e}_i, \tilde{f}_i, \varepsilon_i, \varphi_i)$ for 
the crystal structure. As this is the third of our series of papers, we assume that the reader is familar with the 
strategy used in our series. The following is a consequence of derived equivalence explained in \cite{AP12}.

\begin{prop} [\protect{\cite[Cor.4.8]{AP12}}] \label{Prop: repn type}
For $w \in \weyl$ and $k\in \Z_{\ge0}$, $R^\Lambda(k\delta)$ and $R^\Lambda(\Lambda - w\Lambda + k\delta)$
have the same number of irreducible modules and the same representation type.
\end{prop}

\subsection{Dimension formula for $\fqH(\beta)$}

For $\lambda \vdash n$ and $\nu \in I^n$, define a non-negative integer $K(\lambda,\nu)$ by
$$
K(\lambda, \nu) = | \{ T\in \ST(\lambda) \mid \nu = \res(T)  \} |.
$$
Then we may give a formula for $\dim \fqH(\beta)$ in terms of $K(\lambda,\nu)$.
The proof is entirely similar to the cases $A^{(2)}_{2\ell}$ and $D^{(2)}_{\ell+1}$ and
we omit the proof. Note that $\wt(\lambda)=\Lambda_0-\sum_{(i,j)\in\lambda} \alpha_{-i+j}$.

\begin{thm} \label{Thm: dimension formula}
For $\beta \in \rlQ^+$ with $|\beta|=n$ and $\nu, \nu' \in I^\beta$, we have
\begin{align*}
\dim e(\nu')\fqH(n)e(\nu) &= \sum_{\lambda \vdash n} K(\lambda, \nu') K(\lambda, \nu), \\
\dim \fqH(\beta) &= \sum_{\lambda \vdash n,\ \wt(\lambda)= \Lambda_0 - \beta} |\ST(\lambda)|^2, \\
\dim \fqH(n) &= \sum_{\lambda \vdash n} |\ST(\lambda)|^2= n!\,.
\end{align*}
\end{thm}

\vskip 1em

\section{Irreducible representations of $\fqH(\delta)$} \label{Sec: repns}

By explicit computation, we have
\begin{align*}
r_{\ell-1}r_{\ell-2}\cdots r_1r_0\Lambda_0&=\Lambda_0-\delta+\alpha_\ell,\\
r_kr_{k-1}(\Lambda_0-\delta+\alpha_k)&=\Lambda_0-\delta+\alpha_{k-1},\;\;\text{for $2\le k \le\ell$.}
\end{align*}
Hence $\Lambda_0-\delta+\alpha_i\in \weyl\Lambda_0$, which implies that
$\fqH(\delta-\alpha_i)$, for $1\le i\le \ell$, are simple algebras. In particular, $\fqH(\delta-\alpha_i)$ has a unique
irreducible module if $i\ne0$.

\subsection{Representations of $\fqH(\delta-\alpha_i)$} \label{Section: repn of R(delta-alpha)}
In this section, we give explicit description of the irreducible $\fqH(\delta-\alpha_i)$-module.
We will use the result to determine the structure of $\fqH(\delta)$.

We consider a hook partition $\lambda(i)= (i, 1^{\ell-i}) \vdash \ell$ of weight $\Lambda_0 - \delta + \alpha_i$,
for $1\le i\le\ell$. We shall define $\fqH(\delta-\alpha_i)$-modules $\mathcal{L}_i$, for $1\le i\le\ell$. Let
$$ \mathcal{L}_i = \bigoplus_{T\in\ST(\lambda(i))} \bR T.$$

\begin{lemma}
We may define an $\fqH(\delta-\alpha_i)$-module structure on $\mathcal{L}_i$, for $1\le i\le\ell$, by
$x_k=0$ and
\begin{equation} \label{Eq: Li}
\begin{aligned}
 e(\nu) T &= \left\{
                \begin{array}{ll}
                  T & \hbox{ if } \nu = \res(T), \\
                  0 & \hbox{ otherwise,}
                \end{array}
              \right. \\
\psi_k T &= \left\{
                         \begin{array}{ll}
                           s_k T  & \hbox{ if $s_k T$ is standard,}   \\
                           0 & \hbox{ otherwise.}
                         \end{array}
                       \right.
\end{aligned}
\end{equation}
\end{lemma}
\begin{proof}
To check the defining relations on $T$, we may assume that
$$\nu=(\nu_1, \nu_2, \ldots, \nu_\ell)=\res(T).$$
The relation for $(\psi_k x_l - x_{s_k(l)} \psi_k)e(\nu)$ is clear because $\nu_k=\nu_{k+1}$ does not occur.
Next observe that one of the following holds.
\begin{itemize}
\item[(i)]
$a_{\nu_k\nu_{k+1}}=0$ and $s_kT$ is standard.
\item[(ii)]
$s_kT$ is not standard and $\deg\mathcal{Q}_{\nu_k,\nu_{k+1}}(u,v)>0$.
\end{itemize}
Further, $x_k=0, x_{k+1}=0$ imply $\mathcal{Q}_{\nu_k,\nu_{k+1}}(x_k, x_{k+1})=0$ in the latter case.
Hence the relation for $\psi_k^2e(\nu)$ holds. Finally, we prove $\psi_{k+1}\psi_{k}\psi_{k+1}T=\psi_{k}\psi_{k+1}\psi_{k}T$,
because $\nu_k=\nu_{k+2}$ does not occur. But it follows from the fact that $s_{k+1}T, s_ks_{k+1}T, s_{k+1}s_ks_{k+1}T$ are all standard
if and only if $s_kT,s_{k+1}s_kT, s_ks_{k+1}s_kT$ are all standard. It is straightforward to check
other relations.
\end{proof}

\begin{lemma} \label{Lem: L_i}
The $\fqH(\delta-\alpha_i)$-module $\mathcal{L}_i$, for $1\le i\le\ell$, is irreducible of dimension $\binom{\ell-1}{i-1}$.
\end{lemma}
\begin{proof}
Note that $\dim e(\nu)\mathcal{L}_i\le 1$. Then the standard argument shows that $\mathcal{L}_i$ is irreducible.
Its dimension is $|\ST(\lambda(i))|=\binom{\ell-1}{i-1}$.
\end{proof}

\subsection{Representations of $\fqH(\delta)$} \label{Section: repn of R(2delta)}
In this section, we construct irreducible $\fqH(\delta)$-modules by extending the modules $\mathcal{L}_i$
from Section \ref{Section: repn of R(delta-alpha)}.

\begin{lemma} \label{S for i}
Let $h=\ell+1$. By declaring that $x_{h}$ and $\psi_{h-1}$ act as $0$, and $e(\nu)$, for $\nu\in I^{\delta}$, as
$$ e(\nu) T = \left\{
                           \begin{array}{ll}
                             T & \hbox{ if } \nu = \res(T) * i, \\
                             0 & \hbox{ otherwise,}
                           \end{array}
                         \right.
$$
where $\res(T) * i$ is the concatenation of $\res(T)$ and $(i)$, the irreducible $\fqH(\delta-\alpha_i)$-module
$\mathcal{L}_i$ extends to an irreducible $\fqH(\delta)$-module, for $1\le i\le\ell$.
\end{lemma}
\begin{proof}
We may assume that $\nu = \res(T)*i$. As $(\nu_{h-1},\nu_h)=(i\pm1,i)$, $\nu_{h-1}=\nu_h$ does not occur. Thus,
the relation for $(\psi_{h-1}x_l-x_{s_k(l)}\psi_{h-1})e(\nu)$ holds. Since $\deg\mathcal{Q}_{\nu_{h-1},\nu_h}(u,v)>0$, it also follows
$\mathcal{Q}{\nu_{h-1},\nu_h}(x_{h-1},x_h)=0$. Hence, we have also proved the relation for $\psi_{h-1}^2e(\nu)$. Finally,
$\nu_{h-2}\ne i=\nu_h$ if $\ell\ge2$ implies the relation for $(\psi_{h-1}\psi_{h-2}\psi_{h-1}-\psi_{h-2}\psi_{h-1}\psi_{h-2})e(\nu)$.
It is easy to check the remaining defining relations.
\end{proof}

\begin{defn}
We denote the irreducible $\fqH(\delta)$-module defined in Lemma \ref{S for i} by $\mathcal{S}_i$,
for $i=1, 2, \dots, \ell$.
\end{defn}

By construction, $ E_j(\mathcal{S}_i) = \delta_{ij}$, for $1\le i,j\le \ell$.
Hence $\mathcal{S}_1 \dots \mathcal{S}_\ell$ are pairwise non-isomorphic $\fqH(\delta)$-modules.
But we know from the categorication theorem that the number of irreducible $\fqH(\delta)$-modules is $\ell$.
Hence we have the following lemma.

\begin{lemma}
The modules $\mathcal{S}_i$ $(i=1, 2, \dots, \ell)$ form a complete list of irreducible $\fqH(\delta)$-modules.
\end{lemma}

We show that $\fqH(\delta) \simeq \bR[x]/(x^2)$ if $\ell=1$. As the Young diagrams $(2)$ and $(1,1)$ are all for contributing to $\dim \fqH(\delta )$ by Theorem \ref{Thm: dimension formula}, we have
$ \dim \fqH(\delta) = 2 $. Then, the map $\fqH(\delta) \rightarrow \bR[x]/(x^2)$ given by
\begin{align*}
e(\nu) \mapsto \left\{
                                  \begin{array}{ll}
                                    1 & \hbox{ if } \nu = (0,1), \\
                                    0 & \hbox{ otherwise, }
                                  \end{array}
                                \right.
\ x_k \mapsto \left\{
                                  \begin{array}{ll}
                                    0 & \hbox{ if } k = 1, \\
                                    x & \hbox{ if } k = 2,
                                  \end{array}
                                \right.
\
\psi_1 \mapsto 0,
\end{align*}
defines a surjective algebra homomorphism, which is an isomorphism by $ \dim \fqH(\delta) = 2 $.

\vskip 1em

\section{Representations of $\fqH(\beta)$ when $\ell=1$} \label{Sec: repns for ell=1}

\subsection{Representations of $\fqH(2\delta - \alpha_i)$ for $\ell = 1$} \label{Sec: repn of R(2delta-alpha_i) l=1}
For the case $\ell = 1$, we have chosen the parameter
$ \mathcal{Q}_{0,1}(u,v)=u^2+\lambda uv+v^2 $ in Section \ref{Sec: quiver Hecke algs}. Note that
$$ \frac{  \mathcal{Q}_{i,j}(u,v) - \mathcal{Q}_{i,j}(w,v) }{u-w} = u + \lambda v + w $$
for $i \ne j \in I$.

As the Young diagram $(2,1)$ is a unique one to contribute to $\dim \fqH(2\delta - \alpha_0)$ by Theorem \ref{Thm: dimension formula}, we have
$$ \dim \fqH(2\delta - \alpha_0) = 4 .$$
Proposition \ref{Prop: repn type} and $ r_1r_0 (\Lambda_0) = \Lambda_0 - 2\delta + \alpha_0 $ imply
$ \fqH(2\delta - \alpha_0) \simeq \Mat(2, \bR)$. More explicitly, it is easy to check that the correspondence
\begin{align*}
& e(\nu ) \mapsto \left\{
                                  \begin{array}{ll}
                                    (\begin{smallmatrix} 1 & 0 \\ 0 & 1 \end{smallmatrix}) & \hbox{ if } \nu = (0,1,1), \\
                                    (\begin{smallmatrix} 0 & 0 \\ 0 & 0 \end{smallmatrix}) & \hbox{ otherwise, }
                                  \end{array}
                                \right.
\  x_k \mapsto \left\{
                                  \begin{array}{ll}
                                    (\begin{smallmatrix} 0 & 0 \\ 0 & 0 \end{smallmatrix}) & \hbox{ if } k = 1, \\
                                    (\begin{smallmatrix} 0 & 1 \\ 0 & 0 \end{smallmatrix}) & \hbox{ if } k = 2, \\
                                    (\begin{smallmatrix} 0 & -1 \\ 0 & 0 \end{smallmatrix}) & \hbox{ if } k = 3,
                                  \end{array}
                                \right.
\
\psi_l \mapsto \left\{
                                  \begin{array}{ll}
                                    (\begin{smallmatrix} 0 & 0 \\ 0 & 0 \end{smallmatrix}) & \hbox{ if } l = 1, \\
                                    (\begin{smallmatrix} 0 & 0 \\ -1 & 0 \end{smallmatrix}) & \hbox{ if } l = 2,
                                  \end{array}
                                \right.
\end{align*}
defines an irreducible representation of $\fqH(2\delta - \alpha_0)$ and dimension counting shows that it
gives an isomorphism between $ \fqH(2\delta - \alpha_0) $ and $\Mat(2, \bR)$. Let
$$\mathcal{M}_0 := \Span_{\bR}\{ \mathsf{v}_1 := (\begin{smallmatrix} 1 \\ 0  \end{smallmatrix}), \mathsf{v}_2 := (\begin{smallmatrix} 0 \\ 1  \end{smallmatrix}) \}$$
be the corresponding irreducible $\fqH(2\delta - \alpha_0)$-module.

Similarly, the Young diagrams $(3)$ and $(1,1,1)$ are all for contributing to $\dim \fqH(2\delta - \alpha_1)$ and we have
$$ \dim \fqH(2\delta - \alpha_1) = 2 .$$ It is straightforward to check that
\begin{align*}
e(\nu) \mapsto \left\{
                                  \begin{array}{ll}
                                    (\begin{smallmatrix} 1 & 0 \\ 0 & 1 \end{smallmatrix}) & \hbox{ if } \nu = (0,1,0), \\
                                    (\begin{smallmatrix} 0 & 0 \\ 0 & 0 \end{smallmatrix}) & \hbox{ otherwise, }
                                  \end{array}
                                \right.
\ x_k \mapsto \left\{
                                  \begin{array}{ll}
                                    (\begin{smallmatrix} 0 & 0 \\ 0 & 0 \end{smallmatrix}) & \hbox{ if } k = 1, \\
                                    (\begin{smallmatrix} 0 & 1 \\ 0 & 0 \end{smallmatrix}) & \hbox{ if } k = 2, \\
                                    (\begin{smallmatrix} 0 & - \lambda \\ 0 & 0 \end{smallmatrix}) & \hbox{ if } k = 3,
                                  \end{array}
                                \right.
\
\hspace{3mm}
\psi_1, \psi_2 \mapsto (\begin{smallmatrix} 0 & 0 \\ 0 & 0 \end{smallmatrix}),
\end{align*}
is a well-defined representation of $\fqH(2\delta - \alpha_1)$. Hence, we have an algebra isomorphism $\fqH(2\delta - \alpha_1) \simeq \bR[x]/(x^2)$.
Let
$$\widehat{\mathcal{M}}_1 := \Span_{\bR}\{ \mathsf{w}_1 := (\begin{smallmatrix} 1 \\ 0  \end{smallmatrix}), \mathsf{w}_2  := (\begin{smallmatrix} 0 \\ 1  \end{smallmatrix})   \}$$
be the corresponding $\fqH(2\delta - \alpha_1)$-module and let $\mathcal{M}_1$ be the irreducible quotient of $\widehat{\mathcal{M}}_1$.
The module $\widehat{\mathcal{M}}_1$ is uniserial of length $2$ and the composition factors are $\mathcal{M}_1$.
Thus, $\widehat{\mathcal{M}}_1$ is an indecomposable projective $\fqH(2\delta - \alpha_1)$-module. Further,
\begin{align} \label{Eq: M0 M1}
E_0 \mathcal{M}_0 = 0, \quad E_1 \widehat{\mathcal{M}}_1 = 0, \quad E_1 \mathcal{M}_0 \simeq  E_0 \widehat{\mathcal{M}}_1,
\end{align}
and $ E_1 \mathcal{M}_0 \simeq  E_0 \widehat{\mathcal{M}}_1$ is the regular representation of $\fqH(\delta) \simeq \bR[x]/(x^2)$ via $x\mapsto x_2$.

\subsection{Representations of $\fqH(2\delta)$ for $\ell = 1$} \label{Sec: repn of R(2delta) ell=1}

\begin{lemma} \label{Lem: repn of R(2delta)}
\begin{enumerate}
\item By declaring that $x_4$ and $\psi_3$ act as $0$, and $e(\nu)$, for $\nu \in I^4$, as
$$ e(\nu) \mathsf{v}_i = \left\{
                           \begin{array}{ll}
                             \mathsf{v}_i & \hbox{ if } \nu = (0110), \\
                             0 & \hbox{ otherwise},
                           \end{array}
                         \right.
 $$
the irreducible $\fqH(2\delta-\alpha_0)$-module $\mathcal{M}_0$ in Section \ref{Sec: repn of R(2delta-alpha_i) l=1} extends to an irreducible $\fqH(2\delta)$-module.
\item
By declaring that $\psi_3$ acts as $0$, and $x_4$, $e(\nu)$, for $\nu \in I^4$, act as
$$ x_4 \mathsf{w}_i = \left\{
                        \begin{array}{ll}
                          0 & \hbox{ if } i = 1, \\
                          (\lambda^2 - 1) \mathsf{w}_1 & \hbox{ if } i =2,
                        \end{array}
                      \right. \quad
 e(\nu) \mathsf{w}_i = \left\{
                           \begin{array}{ll}
                             \mathsf{w}_i & \hbox{ if } \nu = (0101), \\
                             0 & \hbox{ otherwise},
                           \end{array}
                         \right.
 $$
the $\fqH(2\delta-\alpha_1)$-module $\widehat{\mathcal{M}}_1$ in Section \ref{Sec: repn of R(2delta-alpha_i) l=1} extends to an $\fqH(2\delta)$-module.
\end{enumerate}
\end{lemma}
\begin{proof}
(1) It is straightforward to check the defining relations except for $( \psi_3 x_k - x_{s_3(k)}\psi_3 )e(0110) $, $(\psi_3\psi_2\psi_3 - \psi_2\psi_3\psi_2) e(0110)$ and $\psi_3^2 e(0110)$.
For the remaining relations, we have
\begin{align*}
&( \psi_3 x_k - x_{s_3(k)}\psi_3 )e(0110) \mathsf{v}_i = 0, \\
& (\psi_3\psi_2\psi_3 - \psi_2\psi_3\psi_2) e(0110)  \mathsf{v}_i = 0, \\
& \psi_3^2 e(0110)  \mathsf{v}_i = 0 = (x_3^2+\lambda x_3 x_4 +x_4 ^2)e(0110) \mathsf{v}_i,\quad\text{for $i=1,2$},
\end{align*}
by direct computation, the action is well-defined.

(2) Similarly, it is easy to verify the defining relations except for
$( \psi_3 x_k - x_{s_3(k)}\psi_3 )e(0101) $, $(\psi_3\psi_2\psi_3 - \psi_2\psi_3\psi_2) e(0101)$ and $\psi_3^2 e(0101)$, and we have
\begin{align*}
&( \psi_3 x_k - x_{s_3(k)}\psi_3 )e(0101) \mathsf{w}_i = 0, \\
& (\psi_3\psi_2\psi_3 - \psi_2\psi_3\psi_2) e(0101)  \mathsf{w}_i = 0 = (x_2 + \lambda x_3 + x_4)e(0101) \mathsf{w}_i, \\
& \psi_3^2 e(0101)  \mathsf{w}_i = 0 = (x_3^2+\lambda x_3 x_4 +x_4 ^2)e(0101) \mathsf{w}_i.
\end{align*}
Hence the action is well-defined.
\end{proof}
We denote by $\mathcal{N}_0$ (resp.\ $\widehat{\mathcal{N}}_1$) the $\fqH(2\delta)$-module defined in Lemma \ref{Lem: repn of R(2delta)} (1) (resp.\ (2)),
and let $\mathcal{N}_1$ be the irreducible quotient of $\widehat{\mathcal{N}}_1$. Note that $\mathcal{M}_1$ extends to $\mathcal{N}_1$ by declaring
that $\psi_3$ and $x_4$ and $e(\nu)$, for $\nu\ne(0101)$, act as $0$ and $e(0101)$ acts as $1$. Hence, $\widehat{\mathcal{N}}_1$ is uniserial of length $2$
whose composition factors are $\mathcal{N}_1$.
By construction, $\mathcal{N}_0$ and $\mathcal{N}_1$ are irreducible and they are non-isomorphic
since
\begin{align} \label{Eq: N0 N1}
E_i(\mathcal{N}_j) \simeq \delta_{ij} \mathcal{M}_i.
\end{align}
 As the categorification theorem tells that the number of irreducible $\fqH(2\delta)$-modules is 2, we have the following lemma.

\begin{lemma} \label{Lem: irr of R(2delta)}
The modules $\mathcal{N}_0$ and $\mathcal{N}_1$ form a complete list of irreducible $\fqH(2\delta)$-modules.
\end{lemma}

We now construct some $\fqH(2\delta)$-modules which will be used for proving that $\fqH(2\delta)$ has tame representation type.

\begin{lemma} \label{Lem: T0}
If $\lambda = 0$, there exists a uniserial $\fqH(2\delta)$-module $\mathcal{T}_0$ whose radical series is
$$ \mathcal{T}_0 / \Rad (\mathcal{T}_0) \simeq \mathcal{N}_0, \quad
\Rad (\mathcal{T}_0) / \Rad^2 (\mathcal{T}_0 ) \simeq \mathcal{N}_1, \quad
\Rad^2 (\mathcal{T}_0) \simeq \mathcal{N}_1.  $$
\end{lemma}
\begin{proof}
% Recall the $\fqH(2\delta)$-modules $\mathcal{N}_0$ and $\widehat{\mathcal{N}}_1$ from Lemma \ref{Lem: repn of R(2delta)}.
We change the action of $\psi_3$ on
\begin{align} \label{Eq: T0}
\mathcal{T}_0 = \mathcal{N}_0 \oplus \widehat{\mathcal{N}}_1
\end{align}
to $\psi_3 \mathsf{v}_i = \mathsf{w}_i$, for $i=1,2$.
Then, $\lambda = 0 $ in mind, direct computation shows
\begin{align*}
& \psi_3^2 e(0110) \mathsf{v}_i = \psi_3 e(0101) \mathsf{w}_i = 0 =  (x_3^2 + \lambda x_3 x_4 + x_4^2)  e(0110) \mathsf{v}_i, \\
& (\psi_3\psi_2\psi_3 - \psi_2\psi_3\psi_2) e(0110) \mathsf{v}_1 = \psi_3\psi_2 \mathsf{w}_1 - \psi_2\psi_3 (-\mathsf{v}_2) = \psi_2 \mathsf{w}_2  =  0 , \\
& (\psi_3\psi_2\psi_3 - \psi_2\psi_3\psi_2) e(0110) \mathsf{v}_2 = \psi_3\psi_2 \mathsf{w}_2 =  0 , \\
& (\psi_3 x_3 - x_{4}\psi_3) e(0110) \mathsf{v}_1 = - x_{4} \mathsf{w}_1 = 0, \\
& (\psi_3 x_3 - x_{4}\psi_3) e(0110) \mathsf{v}_2 = \psi_3 (- \mathsf{v}_1) - x_{4} \mathsf{w}_2 = -\mathsf{w}_1 - (\lambda^2 - 1) \mathsf{w}_1= - \lambda^2  \mathsf{w}_1 = 0, \\
& (\psi_3 x_4 - x_{3}\psi_3) e(0110) \mathsf{v}_1 = -\, x_{3} \mathsf{w}_1 = 0, \\
& (\psi_3 x_4 - x_{3}\psi_3) e(0110) \mathsf{v}_2 = -\, x_{3} \mathsf{w}_2 = \lambda \mathsf{w}_1 = 0,
\end{align*}
which implies that $\mathcal{T}_0$ is well-defined.
As $\mathcal{N}_0$ is irreducible, $\Rad(\mathcal{T}_0) \subseteq  \widehat{\mathcal{N}}_1 $. Thus, we have either
$\Rad(\mathcal{T}_0)=\Soc(\widehat{\mathcal{N}}_1)$ or $\Rad(\mathcal{T}_0) = \widehat{\mathcal{N}}_1 $.
On the other hand, $\psi_3\mathsf{v}_2=\mathsf{w}_2 \not\in \Soc(\widehat{\mathcal{N}}_1)$ implies that
$\mathcal{T}_0/\Soc(\widehat{\mathcal{N}}_1)$ is not a semisimple module.
Thus, we conclude that $\Rad(\mathcal{T}_0) \simeq  \widehat{\mathcal{N}}_1 $ and $ \Rad^2(\mathcal{T}_0)  \simeq  \mathcal{N}_1$.
\end{proof}
\begin{Rmk}
If $\lambda \ne 0$, then $\mathcal{T}_0$ is not well-defined.
\end{Rmk}

\begin{lemma} \label{Lem: T1}
If $\lambda \ne 0$, there exists a uniserial $\fqH(2\delta)$-module $\mathcal{T}_1$ whose radical series is
$$ \mathcal{T}_1 / \Rad (\mathcal{T}_1) \simeq \mathcal{N}_0, \quad
\Rad (\mathcal{T}_1) / \Rad^2 (\mathcal{T}_1 ) \simeq \mathcal{N}_1, \quad
\Rad^2 (\mathcal{T}_1 ) \simeq \mathcal{N}_0.  $$
\end{lemma}
\begin{proof}
We choose the basis $\{ \mathsf{v}_1, \mathsf{v}_2 \}$ of $\mathcal{N}_0$ as in the definition, and we denote the basis of
$\mathcal{N}_1$ by $\{ \mathsf{w}\} $. Let
$$ \mathcal{T}_1 = \mathcal{N}_0 \oplus \mathcal{N}_1 \oplus \mathcal{N}_0 . $$
To distinguish the third direct summand $\mathcal{N}_0$ of $\mathcal{T}_1$ from the first direct summand $\mathcal{N}_0$,
we rename the basis $\{ \mathsf{v}_1, \mathsf{v}_2 \}$ for the third direct summand $\mathcal{N}_0$ to $\{ \tilde{\mathsf{v}}_1, \tilde{\mathsf{v}}_2  \}$.
Thus, we have
\begin{align} \label{Eq: T1}
\mathcal{T}_1 = \Span_\bR\{ \mathsf{v}_1, \mathsf{v}_2, \mathsf{w} , \tilde{\mathsf{v}}_1, \tilde{\mathsf{v}}_2 \}.
\end{align}
We keep the action of $e(\nu)$, $x_k$ and $\psi_l$ on $\mathcal{T}_1$ unchanged except
\begin{align*}
x_4 \mathsf{v}_i = \tilde{\mathsf{v}}_i\ (i=1,2),\quad \psi_3 \mathsf{v}_1 = 0, \quad  \psi_3 \mathsf{v}_2 = - \lambda \mathsf{w}, \quad
\psi_3 \mathsf{w} = \tilde{\mathsf{v}}_1.
\end{align*}
To show that $\mathcal{T}_1$ is well-defined, it suffices to check $x_kx_4=x_4x_k$ and the defining relations for $\psi_3^2$, $\psi_3 x_4 - x_3 \psi_3$,
$\psi_3 x_3 - x_4 \psi_3$, $\psi_3 x_2 - x_2 \psi_3$ and $\psi_3\psi_2\psi_3 - \psi_2\psi_3\psi_2$ on $\{ \mathsf{v}_1, \mathsf{v}_2, \mathsf{w} \} $.
It is easy to check $x_kx_4=x_4x_k$. By direct computation, we obtain
\begin{align*}
&\psi_3^2 e(0110) \mathsf{v}_1 = 0 = \lambda x_3 \tilde{\mathsf{v}}_1 + x_4 \tilde{\mathsf{v}}_1 = (x_3^2 + \lambda x_3x_4 + x_4^2) \mathsf{v}_1, \\
&\psi_3^2 e(0110) \mathsf{v}_2 = \psi_3 (-\lambda \mathsf{w}) = -\lambda \tilde{\mathsf{v}}_1 = x_3(-\mathsf{v}_1) + \lambda x_3 \tilde{\mathsf{v}}_2 + x_4 \tilde{\mathsf{v}}_2
 = (x_3^2 + \lambda x_3x_4 + x_4^2) \mathsf{v}_2, \\
& \psi_3^2 e(0101) \mathsf{w} = \psi_3 \tilde{\mathsf{v}}_1 = 0 =  (x_3^2 + \lambda x_3x_4 + x_4^2) \mathsf{w}.
\end{align*}
Hence the relation for $\psi_3^2$ holds. The equations
\begin{align*}
& (\psi_3 x_4 - x_3 \psi_3) e(0110) \mathsf{v}_1 = \psi_3 \tilde{\mathsf{v}}_1 = 0, \\
& (\psi_3 x_3 - x_4 \psi_3) e(0110) \mathsf{v}_1 =  0, \\
& (\psi_3 x_4 - x_3 \psi_3) e(0110) \mathsf{v}_2 = \psi_3 \tilde{\mathsf{v}}_2 - x_3(-\lambda \mathsf{w}) = 0, \\
& (\psi_3 x_3 - x_4 \psi_3) e(0110) \mathsf{v}_2 = \psi_3 (- \mathsf{v}_1) - x_4 (-\lambda \mathsf{w}) = 0, \\
& (\psi_3 x_4 - x_3 \psi_3) e(0101) \mathsf{w} = -\, x_3 \tilde{\mathsf{v}}_1 = 0,  \\
& (\psi_3 x_3 - x_4 \psi_3) e(0101) \mathsf{w} = -\, x_4 \tilde{\mathsf{v}}_1 = 0,  \\
& (\psi_3 x_2 - x_2 \psi_3) e(0101) \mathsf{v}_1 = 0,    \\
& (\psi_3 x_2 - x_2 \psi_3) e(0101) \mathsf{v}_2 = \psi_3 \mathsf{v}_1 - x_2 (-\lambda\mathsf{w}) = 0,  \allowdisplaybreaks   \\
& (\psi_3 x_2 - x_2 \psi_3) e(0101) \mathsf{w} = 0,
\end{align*}
imply that the relations for $\psi_3 x_4 - x_3 \psi_3$, $\psi_3 x_3 - x_4 \psi_3$ and $\psi_3 x_2 - x_2 \psi_3$ hold, and
\begin{align*}
& (\psi_3\psi_2\psi_3 - \psi_2\psi_3\psi_2) e(0110) \mathsf{v}_1 = -\, \psi_2\psi_3 (- \mathsf{v}_2 ) = \psi_2 ( - \lambda \mathsf{w} ) = 0 , \\
& (\psi_3\psi_2\psi_3 - \psi_2\psi_3\psi_2) e(0110) \mathsf{v}_2 =  \psi_3\psi_2 (- \lambda \mathsf{w} ) =  0 , \\
& (\psi_3\psi_2\psi_3 - \psi_2\psi_3\psi_2) e(0101) \mathsf{w} =  \psi_3\psi_2 \tilde{\mathsf{v}}_1  = \psi_3 ( -\tilde{\mathsf{v}}_2 ) = 0
= ( x_2 + \lambda x_3 + x_4 ) e(0101) \mathsf{w}
\end{align*}
verify the relation for $\psi_3\psi_2\psi_3 - \psi_2\psi_3\psi_2$. Thus, $\mathcal{T}_1$ is well-defined.

We show that $\mathcal{T}_1$ is uniserial. It is clear that
$\Rad(\mathcal{T}_1) \subseteq \Span_{\bR}\{  \mathsf{w} , \tilde{\mathsf{v}}_1, \tilde{\mathsf{v}}_2 \}$. Then
$\psi_3\mathsf{w} = \tilde{\mathsf{v}}_1$ implies that $\Span_{\bR}\{  \mathsf{w} , \tilde{\mathsf{v}}_1, \tilde{\mathsf{v}}_2 \}$ is
not semisimple, and we have either
$$
\Rad(\mathcal{T}_1) = \Span_{\bR}\{  \mathsf{w} , \tilde{\mathsf{v}}_1, \tilde{\mathsf{v}}_2 \} \quad \text{or}\quad
\Rad(\mathcal{T}_1) = \Span_{\bR}\{  \tilde{\mathsf{v}}_1, \tilde{\mathsf{v}}_2 \}.
$$
Since $\mathcal{T}_1/\Span_{\bR}\{  \tilde{\mathsf{v}}_1, \tilde{\mathsf{v}}_2 \}$ is not semisimple by $\psi_3\mathsf{v}_2 = -\lambda \mathsf{w}$,
we have
$$ \Rad(\mathcal{T}_1) = \Span_{\bR}\{  \mathsf{w} , \tilde{\mathsf{v}}_1, \tilde{\mathsf{v}}_2 \}, \quad
\Rad^2(\mathcal{T}_1) = \Span_{\bR}\{  \tilde{\mathsf{v}}_1, \tilde{\mathsf{v}}_2 \}, $$
which completes the proof.
\end{proof}

\begin{Rmk}
If $\lambda = 0$, then $\Rad(\mathcal{T}_1) \simeq \mathcal{N}_0$ in the proof of Lemma \ref{Lem: T1}.
\end{Rmk}

\begin{lemma} \label{Lem: hat T1}
If $\lambda \ne 0$, there exists a uniserial $\fqH(2\delta)$-module $\widehat{\mathcal{T}}_1$ whose radical series is
$$
\widehat{\mathcal{T}}_1 / \Rad (\widehat{\mathcal{T}}_1) \simeq \mathcal{N}_1, \;\;
\Rad (\widehat{\mathcal{T}}_1) / \Rad^2 (\widehat{\mathcal{T}}_1 ) \simeq \mathcal{N}_0, \;\;
\Rad^2(\widehat{\mathcal{T}}_1) / \Rad^3 (\widehat{\mathcal{T}}_1 ) \simeq \mathcal{N}_1, \;\;
\Rad^4 (\widehat{\mathcal{T}}_1 ) \simeq \mathcal{N}_0.
$$
\end{lemma}
\begin{proof}
Recall the $\fqH(2\delta)$-module $\mathcal{T}_1 = \Span_\bR\{ \mathsf{v}_1, \mathsf{v}_2, \mathsf{w} , \tilde{\mathsf{v}}_1, \tilde{\mathsf{v}}_2 \}$
given in $\eqref{Eq: T1}$. Let
\begin{align} \label{Eq: hat T1}
\widehat{\mathcal{T}}_1 = \mathcal{T}_1 \oplus \bR \mathsf{u},
\end{align}
where $\mathcal{T}_1$ is its submodule, and define
\begin{align*}
e(\nu) \mathsf{u} = \left\{
           \begin{array}{ll}
             \mathsf{u} & \hbox{ if } \nu = (0101), \\
             0 & \hbox{ otherwise,}
           \end{array}
         \right.
\quad x_k \mathsf{u} = \left\{
           \begin{array}{ll}
             0 & \hbox{ if } k = 1, \\
             \mathsf{w} & \hbox{ if } k = 2, \\
            -\lambda \mathsf{w} & \hbox{ if } k = 3,\\
            -\mathsf{w} & \hbox{ if } k = 4,
           \end{array}
         \right.
\quad \psi_l \mathsf{u} = \left\{
           \begin{array}{ll}
             0 & \hbox{ if } l = 1,2,  \\
             -\lambda \mathsf{v_1} + \tilde{\mathsf{v}}_2 & \hbox{ if } l = 3.
           \end{array}
         \right.
\end{align*}
We show that $\widehat{\mathcal{T}}_1$ is well-defined. It is straightforward to check that
$$ x_jx_ke(0101)u = 0 = x_kx_je(0101)u, $$
for $1\le j, k\le 4$. Then $(x_k^2 + \lambda x_kx_{k+1} + x_{k+1}^2) \mathsf{u} = 0 $, for $k=1,2,3$, and
\begin{align*}
& \psi_1^2 e(0101)\mathsf{u} = \psi_2^2 e(0101)\mathsf{u} = 0, \quad  \psi_3^2 e(0101)\mathsf{u} = \psi_3 ( -\lambda \mathsf{v_1} + \tilde{\mathsf{v}}_2) = 0
\end{align*}
imply that the relation for $\psi_k^2$ holds.
Next, using $\psi_1\mathsf{w}=0$ and $\psi_2\mathsf{w}=0$, we have
$$ (\psi_k x_l - x_{s_k(l)} \psi_k)e(0101)\mathsf{u} = 0 ,$$
for $k=1,2$. When $k=3$, we have
\begin{align*}
& (\psi_3 x_{4} - x_{3} \psi_3)e(0101)\mathsf{u} =  \psi_3 (-\mathsf{w}) - x_{3} (-\lambda \mathsf{v_1} + \tilde{\mathsf{v}}_2)
= - \tilde{\mathsf{v}}_1 - (-\tilde{\mathsf{v}}_1) = 0, \\
& (\psi_3 x_{3} - x_{4} \psi_3)e(0101)\mathsf{u} =  \psi_3 (-\lambda\mathsf{w}) - x_{4} (-\lambda \mathsf{v_1} + \tilde{\mathsf{v}}_2)
= - \lambda \tilde{\mathsf{v}}_1 - (-\lambda \tilde{\mathsf{v}}_1) = 0, \\
& (\psi_3 x_{2} - x_{2} \psi_3)e(0101)\mathsf{u} =  \psi_3 \mathsf{w} - x_{2} (-\lambda \mathsf{v_1} + \tilde{\mathsf{v}}_2)
= \tilde{\mathsf{v}}_1 - \tilde{\mathsf{v}}_1 = 0.
\end{align*}
Thus the relation for $\psi_k x_{l} - x_{s_k(l)} \psi_k$ holds. Finally, we compute
\begin{align*}
& (x_1 + \lambda x_2 + x_3) e(0101)\mathsf{u} = 0 , \quad  (x_2 + \lambda x_3 + x_4) e(0101)\mathsf{u} = -\lambda^2 \mathsf{w}, \\
& (\psi_{2}\psi_1\psi_{2} - \psi_{1}\psi_{2}\psi_{1}) e(0101) \mathsf{u} = 0, \\
& (\psi_{3}\psi_2\psi_{3} - \psi_{2}\psi_{3}\psi_{2}) e(0101) \mathsf{u} = \psi_{3}\psi_2 (-\lambda \mathsf{v_1} + \tilde{\mathsf{v}}_2 ) = \psi_3 (\lambda \mathsf{v}_2) = -\lambda^2 \mathsf{w}.
\end{align*}
They show that the relation for $\psi_{k+1}\psi_k\psi_{k+1} - \psi_{k}\psi_{k+1}\psi_{k}$ holds, and $\widehat{\mathcal{T}}_1$ is well-defined.

We consider the radical series.
For any $\mathsf{y} = \mathsf{t} + \mathsf{u}$ with $\mathsf{t}\in \mathcal{T}_1$, we have
$$
\psi_3\mathsf{u} \not\in \Span_{\bR}\{ \mathsf{w} , \tilde{\mathsf{v}}_1, \tilde{\mathsf{v}}_2 \} = \Rad(\mathcal{T}_1),
$$
which implies that $\widehat{\mathcal{T}}_1/\Rad(\mathcal{T}_1)$ is not semisimple. Hence $\Rad(\widehat{\mathcal{T}}_1) = \mathcal{T}_1 $
and the assertion follows from Lemma \ref{Lem: T1}.
\end{proof}

\subsection{Representations of $\fqH(2\delta+\alpha_0)$ for $\ell = 1$}
We extend the $\fqH(2\delta)$-module $\widehat{\mathcal{T}}_1$ described in the proof of Lemma \ref{Lem: hat T1} to
$\fqH(2\delta+\alpha_0)$. Note that $\widehat{\mathcal{T}}_1$ is well-defined regardless of the choice of $\lambda$.
We write
$$\widehat{\mathcal{T}}_1 = \Span_\bR\{ \mathsf{v}_1, \mathsf{v}_2, \mathsf{w} , \tilde{\mathsf{v}}_1, \tilde{\mathsf{v}}_2, \mathsf{u}\}$$
as $\eqref{Eq: hat T1}$ in the proof of Lemma \ref{Lem: hat T1}.
Let us declare that $x_5$, $\psi_4$ and $ e(\nu)$, for $\nu \in I^{2\delta + \alpha_0}$, act as follows:
\begin{align*}
x_5 \mathsf{v}_k &= - \tilde{ \mathsf{v}}_k, \qquad x_5 \tilde{\mathsf{v}}_k = 0, \qquad \  \ \ x_5 \mathsf{w} = 0,  \qquad x_5 \mathsf{u} = \lambda \mathsf{w},\\
\psi_4 \mathsf{v}_k &= 0, \qquad \  \ \  \psi_4 \tilde{\mathsf{v}}_k = - \mathsf{v}_k, \qquad \psi_4 \mathsf{w} = 0, \qquad \psi_4 \mathsf{u} = 0,\\
e(\nu) \mathsf{v}_k &= \left\{
                        \begin{array}{ll}
                          \mathsf{v}_k & \hbox{ if } \nu = (01100), \\
                          0 & \hbox{ otherwise,}
                        \end{array}
                      \right.
\quad
e(\nu)  \tilde{\mathsf{v}}_k = \left\{
                        \begin{array}{ll}
                           \tilde{\mathsf{v}}_k & \hbox{ if } \nu = (01100), \\
                          0 & \hbox{ otherwise,}
                        \end{array}
                      \right. \\
e(\nu) \mathsf{w} &= \left\{
                        \begin{array}{ll}
                          \mathsf{w} & \hbox{ if } \nu = (01010), \\
                          0 & \hbox{ otherwise,}
                        \end{array}
                      \right.
\quad
e(\nu)  \mathsf{u} = \left\{
                        \begin{array}{ll}
                           \mathsf{u} & \hbox{ if } \nu = (01010), \\
                          0 & \hbox{ otherwise.}
                        \end{array}
                      \right.
\end{align*}
Thus, $x_1=\psi_1=0$, $e(01100)+e(01010)=1$, where
$$
e(01100)=
\begin{pmatrix}
1 & 0 & 0 & 0 & 0 & 0 \\
0 & 1 & 0 & 0 & 0 & 0 \\
0 & 0 & 0 & 0 & 0 & 0 \\
0 & 0 & 0 & 1 & 0 & 0 \\
0 & 0 & 0 & 0 & 1 & 0 \\
0 & 0 & 0 & 0 & 0 & 0
\end{pmatrix},
$$
and the other generators in matrix form are,
$$
x_2=
\begin{pmatrix}
0 & 1 & 0 & 0 & 0 & 0 \\
0 & 0 & 0 & 0 & 0 & 0 \\
0 & 0 & 0 & 0 & 0 & 1 \\
0 & 0 & 0 & 0 & 1 & 0 \\
0 & 0 & 0 & 0 & 0 & 0 \\
0 & 0 & 0 & 0 & 0 & 0
\end{pmatrix}, \quad
x_3=
\begin{pmatrix}
0 & -1 & 0 & 0 & 0 & 0 \\
0 & 0 & 0 & 0 & 0 & 0 \\
0 & 0 & 0 & 0 & 0 & -\lambda \\
0 & 0 & 0 & 0 & -1 & 0 \\
0 & 0 & 0 & 0 & 0 & 0 \\
0 & 0 & 0 & 0 & 0 & 0
\end{pmatrix}
$$
\begin{gather*}
x_4=
\begin{pmatrix}
0 & 0 & 0 & 0 & 0 & 0 \\
0 & 0 & 0 & 0 & 0 & 0 \\
0 & 0 & 0 & 0 & 0 &-1 \\
1 & 0 & 0 & 0 & 0 & 0 \\
0 & 1 & 0 & 0 & 0 & 0 \\
0 & 0 & 0 & 0 & 0 & 0
\end{pmatrix}, \quad
x_5=
\begin{pmatrix}
0 & 0 & 0 & 0 & 0 & 0 \\
0 & 0 & 0 & 0 & 0 & 0 \\
0 & 0 & 0 & 0 & 0 & \lambda \\
-1& 0 & 0 & 0 & 0 & 0 \\
0 & -1& 0 & 0 & 0 & 0 \\
0 & 0 & 0 & 0 & 0 & 0
\end{pmatrix}, \\[5pt]
\psi_2=
\begin{pmatrix}
0 & 0 & 0 & 0 & 0 & 0 \\
-1& 0 & 0 & 0 & 0 & 0 \\
0 & 0 & 0 & 0 & 0 & 0 \\
0 & 0 & 0 & 0 & 0 & 0 \\
0 & 0 & 0 &-1 & 0 & 0 \\
0 & 0 & 0 & 0 & 0 & 0
\end{pmatrix}, \;
\psi_3=
\begin{pmatrix}
0 & 0 & 0 & 0 & 0 &-\lambda \\
0 & 0 & 0 & 0 & 0 & 0 \\
0 &-\lambda&0 & 0 & 0 & 0 \\
0 & 0 & 1 & 0 & 0 & 0 \\
0 & 0 & 0 & 0 & 0 & 1 \\
0 & 0 & 0 & 0 & 0 & 0
\end{pmatrix}, \;
\psi_4=
\begin{pmatrix}
0 & 0 & 0 &-1 & 0 & 0 \\
0 & 0 & 0 & 0 &-1 & 0 \\
0 & 0 & 0 & 0 & 0 & 0 \\
0 & 0 & 0 & 0 & 0 & 0 \\
0 & 0 & 0 & 0 & 0 & 0 \\
0 & 0 & 0 & 0 & 0 & 0
\end{pmatrix}.
\end{gather*}

We check that the action is well-defined. It is easy to verify the defining relations except for $x_kx_5 - x_5x_k$, $\psi_4^2$, $\psi_k x_l - x_{s_k(l)}\psi_k$
and $\psi_4\psi_3\psi_4 - \psi_3\psi_4\psi_3$.

Direct computation shows that we have, for $1 \le k \le 4$ and $1 \le i \le 2$, $x_k x_5 \tilde{\mathsf{v}}_i = x_5 x_k \tilde{\mathsf{v}}_i = 0$,
$x_k x_5 \mathsf{w} = x_5 x_k \mathsf{w} = 0$, $x_k x_5 \mathsf{u} = x_5 x_k \mathsf{u} = 0$ and
$$ x_k x_5 \mathsf{v}_i = x_5 x_k \mathsf{v}_i = \left\{
                                                   \begin{array}{ll}
                                                     -\tilde{\mathsf{v}}_1 & \hbox{ if } i=2, k=2, \\
                                                     \tilde{\mathsf{v}}_1  & \hbox{ if } i=2, k=3, \\
                                                     0 & \hbox{ otherwise, }
                                                   \end{array}
                                                 \right.
 $$
which means that the relation for $x_kx_5 - x_5x_k$ holds. On the other hand,
$$
\psi_4^2 e(01100) \mathsf{v}_i = 0, \quad \psi_4^2 e(01100) \tilde{\mathsf{v}}_i = \psi_4 ( - \mathsf{v}_i) =0,
$$
for $1\le i\le 2$, and
\begin{align*}
\psi_4^2 e(01010) \mathsf{w} &= 0 = (x_4^2 + \lambda x_4 x_5 + x_5^2) \mathsf{w},  \quad
\psi_4^2 e(01010) \mathsf{u} = 0 = (x_4^2 + \lambda x_4 x_5 + x_5^2) \mathsf{u},
\end{align*}
verify the relation for $\psi_4^2$. Direct computation shows
\begin{align*}
(\psi_4x_5 - x_4 \psi_4) e(01100) \mathsf{v}_i &= \psi_4 (- \tilde{\mathsf{v}}_i) = \mathsf{v}_i, \quad \
(\psi_4x_4 - x_5 \psi_4) e(01100) \mathsf{v}_i = \psi_4 (\tilde{\mathsf{v}}_i) = -\mathsf{v}_i, \\
(\psi_4x_5 - x_4 \psi_4) e(01100) \tilde{\mathsf{v}}_i &= -x_4 (- \mathsf{v}_i) = \tilde{\mathsf{v}}_i, \ \
(\psi_4x_4 - x_5 \psi_4) e(01100) \tilde{\mathsf{v}}_i = -x_5 ( -\mathsf{v}_i) = -\tilde{\mathsf{v}}_i, \\
(\psi_4x_5 - x_4 \psi_4) e(01010) \mathsf{w} &= 0,\qquad \qquad \qquad \ \    (\psi_4x_4 - x_5 \psi_4) e(01010) \mathsf{w} = 0, \\
(\psi_4x_5 - x_4 \psi_4) e(01010) \mathsf{u} &= \psi_4 (\lambda \mathsf{w}) = 0, \qquad   \  (\psi_4x_4 - x_5 \psi_4) e(01010) \mathsf{u} = \psi_4 (- \mathsf{w}) = 0,
\end{align*}
and
\begin{align*}
(\psi_3x_5 - x_5 \psi_3) e(01100) \mathsf{v}_1 &= \psi_3 (- \tilde{\mathsf{v}}_1) = 0, \\
(\psi_3x_5 - x_5 \psi_3) e(01100) \mathsf{v}_2 &= \psi_3 (- \tilde{\mathsf{v}}_2) - x_5(-\lambda \mathsf{w}) = 0, \\
(\psi_3x_5 - x_5 \psi_3) e(01100) \tilde{\mathsf{v}}_i &= 0, \\
(\psi_3x_5 - x_5 \psi_3) e(01010) \mathsf{w} &= - x_5 \tilde{\mathsf{v}}_1 = 0, \\
(\psi_3x_5 - x_5 \psi_3) e(01010) \mathsf{u} &= \psi_3 (\lambda \mathsf{w}) - x_5 (-\lambda \mathsf{v}_1 + \tilde{\mathsf{v}}_2 ) = \lambda \tilde{\mathsf{v}}_1 - \lambda \tilde{\mathsf{v}}_1=0.
\end{align*}
It is straightforward to verify the remaining $\psi_k x_l - x_{s_k(l)}\psi_k$. Lastly, we have
\begin{align*}
(\psi_4 \psi_3 \psi_4 - \psi_3 \psi_4 \psi_3) e(01100) \mathsf{v}_1 &= 0, \\
(\psi_4 \psi_3 \psi_4 - \psi_3 \psi_4 \psi_3) e(01100) \mathsf{v}_2 &= - \psi_3 \psi_4 (-\lambda \mathsf{w}) = 0, \\
(\psi_4 \psi_3 \psi_4 - \psi_3 \psi_4 \psi_3) e(01100) \tilde{\mathsf{v}}_1 &= \psi_4 \psi_3 (- \mathsf{v}_1) = 0, \\
(\psi_4 \psi_3 \psi_4 - \psi_3 \psi_4 \psi_3) e(01100) \tilde{\mathsf{v}}_2 &= \psi_4 \psi_3 (- \mathsf{v}_2) = \psi_4(\lambda \mathsf{w})=0, \\
(\psi_4 \psi_3 \psi_4 - \psi_3 \psi_4 \psi_3) e(01010) \mathsf{w} &= -\psi_3 \psi_4 ( \tilde{\mathsf{v}}_1) = -\psi_3(- \mathsf{v}_1)=0
= (x_3 + \lambda x_4 + x_5) \mathsf{w}, \\
(\psi_4 \psi_3 \psi_4 - \psi_3 \psi_4 \psi_3) e(01010) \mathsf{u} &= -\psi_3 \psi_4 ( - \lambda \mathsf{v}_1 + \tilde{\mathsf{v}}_2) = -\psi_3(- \mathsf{v}_2)=- \lambda \mathsf{w}
= (x_3 + \lambda x_4 + x_5) \mathsf{u}.
\end{align*}
Thus, the relation for $\psi_4\psi_3\psi_4 - \psi_3\psi_4\psi_3$ holds. Therefore, it is well-defined as an $\fqH(2\delta+\alpha_0)$-module. We denote this $\fqH(2\delta+\alpha_0)$-module by $\mathcal{V}$.

Let
$$ \mathcal{O}_0 = \left\{
                     \begin{array}{ll}
                       \Span_\bR\{ \mathsf{v}_1, \mathsf{v}_2, \tilde{\mathsf{v}}_1, \tilde{\mathsf{v}}_2  \} & \hbox{ if } \lambda=0, \\
                       \Span_\bR\{ \mathsf{v}_1, \mathsf{v}_2, \mathsf{w}, \tilde{\mathsf{v}}_1, \tilde{\mathsf{v}}_2  \} & \hbox{ if } \lambda \ne 0.
                     \end{array}
                   \right.
 $$

 It is easy to check that $\mathcal{O}_0$ is a submodule of $\mathcal{V}$. If $\lambda\ne0$ then $\mathcal{O}_0$ viewed as an $\fqH(2\delta)$-module is $\mathcal{T}_1$ and it has the simple socle
$\Span \{\mathsf{v}_1, \mathsf{v}_2 \}$. We know that the same is true for $\lambda=0$ by considering the action of $\psi_4$.
Thus, any $\fqH(2\delta+\alpha_0)$-submodule of $\mathcal{O}_0$ contains $\Span \{\mathsf{v}_1, \mathsf{v}_2\}$,
which shows that $\mathcal{O}_0$ is irreducible. Observe that
$$\mathcal{U} = \Span_{\bR}\{ \mathsf{v}_1, \mathsf{v}_2, \mathsf{w}, \tilde{\mathsf{v}}_1, \tilde{\mathsf{v}}_2 \} $$
is a submodule of $\mathcal{V}$. Let
$$ \mathcal{O}_1 = \mathcal{V} / \mathcal{U}. $$
Note that the module $\mathcal{O}_1$ is a 1-dimensional module on which $x_1,\dots,x_5$ and $\psi_1, \psi_2, \psi_3, \psi_4$ act as 0. From the categorification theorem, we know that
the number of irreducible $\fqH(2\delta + \alpha_0)$-modules is 2. Combining $\eqref{Eq: N0 N1}$ with
$$ \varepsilon_0(\mathcal{O}_0) = 2, \quad \varepsilon_0(\mathcal{O}_1) = 1, $$
we have the following lemma.

\begin{lemma} \label{Lem: irr of R(2delta+alpha0)}
The module $\mathcal{O}_0$ and $\mathcal{O}_1$ form a complete list of irreducible $\fqH(2\delta + \alpha_0)$-modules. Moreover,
$\mathcal{O}_0 = \tilde{f}_0 \mathcal{N}_0$ and $\mathcal{O}_1 = \tilde{f}_0 \mathcal{N}_1$.
\end{lemma}

The $\fqH(2\delta)$-module $ \widehat{\mathcal{N}}_1 = \Span_\bR\{ \mathsf{w}_1, \mathsf{w}_2 \}$ in Section \ref{Sec: repn of R(2delta) ell=1}
can also be extended to an $\fqH(2\delta+\alpha_0)$-module. Indeed,
by declaring that $\psi_4$ acts as 0, and $x_5$, $e(\nu)$ act as
$$ x_5 \mathsf{w}_i = \left\{
                        \begin{array}{ll}
                          0 & \hbox{ if } i=1, \\
                          (2\lambda - \lambda^3) \mathsf{w}_1 & \hbox{ if } i=2,
                        \end{array}
                      \right. \quad
e(\nu) \mathsf{w}_i = \left\{
                        \begin{array}{ll}
                          \mathsf{w}_i & \hbox{ if } \nu = (01010), \\
                          0 & \hbox{ otherwise},
                        \end{array}
                      \right.
  $$
we have a well-defined action. We denote this module by $\widehat{\mathcal{O}}_1$. By construction,
$\widehat{\mathcal{O}}_1$ is uniserial of length 2 whose composition factors are $\mathcal{O}_1$.

On the other hand, when $\lambda = 0$, using $\psi_3 \mathsf{w} = \tilde{\mathsf{v}}_1$, we know that the exact sequence
$$ 0  \longrightarrow \mathcal{O}_0 \longrightarrow \mathcal{U} \longrightarrow \mathcal{O}_1 \longrightarrow 0  $$
is non-split. When $\lambda \ne 0$, we use $ x_5 \mathsf{u} = \lambda \mathsf{w} $ to show that the exact sequence
$$ 0  \longrightarrow \mathcal{O}_0 \longrightarrow \mathcal{V} \longrightarrow \mathcal{O}_1 \longrightarrow 0  $$
does not split. Thus, we have the following lemma.

\begin{lemma} \label{Lem: radical series for 2delta+alpha0}
There exist uniserial $\fqH(2\delta+\alpha_0)$-modules whose radical series are
$$ \begin{array}{c}
     \mathcal{O}_1 \\
     \mathcal{O}_1
   \end{array}
\quad \text{ and }\quad
\begin{array}{c}
     \mathcal{O}_1 \\
     \mathcal{O}_0
   \end{array}.
 $$
\end{lemma}

\vskip 1em

\section{Representation type} \label{Sec: repn type}

\subsection{The algebra $\fqH(\delta)$} In this section, we show that $\fqH(\delta)$ is the Brauer tree algebra
whose Brauer tree is the straight line without exceptional vertex.

Recall the irreducible $\fqH(\delta-\alpha_i)$-modules $\mathcal{L}_i$ defined in Section
\ref{Section: repn of R(delta-alpha)}. We define
\begin{align*}
\mathcal{P}_i = F_i \mathcal{L}_i, \;\;\text{for $1\le i \le \ell$}.
\end{align*}
As $\fqH(\delta-\alpha_i)$ is a simple algebra,
$\mathcal{L}_i$ is a projective $\fqH(\delta-\alpha_i)$-module. It follows that
$\mathcal{P}_i$ are projective $\fqH(\delta)$-modules, since
the functor $F_i$ sends projective objects to projective objects.
Recall the irreducible $\fqH(\delta)$-modules $\mathcal{S}_i$ from Lemma \ref{S for i} and
the biadjointness of the functors $E_i$ and $F_i$ \cite[Thm.3.5]{Kash11}. Then, we have
\begin{equation}
\begin{aligned}
\Hom (\mathcal{S}_j, \mathcal{P}_i) &\simeq \Hom (E_i \mathcal{S}_j, \mathcal{L}_i) \simeq \left\{
                                                                                            \begin{array}{ll}
                                                                                              \bR & \hbox{ if } j=i,  \\
                                                                                              0 & \hbox{ if } j \ne i,
                                                                                            \end{array}
                                                                                          \right. \\
\Hom (\mathcal{P}_i, \mathcal{S}_j) &\simeq \Hom (\mathcal{L}_i, E_i \mathcal{S}_j ) \simeq \left\{
                                                                                            \begin{array}{ll}
                                                                                              \bR & \hbox{ if } j=i,  \\
                                                                                              0 & \hbox{ if } j \ne i.
                                                                                            \end{array}
                                                                                          \right.
\end{aligned}
\end{equation}
Thus, $\mathcal{P}_i$ is the projective cover of $\mathcal{S}_i$, for $1\le i\le \ell$.

\begin{thm} \label{Thm: radical series}
If $\ell=1$ then $\mathcal{P}_1$ is uniserial of length $2$ whose composition factors are $\mathcal{S}_{1}$. If $\ell\ge2$, then
the radical series of $\mathcal{P}_i$ are given as follows:

\begin{align*}
\xy
(2,0)*{\mathcal{P}_1 \simeq}; (11,0)*{\mathcal{S}_{2} }; (11,-5)*{\mathcal{S}_{1} }; (11,5)*{\mathcal{S}_{1} };
(15,-2)*{,};
(25,0)*{\mathcal{P}_i\simeq}; (34,0)*{\mathcal{S}_{i-1} }; (43,0)*{\mathcal{S}_{i+1} }; (37.5,5)*{\mathcal{S}_{i} }; (37.5,-5)*{\mathcal{S}_{i} }; (57,0)*{(i\ne 1, \ell),};
(74,0)*{\mathcal{P}_\ell \simeq}; (85,0)*{\mathcal{S}_{\ell-1} }; (85,-5)*{\mathcal{S}_{\ell} }; (85,5)*{\mathcal{S}_{\ell} };
\endxy
\end{align*}
\end{thm}
\begin{proof}
We compute $\dim \Hom (\mathcal{P}_i, \mathcal{P}_j)$.
Suppose that $i\ne j$. Then $\ell\ge2$ and the isomorphism of functors $E_jF_i\simeq F_iE_j$ implies
$$
\Hom(\mathcal{P}_i, \mathcal{P}_j) \simeq \Hom( E_jF_i \mathcal{L}_i, \mathcal{L}_j)
\simeq \Hom( F_i E_j \mathcal{L}_i, \mathcal{L}_j) \simeq \Hom(  E_j \mathcal{L}_i, E_i \mathcal{L}_j).
$$
Hence, if $i\ne j$ then
$$ \dim \Hom (\mathcal{P}_i, \mathcal{P}_j) = \left\{
                                                \begin{array}{ll}
                                                  1 & \hbox{ if } j = i\pm 1, \\
                                                  0 & \hbox{ otherwise.}
                                                \end{array}
                                              \right.
$$

\medskip
Suppose that $i=j$. Then $E_i \mathcal{L}_i = 0 $ and $ \langle h_i, \Lambda_0 - \delta + \alpha_i \rangle =2$, for $1\le i\le\ell$, imply
$E_iF_i \mathcal{L}_i \simeq \mathcal{L}_i \oplus \mathcal{L}_i$.
Thus, we have
$$ \dim \Hom(\mathcal{P}_i, \mathcal{P}_i) = \dim \Hom( \mathcal{L}_i, E_iF_i \mathcal{L}_i) = 2. $$
Therefore, $[\mathcal{P}_1] = 2[\mathcal{S}_1]$ if $\ell=1$, and if $\ell\ge2$ then
\begin{align*}
[\mathcal{P}_1] &= 2[\mathcal{S}_1] + [\mathcal{S}_2], \\
[\mathcal{P}_i] &= [\mathcal{S}_{i-1}] + 2[\mathcal{S}_i] + [\mathcal{S}_{i+1}]\quad (2\le i\le\ell-1), \\
[\mathcal{P}_{\ell}] &= [\mathcal{S}_{\ell-1}] + 2[\mathcal{S}_{\ell}],
\end{align*}
in the Grothendieck group $K_0(\fqH(\delta)\text{-mod})$.

Recall that the algebra $\fqH(\beta)$ has an anti-involution which fixes all the defining generators elementwise, and we have the
corresponding duality on the category of $\fqH(\beta)$-modules:
$$ M\mapsto M^\vee=\Hom_\bR(M,\bR). $$
It is straightforward to check that $\mathcal{S}_i$ are self-dual, so that
the heart of $\mathcal{P}_i$ is self-dual. Then the self-duality implies that the heart of $\mathcal{P}_i$ is semi-simple.
\end{proof}

\begin{cor} \label{Cor: fqH(delta) repn type}
The algebra $\fqH(\delta)$ is the Brauer tree algebra whose tree is the straight line of length $\ell$ and without
exceptional vertex. In particular, $\fqH(\delta)$ is representation-finite.
\end{cor}

\subsection{ The algebra $\fqH(2\delta)$} In this section, we show that $\fqH(2\delta)$ is of wild type if $\ell\ge2$.
Let $\nu = (\nu_1,\dots,\nu_{\ell+1}) = (0,1,2, \dots, \ell )$ and
$$ e_1 = e( \nu * \nu ), \quad e_2 = e( s_{\ell}( \nu) *  s_{\ell}(\nu) ) .  $$
It follows from Theorem \ref{Thm: dimension formula} that
\begin{align*}
\dim e_i \fqH(2\delta) e_i = 4, \quad \text{for $i=1,2$},
\end{align*}
because Young diagrams which contribute to $\dim e_1\fqH(2\delta)e_1$ are
\begin{align} \label{Eq: YD for e1Re1}
(2\ell+2),\;\;(2\ell+1,1),\;\;(\ell+1,\ell,1),\;\;(\ell,\ell,2),
\end{align}
and Young diagrams which contribute to $\dim e_2\fqH(2\delta)e_2$ are
\begin{align} \label{Eq: YD for e2Re2}
(\ell+1, \ell+1),\;\;(\ell+1, \ell ,1),\;\;(\ell,\ell-1,2,1),\;\;(\ell-1,\ell-1,2,2).
\end{align}

\begin{lemma} \label{Lem: eR(2delta)e}
\begin{enumerate}
\item
There exists an algebra isomorphism
$$ e_1 \fqH(2\delta) e_1 \simeq \bR[x,y]/(x^2, y^2-axy) $$
for some $a \in \bR$. Under this isomorphism, $x_{\ell+1}$, $x_{2\ell+2}$ correspond to $x$, $y$ respectively, and
$(x_{\ell} + x_{\ell+1}) (x_{2\ell+1} + x_{2\ell+2})\ne 0 $ holds in $e_1 \fqH(2\delta) e_1$.
\item
There exists an algebra isomorphism
$$ e_2 \fqH(2\delta) e_2 \simeq \bR[z,w]/(z^2, w^2-bzw) $$
for some $b \in \bR$. Under this isomorphism, $x_{\ell+1}$, $x_{2\ell+2}$ correspond to $z$, $w$ respectively, and
$(x_{\ell} + x_{\ell+1}) (x_{2\ell+1} + x_{2\ell+2})\ne 0 $ holds in $e_2 \fqH(2\delta) e_2$.
\end{enumerate}
\end{lemma}
\begin{proof}
(1) For $t\ge0$ and $0\le k \le \ell$, we set
$$ e^{t,k} = e(  \underbrace{\nu * \cdots * \nu}_{t} * (0)*\cdots *(k-1) ), \qquad  \beta^{t,k} = t \delta +  \alpha_0 + \alpha_1 + \cdots + \alpha_{k-1}.  $$
Note that $ e^{2,0} = e_1$. By Theorem \ref{Thm: dimension formula}, we have
\begin{align} \label{Eq: dimension of eRe}
\dim e^{t,k} \fqH(\beta^{t,k}) e^{t,k} = \left\{
                                           \begin{array}{ll}
                                             1 & \hbox{ if $ t = 0 $ and $ 0\le k\le \ell $}, \\
                                             2 & \hbox{ if $ t = 1 $ and $ 0\le k\le \ell $},  \\
                                             4 & \hbox{ if $ (t,k) = (2,0) $}.
                                           \end{array}
                                         \right.
\end{align}
Note that Young diagram which contributes to $\dim e^{0,k}\fqH( \beta^{0,k} )e^{0,k} $ is $(k+1)$, and
Young diagrams which contribute to $\dim e^{1,k}\fqH( \beta^{1,k} )e^{1,k} $ are $(\ell+k+2)$ and $(\ell,k+2)$, and
if $ (t,k) = (2,0) $ then Young diagrams which contribute to $\dim e^{2,0}\fqH( \beta^{2,0} )e^{2,0} $ are
$(2\ell+2)$, $(2\ell+1,1)$, $(\ell+1,\ell,1)$ and $(\ell,\ell,2)$.

As
$$ \langle h_{\nu_{k+1}}, \Lambda_0 - \beta^{t,k}  \rangle  =
\langle h_k, \Lambda_0 - \beta^{t,k}  \rangle  = \left\{
                                                          \begin{array}{ll}
                                                            1 & \hbox{ if } 0 \le k \le \ell-1, \\
                                                            2 & \hbox{ if } k = \ell,
                                                          \end{array}
                                                        \right.
$$
we have $(e^{t, k}\fqH(\beta^{t, k})e^{t, k}$, $e^{t, k}\fqH(\beta^{t, k})e^{t, k})$-bimodule monomorphisms
\begin{align*}
e^{t, k}\fqH(\beta^{t, k})e^{t, k} &\hookrightarrow e^{t,k+1}\fqH(\beta^{t,k+1})e^{t,k+1}, \\
e^{t, \ell}\fqH(\beta^{t, \ell})e^{t, \ell} \otimes (\bR 1\oplus \bR x) &\hookrightarrow e^{t+1, 0}\fqH(\beta^{t+1, 0})e^{t+1, 0}.
\end{align*}
for $t = 0,1$ and $0 \le k \le \ell-1$. Then $\eqref{Eq: dimension of eRe}$ shows that the above monomorphisms are isomorphisms.
Thus $\bR \simeq e^{0,0} \fqH(\beta^{0,0}) e^{0,0} \simeq e^{0,\ell} \fqH(\beta^{0,\ell}) e^{0,\ell} $. The bimodule isomorphism
$$ e^{0, \ell}\fqH(\beta^{0, \ell})e^{0, \ell} \otimes (\bR 1\oplus \bR x)
 \buildrel \sim \over \longrightarrow e^{1, 0}\fqH(\beta^{1, 0})e^{1, 0} $$
is given by $x \mapsto x_{\ell+1}e^{1,0}$ \cite[Thm.3.4]{Kash11}, and it induces an algebra isomorphism
$$ \bR[x]/(x^2)
 \buildrel \sim \over \longrightarrow e^{1, 0}\fqH(\beta^{1, 0})e^{1, 0} .$$
Similarly, $ \bR[x]/(x^2) \simeq e^{1,0} \fqH(\beta^{1,0}) e^{1,0} \simeq e^{1,\ell} \fqH(\beta^{1,\ell}) e^{1,\ell} $ and
the bimodule isomorphism
$$
\bR[x]/(x^2) \otimes (\bR 1\oplus \bR y) \simeq
e^{1, \ell}\fqH(\beta^{1, \ell})e^{1, \ell} \otimes (\bR 1\oplus \bR y)
 \buildrel \sim \over \longrightarrow e^{2, 0}\fqH(\beta^{2, 0})e^{2, 0} $$
given by $x \mapsto x_{\ell+1} e^{2,0}$, $y \mapsto x_{2\ell+2}e^{2,0}$ induces an algebra isomorphism
$$ \bR[x,y]/(x^2,y^2-a xy )
 \buildrel \sim \over \longrightarrow e_1\fqH(2\delta)e_1,   $$
for some $a\in \bR$.
Taking the grading into consideration, the above argument also implies that
$ x_\ell = 0 $ and $x_{2\ell+1}$ is a scalar multiple of $ x_{\ell+1} $ in $e_1\fqH(2\delta)e_1$.
Thus, $xy\ne0$ maps to $(x_{\ell} + x_{\ell+1}) (x_{2\ell+1} + x_{2\ell+2})$ under the isomorphism.

(2)
We replace $\nu$ with $s_\ell(\nu)$ and follow the argument in (1). Then $e^{t,k}$, $\beta^{t,k}$ are similarly defined,
and $ e^{2,0} = e_2$. Moreover, we have
$\langle h_k, \Lambda_0 - \beta^{t,k}  \rangle = 1$, for $ 0 \le k \le \ell-2 $, and
$$
\langle h_\ell, \Lambda_0 - \beta^{t,\ell-1}  \rangle = 1, \quad
\langle h_{\ell-1}, \Lambda_0 - \beta^{t,\ell}  \rangle = 2.
$$
Thus, Theorem \ref{Thm: dimension formula} gives
\begin{align} \label{Eq: dimension of eRe2}
\dim e^{t,k} \fqH( \beta^{t,k}) e^{t,k} = \left\{
                                           \begin{array}{ll}
                                             1 & \hbox{ if } t = 0 \text{ and $0\le k \le \ell$},\\
                                             2 & \hbox{ if } t = 1 \text{ and $0\le k \le \ell$},\\
                                             4 & \hbox{ if } (t,k) = (2,0),
                                           \end{array}
                                         \right.
\end{align}
and we can conclude that $z \mapsto x_{\ell+1} e^{2,0} ,\ w \mapsto x_{2\ell+2}e^{2,0}$ defines
an algebra isomorphism
$$ \bR[z,w]/(z^2, w^2-b zw )
 \buildrel \sim \over \longrightarrow e_2\fqH(2\delta)e_2,  $$
for some $b\in \bR$. It also follows $(x_{\ell} + x_{\ell+1}) (x_{2\ell+1} + x_{2\ell+2}) \ne 0$ in $e_2\fqH(2\delta)e_2$.
\end{proof}

\begin{prop}\label{Prop: fqH(2delta) is wild}
Suppose $\ell \ge 2$. Then $\fqH(2\delta)$ has wild representation type.
\end{prop}
\begin{proof}
Let  $e = e_1 + e_2$. Combining Theorem \ref{Thm: dimension formula} with $\eqref{Eq: YD for e1Re1}$ and $\eqref{Eq: YD for e2Re2}$, we have
$$  \dim e \fqH(2\delta) e = 10. $$
It follows from Lemma \ref{Lem: eR(2delta)e} that
$$ ( \psi_{\ell} \psi_{2 \ell+1} e_2 ) ( \psi_{\ell} \psi_{2\ell+1} e_1 ) =
e_1 \psi_{\ell}^2 \psi_{2 \ell+1}^2 e_1 = (x_\ell + x_{\ell+1} ) (x_{2\ell+1} + x_{2\ell+2} ) e_1 $$
is nonzero, which implies that
$\psi_{\ell} \psi_{2 \ell+1} e_2$ and $\psi_{\ell} \psi_{2\ell+1} e_1$ are nonzero in $e \fqH(2\delta) e$.
Hence,
$$ \{ e_1, e_2, x_{\ell+1}e_1, x_{\ell+1}e_2, x_{2\ell+2}e_1, x_{2\ell+2}e_2, x_{\ell+1}x_{2\ell+2}e_1, x_{\ell+1}x_{2\ell+2}e_2, \psi_{\ell} \psi_{2 \ell+1} e_1, \psi_{\ell} \psi_{2 \ell+1} e_2 \} $$
forms a basis of $e \fqH(2\delta) e$. As $\dim e_1 \fqH(2\delta) e_2 = 1$ and $\dim e_2 \fqH(2\delta) e_1 = 1$,
the degree consideration shows that $\psi_{\ell} \psi_{2\ell+1} e_1$ and $\psi_{\ell} \psi_{2 \ell+1} e_2$ are annihilated by
both $x_{\ell+1}$ and $x_{2\ell+2}$.

Let $ p = \psi_{\ell} \psi_{2\ell+1} e_1 $ and $ q = \psi_{\ell} \psi_{2\ell+1} e_2 $. Using the isomorphisms in Lemma \ref{Lem: eR(2delta)e}, we have
the following quiver presentation of $e \fqH(2\delta) e$:
$$
\xy
(0,0) *{e_1}="A", (20,0) *{e_2}="B"
\ar @/_/_{p} "A";"B"
\ar @/_/_{q} "B";"A"
\ar @(l,u)^x "A";"A"
\ar @(l,d)_y "A";"A"
\ar @(r,u)_z "B";"B"
\ar @(r,d)^w "B";"B"
%\ar @(rdd,ruu) "B";"B"
\endxy
$$
with relations
\begin{align*}
& x^2=0, \quad  y^2 = a xy, \quad xy = yx, \quad z^2 = 0, \quad w^2 = b zw, \quad zw = wz, \\
& pq = xy, \quad qp = zw, \quad xp = yp = pz = pw = 0, \quad  zq = wq = qx = qy = 0,
\end{align*}
where $a, b \in \bR$ are given in Lemma \ref{Lem: eR(2delta)e}. Then, by \cite[I.10.8]{Erd90},
$e \fqH(2\delta) e$ has wild representation type, and so does $\fqH(2\delta)$.
\end{proof}

\subsection{The algebras for $\ell=1$}
Using the $\fqH(2\delta-\alpha_0)$-module $\mathcal{M}_0$ and the
$\fqH(2\delta-\alpha_1)$-module $\widehat{\mathcal{M}}_1$ given in Section \ref{Sec: repn of R(2delta-alpha_i) l=1},
we define
$$ \mathcal{Q}_0 = F_0 \mathcal{M}_0 , \qquad \mathcal{Q}_1 = F_1 \widehat{\mathcal{M}}_1.$$
As the modules $\mathcal{M}_0$ and $\widehat{\mathcal{M}}_1$ are projective, the $\fqH(2\delta)$-modules $\mathcal{Q}_0$ and $\mathcal{Q}_1$ are projective.
Recall the irreducible $\fqH(2\delta)$-modules $\mathcal{N}_0$ and $\mathcal{N}_1$ given in Lemma \ref{Lem: irr of R(2delta)}.
By $\eqref{Eq: N0 N1}$ and the biadjointness \cite[Thm.3.5]{Kash11} of the functors $E_i$ and $F_i$, we have
\begin{align*}
\Hom (\mathcal{N}_i, \mathcal{Q}_0) &\simeq \Hom (E_0 \mathcal{N}_i, \mathcal{M}_0) \simeq \left\{
                                                                                            \begin{array}{ll}
                                                                                              \bR & \hbox{ if } i =0, \\
                                                                                              0 & \hbox{ if } i = 1,
                                                                                            \end{array}
                                                                                          \right. \\
\Hom ( \mathcal{Q}_0, \mathcal{N}_i) &\simeq \Hom (\mathcal{M}_0, E_0 \mathcal{N}_i) \simeq \left\{
                                                                                            \begin{array}{ll}
                                                                                              \bR & \hbox{ if } i =0, \\
                                                                                              0 & \hbox{ if } i = 1,
                                                                                            \end{array}
                                                                                          \right. \\
\Hom (\mathcal{N}_i, \mathcal{Q}_1) &\simeq \Hom (E_1 \mathcal{N}_i, \widehat{\mathcal{M}}_1) \simeq \left\{
                                                                                            \begin{array}{ll}
                                                                                              0 & \hbox{ if } i =0, \\
                                                                                              \bR & \hbox{ if } i =1,
                                                                                            \end{array}
                                                                                          \right. \\
\Hom (\mathcal{Q}_1, \mathcal{N}_i) &\simeq \Hom (\widehat{\mathcal{M}}_1, E_1 \mathcal{N}_i) \simeq \left\{
                                                                                            \begin{array}{ll}
                                                                                              0 & \hbox{ if } i =0, \\
                                                                                              \bR & \hbox{ if } i =1.
                                                                                            \end{array}
                                                                                          \right.
\end{align*}
Hence, the modules $\mathcal{Q}_0$ and $\mathcal{Q}_1$ are projective cover of $\mathcal{N}_0$ and $\mathcal{N}_1$ respectively.
Moreover, since $E_iF_j \simeq  F_jE_i$ for $i\ne j$ and
\begin{align*}
E_0F_0\mathcal{M}_0 &\simeq F_0E_0\mathcal{M}_0 \oplus \mathcal{M}_0^{\oplus\langle h_0, \Lambda_0-2\delta+\alpha_0 \rangle}, \\
E_1F_1\widehat{\mathcal{M}}_1 &\simeq F_1E_1 \widehat{\mathcal{M}}_1 \oplus \widehat{\mathcal{M}}_1^{\oplus\langle h_1, \Lambda_0-2\delta+\alpha_1 \rangle},
\end{align*}
it follows from $\eqref{Eq: M0 M1}$ that
\begin{align*}
\dim \Hom(\mathcal{Q}_0, \mathcal{Q}_0) &= \dim \Hom(\mathcal{M}_0, E_0F_0\mathcal{M}_0) = \dim \Hom(\mathcal{M}_0, \mathcal{M}_0^{\oplus 3} ) = 3, \\
\dim \Hom(\mathcal{Q}_1, \mathcal{Q}_1) &= \dim \Hom(\widehat{\mathcal{M}}_1, E_1F_1 \widehat{\mathcal{M}}_1) =
\dim \Hom(\widehat{\mathcal{M}}_1, \widehat{\mathcal{M}}_1^{\oplus 2} ) = 4, \\
\dim \Hom(\mathcal{Q}_0, \mathcal{Q}_1) &= \dim \Hom(E_1\mathcal{M}_0, E_0\widehat{\mathcal{M}}_1)  = 2, \\
\dim \Hom(\mathcal{Q}_1, \mathcal{Q}_0) &= \dim \Hom(E_0\widehat{\mathcal{M}}_1, E_1\mathcal{M}_0)  = 2.
\end{align*}
Thus, we have
\begin{align} \label{Eq: composition factors of Q}
[\mathcal{Q}_0] = 3[\mathcal{N}_0] + 2[\mathcal{N}_1], \qquad [\mathcal{Q}_1] = 2[\mathcal{N}_0] + 4[\mathcal{N}_1].
\end{align}
in the Grothendieck group $K_0(\fqH(2\delta)\text{-mod})$.

\begin{prop} \label{Prop: radical series of R(2delta) l=1}
Suppose $\ell=1$.
\begin{enumerate}
\item If $\lambda = 0$, then the radical series of $\mathcal{Q}_0$ and $\mathcal{Q}_1$ are given as follows:
$$
\xy
(0,0)*{\mathcal{Q}_0 \simeq}; (9,0)*{\mathcal{N}_{0} }; (16,2.5)*{\mathcal{N}_{1} }; (16,-2.5)*{\mathcal{N}_{1} }; (12.5,7)*{\mathcal{N}_{0} }; (12.5,-7)*{\mathcal{N}_{0} };
(20,-2)*{,};
(40,0)*{\mathcal{Q}_1 \simeq}; (49,2.5)*{\mathcal{N}_{1} }; (49,-2.5)*{\mathcal{N}_{0} }; (56,2.5)*{\mathcal{N}_{0} }; (56,-2.5)*{\mathcal{N}_{1} };
(52.5,7)*{\mathcal{N}_{1} }; (52.5,-7)*{\mathcal{N}_{1} };
(60,-2)*{.};
\endxy
$$
\item If $\lambda \ne 0$, then the radical series of $\mathcal{Q}_0$ and $\mathcal{Q}_1$ are given as follows:
$$
\xy
(0,0)*{\mathcal{Q}_0 \simeq}; (10,0)*{\mathcal{N}_{0} }; (10,5)*{\mathcal{N}_{1} }; (10,10)*{\mathcal{N}_{0} }; (10,-5)*{\mathcal{N}_{1} }; (10,-10)*{\mathcal{N}_{0} };
(15,-2)*{,};
(40,0)*{\mathcal{Q}_1 \simeq}; (49,0)*{\mathcal{N}_{1} }; (56,0)*{\mathcal{N}_{1} }; (56,5)*{\mathcal{N}_{0} }; (56,-5)*{\mathcal{N}_{0} };
(52.5,10)*{\mathcal{N}_{1} }; (52.5,-10)*{\mathcal{N}_{1} };
(60,-2)*{.};
\endxy
$$
\end{enumerate}
\end{prop}
\begin{proof}

Recall the anti-involution of $\fqH(\beta)$ fixing all defining generators elementwise, which yields the duality
$ M \mapsto M^\vee = \Hom_\bR(M, \bR)$ on the category of $\fqH(\beta)$-modules. As $\mathcal{Q}_i$ are indecomposable projective-injective modules,
$\mathcal{Q}_i^\vee$ is isomorphic to $\mathcal{Q}_0$ or $\mathcal{Q}_1$. Then, (\ref{Eq: composition factors of Q}) implies that
$\mathcal{Q}_i$ are self-dual. Thus, $\Soc(\mathcal{Q}_i)\simeq \Top(\mathcal{Q}_i)\simeq\mathcal{N}_i$ and
the heart of $\mathcal{Q}_i$ is self-dual.

Suppose $\lambda = 0$. Then $\Rad(\mathcal{Q}_0)/\Soc(\mathcal{Q}_0)$ has the quotient module
$$
\mathcal{N}'=
\begin{array}{c}
\mathcal{N}_1 \\
\mathcal{N}_1
\end{array}
$$
by Lemma \ref{Lem: T0}. 
Since $[\Rad(\mathcal{Q}_0)/\Soc(\mathcal{Q}_0)]=[\mathcal{N}_0]+2[\mathcal{N}_1]$,  
$\mathcal{N}_0$ appears in $\Soc(\Rad(\mathcal{Q}_0)/\Soc(\mathcal{Q}_0))$. Taking its dual, we know that 
$\mathcal{N}_0$ appears in $\Top(\Rad(\mathcal{Q}_0)/\Soc(\mathcal{Q}_0))$. Hence,
$$
0 \to \mathcal{N}_0 \to \Rad(\mathcal{Q}_0)/\Soc(\mathcal{Q}_0) \to \mathcal{N}' \to 0
$$ 
splits,
and we conclude that $\mathcal{Q}_0$ is as in the assertion.
Similarly, $\Rad(\mathcal{Q}_1)/\Soc(\mathcal{Q}_1)$ has the quotient module
$$
\mathcal{N}''=
\begin{array}{c}
\mathcal{N}_1 \\
\mathcal{N}_0
\end{array}
$$
and its dual as its submodule. Then $\Ext^1(\mathcal{N}_1, \mathcal{N}_i)\ne 0$, for $i=0,1$, implies that
$$ \Top(\Rad(\mathcal{Q}_1)/\Soc(\mathcal{Q}_1))\simeq \mathcal{N}_0 \oplus \mathcal{N}_1. $$
Thus, it suffices to show that $\mathcal{N}''$ appears as a submodule. The definition of $ \mathcal{Q}_1$ implies 
$$
0 \to F_1\mathcal{M}_1 \to \mathcal{Q}_1 \to F_1\mathcal{M}_1 \to 0
$$
and $[F_1\mathcal{M}_1]=[\mathcal{N}_0]+2[\mathcal{N}_1]$. In particular, we have $\Soc(F_1\mathcal{M}_1)\simeq\mathcal{N}_1$, 
$\Top(F_1\mathcal{M}_1)\simeq \mathcal{N}_1$, and $F_1\mathcal{M}_1$ is uniserial. 
Hence, $\Rad(\mathcal{Q}_1)/\Soc(\mathcal{Q}_1)$ has a submodule which is isomorphic to $\mathcal{N}''$. We conclude that
$\mathcal{Q}_1$ is as in the assertion.

Suppose that $\lambda \ne 0$. Lemma \ref{Lem: hat T1} implies that we have uniserial modules
$$
\widehat{\mathcal{T}}_1 = \begin{array}{c}
                     \mathcal{N}_1 \\
                     \mathcal{N}_0 \\
                     \mathcal{N}_1 \\
                     \mathcal{N}_0
                   \end{array},
\qquad
\widehat{\mathcal{T}}_1^\vee = \begin{array}{c}
                     \mathcal{N}_0 \\
                     \mathcal{N}_1 \\
                     \mathcal{N}_0 \\
                     \mathcal{N}_1
                   \end{array} ,
 $$
and epimorphisms $ \mathcal{Q}_0 \twoheadrightarrow \widehat{\mathcal{T}}_1^\vee$ and $\mathcal{Q}_1 \twoheadrightarrow \widehat{\mathcal{T}}_1$.
Therefore, $\mathcal{Q}_0$ is as in the assertion and the assertion for $\mathcal{Q}_1$ follows from $\eqref{Eq: composition factors of Q}$
and the self-duality of $\Rad(\mathcal{Q}_1)/\Soc(\mathcal{Q}_1)$.
\end{proof}

\begin{prop}\label{ell=1 case}
If $\ell=1$, then $\fqH(2\delta)$ is a symmetric algebra of tame representation type.
\end{prop}
\begin{proof}

Suppose that $\lambda = 0$. By Proposition \ref{Prop: radical series of R(2delta) l=1}(1), the basic algebra of $\fqH(2\delta)$ is
$$
\xy
(0,0) *{\begin{array}{c} 0 \end{array}}="A", (20,0) *{ \begin{array}{c} 1 \end{array} }="B"
\ar @/_/_{\alpha} "A";"B"
\ar @/_/_{\beta} "B";"A"
\ar @(lu,ld)_\gamma "A";"A"
\ar @(rd,ru)_\delta "B";"B"
\endxy
$$
with relations
$$ \alpha\beta = \beta\gamma = \gamma\alpha = \delta^2 = 0, \quad \gamma^2=\alpha\delta\beta,  \quad \beta\alpha\delta = c\delta\beta\alpha, \;\;
\text{for some $c\in \bR^\times$.} $$
Here, we choose $\delta$ to be $\mathcal{Q}_1\twoheadrightarrow\mathcal{Q}_1/F_1\mathcal{M}_1\simeq F_1\mathcal{M}_1\hookrightarrow \mathcal{Q}_1$.

As $e(0101)\mathcal{N}_1\ne0$ and $\fqH(2\delta)e(0101)$ is a projective $\fqH(2\delta)$-module, we have a surjective $\fqH(2\delta)$-module homomorphism
$\fqH(2\delta)e(0101)\to \mathcal{Q}_1$. Then,
$$ \dim \fqH(2\delta)e(0101)=8=\dim \mathcal{Q}_1 $$
implies that $\fqH(2\delta)e(0101)\simeq\mathcal{Q}_1$ and we have $\End(\mathcal{Q}_1)\simeq e(0101)\fqH(2\delta)e(0101)$.
Now we observe that Lemma \ref{Lem: eR(2delta)e}(1) is valid for $\ell=1$. In particular, $\End(\mathcal{Q}_1)$ is commutative.
As the paths $\beta\alpha$ and $\delta$ can be regarded as elements in $\End(\mathcal{Q}_1)$, we have $c=1$, i.e.,
$$ \beta\alpha\delta = \delta\beta\alpha. $$
Thus, $\fqH(2\delta)$ is a symmetric algebra.
The number of outgoing arrows and the number of incoming arrows are $2$ at each vertex, and
$\alpha\beta = \beta\gamma = \gamma\alpha = \delta^2 = 0$ implies that $\fqH(2\delta)$ is a special biserial algebra.
Further, we may define a surjective algebra homomorphism from $\End(\mathcal{Q}_1)$ to $\bR[x,y]/(x^2,xy,y^2)$ by
$\delta\mapsto x, \beta\alpha\mapsto y$. Thus, $\fqH(2\delta)$ is tame if $\lambda=0$.

Suppose that $\lambda \ne 0$. By Proposition \ref{Prop: radical series of R(2delta) l=1}(2),
the basic algebra of $\fqH(2\delta)$ is
$$
\xy
(0,0) *{\begin{array}{c} 0 \end{array}}="A", (20,0) *{ \begin{array}{c} 1 \end{array} }="B"
\ar @/_/_{\alpha} "A";"B"
\ar @/_/_{\beta} "B";"A"
\ar @(rd,ru)_\gamma "B";"B"
\endxy
$$
with relations $ \alpha\gamma = \gamma\beta = 0 $, $(\beta\alpha)^2 = \gamma^2$.
It is a special biserial algebra. Indeed, it is the algebra given in \cite[Prop.(A)]{EN02}.
Thus, it is a symmetric algebra of tame type.
\end{proof}

The reason we need extra task for $\ell=1$ is that the two algebras
$e(0101)\fqH(2\delta)e(0101)$ and $e(01010)\fqH(2\delta + \alpha_0)e(01010)$ are different, which we know from
$$ \dim e(0101)\fqH(2\delta)e(0101) = 4, \quad \dim e(01010)\fqH(2\delta + \alpha_0)e(01010) = 8. $$
Thus, the previous argument used for showing wildness is not valid when we
show the wildness of $\fqH(3\delta)$ for $\ell=1$. Hence, we argue as in the next proposition.

\begin{prop}\label{wild for ell=1}
If $\ell=1$, then $\fqH(3\delta)$ has wild representation type.
\end{prop}
\begin{proof}
It follows from Theorem \ref{Thm: dimension formula} that $\dim e(\nu,1) \fqH(2\delta + \alpha_0) = 0 $ for any $\nu  \in I^{2\delta}$.
Here, $e(\nu,1)$ is the idempotent corresponding to the concatenation of $\nu$ and $(1)$. Thus, if we define
$$ e = \sum_{\nu \in I^{2\delta} } e(\nu, 1) $$
then we have $ E_1 \fqH(2\delta + \alpha_0) = e\fqH(2\delta + \alpha_0) = 0$. Since $\langle h_1, \Lambda_0 - 2\delta - \alpha_0 \rangle =2$,
we have an bimodule isomorphism
\begin{align*}
E_1F_1 \fqH( 2\delta+\alpha_0) &\simeq F_1E_1 \fqH(2\delta + \alpha_0) \oplus \fqH(2\delta + \alpha_0) \otimes_{\bR} \bR[t]/(t^2) \\
&\simeq \fqH(2\delta + \alpha_0) \otimes_{\bR} \bR[t]/(t^2),
\end{align*}
by \cite[Theorem 5.2]{KK11}, and it induces an algebra isomorphism
\begin{align*}
e \fqH(3\delta) e  &/ \Rad^2( e \fqH(3\delta) e )  \\
& \simeq  \fqH(2\delta + \alpha_0) \otimes_{\bR} \bR[t]
/ (t^2, t \Rad( \fqH(2\delta + \alpha_0) ), \Rad^2( \fqH(2\delta + \alpha_0) ) ).
\end{align*}
Thus, if we denote the irreducible $\bR[t]/(t^2)$-module by $\mathcal{S}$, then $e \fqH(3\delta) e / \Rad^2( e \fqH(3\delta) e )$
has the irreducible modules $\mathcal{O}_0 \otimes \mathcal{S}$ and $\mathcal{O}_1 \otimes \mathcal{S}$,
where $\mathcal{O}_0$ and $\mathcal{O}_1$ are irreducible $\fqH(2\delta + \alpha_0)$-modules in Lemma \ref{Lem: irr of R(2delta+alpha0)}.
By Lemma \ref{Lem: radical series for 2delta+alpha0}, the projective cover of $\mathcal{O}_1 \otimes \mathcal{S}$ has
the following radical series:
$$ \begin{array}{ccc}
      & \mathcal{O}_1 \otimes \mathcal{S} &  \\
     \mathcal{O}_1 \otimes \mathcal{S} & \mathcal{O}_1 \otimes \mathcal{S} & \mathcal{O}_0 \otimes \mathcal{S}\ .
   \end{array}
$$
It implies that the quiver of $e \fqH(3\delta) e$ has
$$
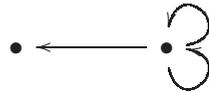
\xy
(0,0) *{\begin{array}{c} \bullet \end{array}}="A", (20,0) *{ \begin{array}{c} \bullet \end{array} }="B"
\ar  "B";"A"
\ar @(r,u) "B";"B"
\ar @(d,r) "B";"B"
\endxy
$$
as a proper subquiver. It follows from \cite[I.10.8]{Erd90} that $e \fqH(3\delta) e$ is wild, and so is $\fqH(3\delta)$.
\end{proof}

\subsection{ Representation type of $\fqH(\beta)$}
In this section, we prove the Erdmann-Nakano type theorem for $\fqH(\beta)$ with arbitrary parameter value $\lambda\in\bR$.

Let $A$ and $B$ be finite dimensional $\bR$-algebras. If there exists a constant $C>0$ and functors
$$
F:\;A\text{\rm -mod} \rightarrow B\text{\rm -mod}, \quad
G:\;B\text{\rm -mod} \rightarrow A\text{\rm -mod}
$$
such that, for any $A$-module $M$,
\begin{itemize}
\item[(1)]
$M$ is a direct summand of $GF(M)$ as an $A$-module,
\item[(2)]
$\dim F(M)\le C\dim M$,
\end{itemize}
then wildness of $A$ implies wildness of $B$ {\cite[Prop.2.3]{EN02}}. As a corollary, we have the following lemma.

\begin{lemma}\label{wild case}
If $\fqH( k\delta + \alpha_0 \cdots + \alpha_{i-1} )$ is wild, so is $\fqH( k\delta + \alpha_0 \cdots + \alpha_{i} )$.
\end{lemma}
\begin{proof}
For $k\in \Z_{\ge0}$ and $0\le i\le\ell$, we have
$$
\langle h_i, \Lambda_0 - k\delta - \alpha_0 \cdots -\alpha_{i-1}  \rangle = \left\{
                                                                                \begin{array}{ll}
                                                                                  1 & \hbox{ if } 0 \le i \le \ell-1, \\
                                                                                  2 & \hbox{ if } i = \ell.
                                                                                \end{array}
                                                                              \right.
$$
Thus, the functor $F_i : \fqH( k\delta + \alpha_0 \cdots + \alpha_{i-1} )\text{-mod} \rightarrow
\fqH( k\delta + \alpha_0 \cdots + \alpha_{i} )\text{-mod}$ satisfies the assumptions (1) and (2) above.
\end{proof}

Recall that a weight $\mu$ with $V(\Lambda_0)_\mu \ne 0$ can be written as
$$\mu = \kappa - k \delta$$
 for some $\kappa \in \weyl \Lambda_0$ and $k\in \Z_{\ge0}$ and a weight $\mu$ of the above form always satisfies $V(\Lambda_0)_{\mu} \ne 0$.
Note that the pair $(\kappa, k)$ is determined uniquely by $\mu$.
Then, the following Erdmann-Nakano type theorem follows from Corollary \ref{Cor: fqH(delta) repn type},
Propositions \ref{Prop: fqH(2delta) is wild}, \ref{ell=1 case} and \ref{wild for ell=1}.

\begin{thm} \label{main theorem}
For $\kappa \in \weyl \Lambda_0$ and $k\in \Z_{\ge0}$, the finite quiver Hecke algebra $\fqH(\Lambda_0-\kappa+k\delta)$ of type $A_{\ell}^{(1)}$ ($\ell \ge 1$) is
\begin{enumerate}
\item simple if $k=0$,
\item of finite representation type but not semisimple if $k=1$,
\item of tame representation type if $\ell=1$ and $k=2$,
\item of wild representation type otherwise.
\end{enumerate}
\end{thm}

When $\ell=1$ and $k=2$, we may study $\fqH(\Lambda_0-\kappa+\delta)$ in more detail. When $\lambda\ne0$, 
they are all biserial algebras. If they are special biserial algebras, we may classify them and 
they already appeared as tame block algebras of the Hecke algebras associated with the symmetric group. 
This is the topic of the next section.

\section{Symmetric quiver Hecke algebras of tame type}

We first classify two-point symmetric special biserial algebras up to Morita equivalence. 
For those who are not familiar with special biserial algebras, see \cite[II]{Erd90}.

\begin{thm} \label{first classification}
Suppose that a symmetric $\bR$-algebra $A=\bR Q/I$, where $Q$ is a connected quiver and $I$ is an admissible ideal,
has the following properties.
\begin{itemize}
\item[(a)]
$Q$ has two vertices, which we denote $Q_0=\{0,1\}$.
\item[(b)]
$A$ is a special biserial algebra.
\end{itemize}
Then, $A$ is one of the following algebras.
\begin{itemize}
\item[(1)]
$Q=$\begin{xy} (0,0) *{\bullet}="A", (15,0) *{\bullet}="B", \ar @/^4mm/ "A";"B"^\alpha \ar @/^4mm/ "B";"A"^\beta\end{xy}\quad
$(\alpha\beta)^m\alpha=(\beta\alpha)^m\beta=0$.
\item[(2)]
$Q=$\begin{xy} (0,0) *{\bullet}="A", (15,0) *{\bullet}="B", \ar @(lu,ld) "A";"A"_{\gamma} \ar @/^4mm/ "A";"B"^\alpha \ar @/^4mm/ "B";"A"^\beta \end{xy}
\quad such that the relations are either (2a) or (2b).
\begin{align*}
(2a)&\quad \beta\gamma=\gamma\alpha=0,\;\; \gamma^p=(\alpha\beta)^q. \\
(2b)&\quad \beta\alpha=\gamma^2=0,\;\; (\gamma\alpha\beta)^m=(\alpha\beta\gamma)^m.
\end{align*}
\item[(3)]
$Q=$\begin{xy} (0,0) *{\bullet}="A", (15,0) *{\bullet}="B", \ar @/^4mm/ "A";"B"^{\alpha,\alpha'} \ar @/^4mm/ "B";"A"^{\beta,\beta'} \end{xy}
\quad such that the relations are either (3a) or (3b).
\begin{align*}
(3a)&\quad \alpha\beta'=\beta'\alpha=\alpha'\beta=\beta\alpha'=0,\;\;(\alpha\beta)^p=(\alpha'\beta')^q,\;\;(\beta\alpha)^p=(\beta'\alpha')^q.\\
(3b)&\quad \alpha\beta'=\beta\alpha=\alpha'\beta=\beta'\alpha'=0,\;\;(\alpha\beta\alpha'\beta')^m=(\alpha'\beta'\alpha\beta)^m,\;\;
(\beta\alpha'\beta'\alpha)^m=(\beta'\alpha\beta\alpha')^m.
\end{align*}
\item[(4)]
$Q=$\begin{xy} (0,0) *{\bullet}="A", (15,0) *{\bullet}="B", \ar @(lu,ld) "A";"A"_{\gamma} \ar @/^4mm/ "A";"B"^\alpha \ar @/^4mm/ "B";"A"^\beta
\ar @(ru,rd) "B";"B"^\delta \end{xy}
\quad such that the relations are either (4a) or (4b) or (4c).
\begin{align*}
(4a)&\quad \beta\gamma=\gamma\alpha=\alpha\delta=\delta\beta=0,\;\; (\alpha\beta)^p=\gamma^q,\;\; (\beta\alpha)^p=\delta^r.\\
(4b)&\quad \beta\alpha=\gamma^2=\alpha\delta=\delta\beta=0,\;\; (\gamma\alpha\beta)^p=(\alpha\beta\gamma)^p,\;\;
(\beta\gamma\alpha)^p=\delta^q.\\
(4c)&\quad \alpha\beta=\beta\alpha=\gamma^2=\delta^2=0,\;\;(\beta\gamma\alpha\delta)^m=(\delta\beta\gamma\alpha)^m,\;\;
(\gamma\alpha\delta\beta)^m=(\alpha\delta\beta\gamma)^m.
\end{align*}
\end{itemize}
\end{thm}

\begin{Rmk}
There are algebras from (3a) that appear in a different context: the principal blocks of restricted Lie algebras \cite{FS02}.
\end{Rmk}

When we asked her comment on our paper, 
Professor Erdmann informed us Donovan's work \cite{D79}. In page 189 of the paper, he claimed classification of Brauer graph algebras 
and its twisted forms with one or two irreducible modules,  
without the definition of twisted Brauer graph algebras and without the proof of his classification. Then his list is given in the same page. 
After excluding algebras which are not special biserial, there remain $A_{200}$ with $u=2$, $A_{210}, A_{220}, A_{230}, A_{310}, A_{320}$. 
The first algebra is (2a) with $(p,q)=(2,1)$ from our list.  $A_{210}, A_{220}, A_{230}$ are (4a), (4c) and (4b) respectively. 
(But the last exponent $\nu$ in $A_{230}$ should be $\mu$.) If we allow $\nu=1$ in $A_{210}$ and $A_{230}$, and eliminate a generator 
from the relations then we obtain (2a) and (2b). Note that the algebras from (2b) and (4b) are algebras of dihedral type: see 
\cite[Thm.VI.8.1(i)]{Erd90} for (2b) and \cite[Thm.VI.8.2(i)]{Erd90} for (4b). 
$A_{310}$ is (3a) with $p=q$, and $A_{320}$ is (3b). (Thus, last two $\mu$'s in $A_{320}$ seem to be $\nu$.) 
In any case, it is worth mentioning that our algebras are Brauer graph algebras.
See \cite{Kauer98} or \cite{Ro98} for general statement in this direction. 

A \emph{Brauer graph} is a finite graph which may have loops and multiple edges such that
\begin{itemize}
\item[(i)]
for each vertex, a natural number, the multiplicity of the vertex, is assigned,
\item[(ii)]
for each vertex, a cyclic ordering of the edges connected to the vertex is specified. 
\end{itemize}
In (ii), a loop is considered as two edges whose other ends are closed to the loop. Thus, a loop appears twice in the cyclic ordering. 
To define a Brauer graph algebra, 
we consider a quiver whose vertices and arrows are given by the edges of the Brauer graph and
the arrows of the cyclic orderings, respectively. Note that if the number of arrows connected to a vertex is one then 
the cyclic ordering defines a loop of the quiver. The relations are defined as follows.
\begin{itemize}
\item[(1)]
If $\beta$ does not appear immediately after $\alpha$ in any of the cyclic orderings, then $\alpha\beta=0$.
\item[(2)]
For each edge of the Brauer graph, let $\alpha_1\cdots\alpha_p$ and $\beta_1\cdots\beta_q$ be 
two cyclic orderings starting at the edge and let $m$, $n$ be the multiplicities of the two ends of the edge. Then 
$(\alpha_1\cdots\alpha_p)^m=(\beta_1\cdots\beta_q)^n$.
\end{itemize}
In particular, if the number of edges connected to a vertex is $1$ and the vertex has multiplicity $1$ then we may eliminate  
the loop which the cyclic ordering defines from the defining relations. 
\begin{itemize}
\item[(a)]
Let us consider a vertex with multiplicity $m$ and draw $4$ open end edges which connects to the vertex in a clockwise manner, 
and declare that this is the cyclic ordering. 
If we close two adjacent edges to a loop, and the remaining two edges to a loop, we obtain a Brauer graph with one vertex of multiplicity $m$ and $2$ loops. 
Its Brauer graph algebra is (4c). If we close two pairs of diagonal open end edges, 
we obtain another Brauer graph with one vertex of multiplicity $m$ and $2$ loops. Its Brauer graph algebra is (3b). 
\item[(b)]
If we consider two vertices of multiplicity $p$ and $q$ and a loop on the vertex of multiplicity $p$, we obtain (4b) if $q\ge2$ and (2b) if $q=1$. 
\item[(c)]
If we consider two vertices of multiplicity $p$ and $q$ and connect the two vertices with two edges, we obtain (3a). 
\item[(d)]
If we consider three vertices of multiplicity $p$, $q$, $r$ and connect the vertices of multiplicity $p$ and $q$, 
the vertices of multiplicity $p$ and $r$, then we obtain (4a) if $r\ge2$ and (2a) if $r=1$. 
\end{itemize}

The proof of Theorem \ref{first classification} will be given in subsequent sections. Note that the number of outgoing arrows
and incoming arrows at each vertex is at most $4$, so that the number in total is $8$ because the quiver has two vertices.
But each arrow is counted twice and we have that the number of arrows is at most $4$. The following lemma is useful.

\begin{lemma}\label{useful lemma}
Suppose that $A=\bR Q/I$, where $Q$ is a connected quiver and $I$ is an admissible ideal, is a symmetric special biserial algebra.
If $u=u_1\cdots u_r\ne0$, where $u_1,\dots,u_r$ are arrows, satisfies $u\delta=0$ in $A$, for any arrow $\delta$, then the following hold.
\begin{itemize}
\item[(1)]
$\delta u=0$, for any arrow $\delta$.
\item[(2)]
$u$ is a loop or a cycle.
\item[(3)]
If $u=u_1\cdots u_r\ne0$ and $v=v_1\cdots v_s\ne0$ are such that
\begin{itemize}
\item[(i)]
$u\delta=0$ and $v\delta=0$, for any arrow $\delta$,
\item[(ii)]
$u_1=v_1$,
\end{itemize}
then $r=s$ and $u_i=v_i$, for all $i$.
\item[(4)]
If $u=u_1\cdots u_r\ne0$ and $v=v_1\cdots v_s\ne0$ are such that
\begin{itemize}
\item[(i)]
$u\delta=0$ and $v\delta=0$, for any arrow $\delta$,
\item[(ii)]
$u$ and $v$ share their initial point.
\end{itemize}
then $u=cv$, for some $c\in\bR^\times$.
\item[(5)]
Let $\Tr: A \to \bR$ be a non-degenerate trace map. If $u=u_1\cdots u_r\ne0$ is such that
$u\delta=0$, for any arrow $\delta$, then $\Tr(u)\ne0$.
\end{itemize}
\end{lemma}
\begin{proof}
(1) follows from the fact that $\Soc(A)$ is the socle of the right and the left regular representations for
self-injective algebras \cite[Thm.(58.12)]{CR62}. To see (2),
let $i$ be the initial point of $u_1$. Then $u$ spans $\Soc(e_iA)\simeq \Top(e_iA)$, which implies that
$ue_i=u\ne0$ and $ue_j=0$, for $j\ne i$. Hence, $i$ is the endpoint of $u_r$.
To prove (3), observe that $\{u_1,u_1u_2,\dots,u\}$ span a uniserial $A$-submodule of $\Rad(e_iA)$. Thus, if $u_1=v_1$
then $\{v_1,v_1v_2,\dots, v\}$ span the same uniserial $A$-submodule. To prove (4), observe that $u, v\in\Soc(e_iA)$, for
the common initial point $i$. Then $\dim \Soc(e_iA)=1$ implies the result. (5) is clear because if $\Tr(u)=0$ then
$\Tr(ux)=0$, for all $x\in A$, which contradicts the assumptions that the trace map is non-degenerate and $u\ne0$.
\end{proof}

\begin{defn}
If $u=u_1\cdots u_r\ne0$, where $u_1,\dots,u_r$ are arrows, satisfies $u\delta=0$ in $\bR Q/I$, for any arrow $\delta$,
we call $u$ a {\it maximal path which extends $u_1$}.
\end{defn}

\noindent
Note that if $u$ is a maximal path, then $\delta u=0$ in $\bR Q/I$, for any arrow $\delta$, by Lemma \ref{useful lemma}(1).

Because of Lemma \ref{useful lemma}(2), we may exclude the case of one arrow
since the unique arrow must be a loop and the quiver $Q$ can not be connected.

\subsection{The case of two arrows}

This case is easy and we omit the proof. We obtain a symmetric Nakayama algebra, which is
the case (1) of Theorem \ref{first classification}. Note that
it can not be derived equivalent to $\fqH(2\delta)$ with $\ell=1$, since
symmetric Nakayama algebras are of finite representation type.

\subsection{The case of three arrows}

If there is no loop, then we may assume that two arrows, say $\alpha$ and $\alpha'$, start at the vertex $0$ and end at the vertex $1$, and the
other arrow, say $\beta$, starts at the vertex $1$ and ends at the vertex $0$.
Since $A$ is special biserial, we may assume $\alpha'\beta=0$ without loss of generality. Then, $\alpha'$ is a maximal path which is not a cycle,
contradicting Lemma \ref{useful lemma}(2). On the other hand, if there are two loops, then we extend the remaining arrow, say $\alpha$,
to a maximal path $u=u_1\cdots u_r$ with $u_1=\alpha$. But $u$ can not be a cycle, which contradicts Lemma \ref{useful lemma}(2) again.
As $Q$ is connected, it does not have three loops. Thus, $Q$ is as in the case (2) of Theorem \ref{first classification}.
We show that the relations are either (2a) or (2b).

Suppose that $\beta\alpha\ne0$. Then $\beta\gamma=\gamma\alpha=0$ since $A$ is special biserial.
We extend $\beta$ to a maximal path. Then it is of the form $(\beta\alpha)^n$, for some $n\ge1$.
Next we extend $\gamma$ and $\alpha$. They are of the form $\gamma^p$ and $(\alpha\beta)^q$, for some $p,q\ge1$, respectively.
Then Lemma \ref{useful lemma}(4) implies that $\gamma^p=c(\alpha\beta)^q$, for some $c\in\bR^\times$.
If $n>q$ then
$$ (\beta\alpha)^n=\beta\cdot(\alpha\beta)^q\cdot\alpha(\beta\alpha)^{n-q-1}=0 $$
by the maximality of $(\alpha\beta)^q$, which contradicts $(\beta\alpha)^n\ne0$. Simlilarly, $n<q$ leads to a contradiction and we have
$n=q$. We may list basis elements of $A$ as follows.
$$
\{ e_0, \gamma,\gamma^2,\dots,\gamma^{p-1}, \beta, \alpha\beta, \beta\alpha\beta, \dots, (\alpha\beta)^q;
e_1, \alpha, \beta\alpha, \alpha\beta\alpha, \dots, (\beta\alpha)^q \}
$$
We may check that the defining relations may be chosen as in (2a).
We define its trace map by the values on the basis elements. Let the values be $0$ except
$$ \Tr((\alpha\beta)^q)=\Tr((\beta\alpha)^q)=1. $$
Suppose that $\Tr(xy)\ne0$, for two basis elements $x$ and $y$. Then we may show that $(x,y)$ is one of the following.
\begin{gather*}
(e_0, (\alpha\beta)^q),\; (e_1, (\beta\alpha)^q),\\
(\gamma^i, \gamma^{p-i}),\\
((\alpha\beta)^j, (\alpha\beta)^{q-j}),\; (\beta, (\alpha\beta)^{q-1}\alpha),\; (\beta(\alpha\beta)^j, (\alpha\beta)^{q-j-1}\alpha),\\
((\beta\alpha)^j, (\beta\alpha)^{q-j}),\; (\alpha, (\beta\alpha)^{q-1}\beta),\; (\alpha(\beta\alpha)^j, (\beta\alpha)^{q-j-1}\beta),\\
((\alpha\beta)^q, e_0),\; ((\beta\alpha)^q, e_1),
\end{gather*}
for $1\le i\le p-1$ and $1\le j \le q-1$. Thus the trace map is well-defined and non-degenerate.
We have proved that $A$ is a symmetric algebra in case (2a).

Next suppose that $\beta\alpha=0$. We extend $\beta$ to a maximal path. As the maximal path is a cycle, it has the form $(\beta\gamma\alpha)^n$, for some $n\ge1$.
It then follows that $\beta\gamma\ne0$, $\gamma\alpha\ne0$ and we have $\gamma^2=0$.
If we extend $\alpha$ to a maximal path, it is either $\alpha\beta$ or $(\alpha\beta\gamma)^m$ or $(\alpha\beta\gamma)^m\alpha\beta$, for some $m\ge1$.
But our assumption is that the algebra $A$ is symmetric. Thus we have a non-degenerate trace map $\Tr: A \to \bR$ and both
$$ \Tr(\alpha\beta)=\Tr(\beta\alpha)=0, \quad \Tr(((\alpha\beta\gamma)^m\alpha\beta)=\Tr(\beta(\alpha\beta\gamma)^m\alpha)=0 $$
contradicts Lemma \ref{useful lemma}(5). Hence,
$(\alpha\beta\gamma)^m$ is the maximal path which extends $\alpha$.

Moreover, if $m\ge n+1$ then
$$ (\alpha\beta\gamma)^m = \alpha\cdot(\beta\gamma\alpha)^n\cdot(\beta\gamma\alpha)^{m-n-1}\beta\gamma=0, $$
which is a contradiction. If $m\le n-1$ then
$$ (\beta\gamma\alpha)^n = \beta\gamma\cdot(\alpha\beta\gamma)^m\cdot(\alpha\beta\gamma)^{n-m-1}\alpha=0, $$
which is a contradiction again. Thus, $m=n$ follows.
Next we extend $\gamma$ to a maximal path. Then it is either $\gamma$ or $(\gamma\alpha\beta)^l$ or $(\gamma\alpha\beta)^l\gamma$, for some $l\ge1$.
But we may exclude $(\gamma\alpha\beta)^l\gamma$ by $\Tr((\gamma\alpha\beta)^l\gamma)=\Tr(\gamma(\gamma\alpha\beta)^l)=0$.
If $(\gamma\alpha\beta)^l$ is a maximal path then $l=n$ because
\begin{gather*}
(\gamma\alpha\beta)^l=\gamma\alpha\cdot(\beta\gamma\alpha)^n\cdot(\beta\gamma\alpha)^{l-n-1}\beta=0, \quad \hbox{ if $l\ge n+1$,} \\
(\beta\gamma\alpha)^n=\beta\cdot(\gamma\alpha\beta)^l\cdot(\gamma\alpha\beta)^{n-l-1}\gamma\alpha=0. \quad \hbox{ if $l\le n-1$.}
\end{gather*}
Thus, if we extend $\gamma$ to a maximal path, then it is either $\gamma$ or $(\gamma\alpha\beta)^n$.

Recall that maximal paths which extends $\alpha$ and $\gamma$ coincide up to a nonzero scalar multiple by Lemma \ref{useful lemma}(4).
If $\gamma$ is a nonzero scalar multiple of $(\alpha\beta\gamma)^n$,
then we have $\gamma=0$, which contradicts $\beta\gamma\ne0$, $\gamma\alpha\ne0$. Therefore, we may exclude $\gamma$ and
$(\gamma\alpha\beta)^n$ is the maximal path which extends $\gamma$.
Now we write $(\alpha\beta\gamma)^n=c(\gamma\alpha\beta)^n$, for $c\in\bR^\times$. Then
$$ \Tr((\alpha\beta\gamma)^n)=\Tr((\gamma\alpha\beta)^n)\ne0 $$
by Lemma \ref{useful lemma}(5) and we deduce $c=1$.
As a result, the following elements form a basis of $A$ and we may choose (2b) as the defining relations:
\begin{align*}
e_0,\; &\beta,\; \alpha\beta,\; \beta(\gamma\alpha\beta)^i,\; \alpha\beta(\gamma\alpha\beta)^i,\; (\gamma\alpha\beta)^i, \\
&\gamma,\; \beta\gamma,\; \gamma(\alpha\beta\gamma)^i,\; \beta\gamma(\alpha\beta\gamma)^i,\; (\alpha\beta\gamma)^i,\; (\gamma\alpha\beta)^n, \\
e_1,\; &\alpha(\beta\gamma\alpha)^j,\; \gamma\alpha(\beta\gamma\alpha)^j,\; (\beta\gamma\alpha)^{j+1},
\end{align*}
for $1\le i\le n-1$ and $0\le j\le n-1$. To show that this algebra is indeed a symmetric algebra,
we define its trace map by the values on basis elements. Let the values be $0$ except
$$ \Tr((\gamma\alpha\beta)^n)=\Tr((\beta\gamma\alpha)^n)=1. $$
Suppose that $\Tr(xy)\ne0$, for basis elements $x$ and $y$. If $x=e_0$ or $e_1$, then $y=(\gamma\alpha\beta)^n$ and
$(\beta\gamma\alpha)^n$, respectively. Otherwise, $xy$ coincides with $(\alpha\beta\gamma)^n$ or $(\beta\gamma\alpha)^n$ or
$(\gamma\alpha\beta)^n$ as words in alphabet $\{\alpha, \beta, \gamma\}$, by Lemma \ref{useful lemma}(3). Hence, it is easy to check that the Gram matrix is a
non-singular symmetric block diagonal matrix whose block size is $1$ or $2$. Hence $A$ is symmetric in case (2b) as desired.

\subsection{The case of four arrows}
Suppose that $Q$ has no loop. Then, since $A$ is special biserial, we have two arrows, say $\alpha$ and $\alpha'$, which starts at the vertex $0$
and ends at the vertex $1$, and two arrows, say $\beta$ and $\beta'$, which starts at the vertex $1$
and ends at the vertex $0$. If $\alpha\beta=\alpha'\beta=\alpha\beta'=\alpha'\beta'=0$ then we can not have a cycle which starts at the vertex $0$,
a contradiction. Thus, we may assume that $\alpha\beta\ne0$ without loss of generality. Then, we have
$\alpha\beta'=0$ and $\alpha'\beta=0$. We consider the cases $\beta\alpha\ne0$ and $\beta\alpha=0$ separately.

Suppose that $\beta\alpha\ne0$. Then we have $\alpha\beta'=\beta'\alpha=\alpha'\beta=\beta\alpha'=0$. We are going to show that
we are in case (3a). If we extend $\alpha$ to a maximal path then it is of the form $(\alpha\beta)^p$, for some $p\ge1$.
If we extend $\alpha'$ to a maximal path then it is of the form $(\alpha'\beta')^q$, for some $q\ge1$.
$(\alpha\beta)^p$ is a nonzero scalar multiple of $(\alpha'\beta')^q$.
Renormalizing $\alpha'$ or $\beta'$, we may assume that $(\alpha\beta)^p=(\alpha'\beta')^q$. Similarly, let $(\beta\alpha)^{p'}$ and $(\beta'\alpha')^{q'}$ be
maximal paths which extends $\beta$ and $\beta'$ respectively. Then we have $p'=p$ by
\begin{gather*}
(\alpha\beta)^p=\alpha(\beta\alpha)^{p'}(\beta\alpha)^{p-p'-1}\beta, \quad \hbox{ if $p'\le p-1$,} \\
(\beta\alpha)^{p'}=\beta(\alpha\beta)^p(\alpha\beta)^{p'-p-1}\alpha, \quad \hbox{ if $p'\ge p+1$,}
\end{gather*}
and $q=q'$ by the similar argument. As $\Tr((\beta\alpha)^p)=\Tr((\beta'\alpha')^q)\ne0$, we
may conclude $(\beta\alpha)^p=(\beta'\alpha')^q$ as well. Thus, we are in case (3a) and
we may give a set of paths which form a basis of $A$.
Defining $\Tr: A \to \bR$ by $\Tr(x)=0$, for basis elements $x$, except for
$$ \Tr((\alpha\beta)^p)=\Tr((\beta\alpha)^p)=\Tr((\alpha'\beta')^q)=\Tr((\beta'\alpha')^q)=1, $$
we may show that $A$ is a symmetric algebra in case (3a).

Next suppose that $\beta\alpha=0$. If we extends $\beta$ to a maximal path, we deduce that $\beta\alpha'\ne0$. Thus,
we have $\alpha\beta'=\beta\alpha=\alpha'\beta=\beta'\alpha'=0$. We are going to show that
we are in (3b). We extend four arrows to get maximal paths
\begin{align*}
(\alpha\beta\alpha'\beta')^a \quad &\hbox{ or } \quad (\alpha\beta\alpha'\beta')^{a-1}\alpha\beta, \\
(\beta\alpha'\beta'\alpha)^b \quad &\hbox{ or } \quad (\beta\alpha'\beta'\alpha)^{b-1}\beta\alpha', \\
(\alpha'\beta'\alpha\beta)^c \quad &\hbox{ or } \quad (\alpha'\beta'\alpha\beta)^{c-1}\alpha'\beta', \\
(\beta'\alpha\beta\alpha')^d \quad &\hbox{ or } \quad (\beta'\alpha\beta\alpha')^{d-1}\beta'\alpha.
\end{align*}
However, the trace values at the elements on the right hand side are $0$, and they can not be maximal paths.
We may also deduce $a=b=c=d$ by the similar argument as above. Denote $a=b=c=d$ by $m$. Then we are in case (3b).
We may give a set of path which is a basis of $A$ and we may define $\Tr: A \to \bR$ by $\Tr(x)=0$, for basis elements $x$, except for
$$
\Tr((\alpha\beta\alpha'\beta')^m)=\Tr((\beta\alpha'\beta'\alpha)^m)=\Tr((\alpha'\beta'\alpha\beta)^m)=\Tr((\beta'\alpha\beta\alpha')^m)=1.
$$
Then the Gram matrix is a non-singular symmetric block diagonal matrix whose block size is $1$ or $2$. Hence, $A$ is a symmetric algebra in case (3b).
No loop cases are classified.

Suppose that $Q$ has one loop. We may assume that the loop is at the vertex $0$. Then, the other three arrows connect the vertices $0$ and $1$.
Hence, at the vertex $0$, we have either two outgoing arrows which ends at the vertex $1$ or two incoming arrows which starts at the vertex $1$.
If there are two outgoing arrows, say $\alpha$ and $\alpha'$, then $\alpha\beta=0$ or $\alpha'\beta=0$ holds, for the remaining arrow $\beta$,
since $A$ is special biserial. But then the maximal path which extends $\alpha$ or $\alpha'$ is
not a cycle, a contradiction. If there are two incoming arrows, the similar argument leads to a contradiction. Thus, $Q$ must have more than one loop.

Suppose that $Q$ has three loops, then the remaining arrow, say $\alpha$, is the only arrow which connects the vertices $0$ and $1$. Then,
a maximal path which extends $\alpha$ can not be a cycle.

As $Q$ is connected, it can not have four loops. Hence, we are left with two loop cases.

The remaining part is for classifying two loop cases.

Since $Q$ is connected, each vertex must have one loop.
The remaining two arrows, which we denote $\alpha$ and $\beta$, connect the vertices $0$ and $1$, and they must have opposite direction. Otherwise,
a maximal path which extends $\alpha$ and $\beta$ would not be a cycle. Hence the quiver $Q$ is as in Theorem \ref{first classification}(4).

Suppose that $\alpha\beta\ne0$ or $\beta\alpha\ne0$. We may assume $\alpha\beta\ne0$ without loss of generality.
Then $\alpha\delta=0$ and $\delta\beta=0$ hold. We consider the cases $\beta\alpha\ne0$ and $\beta\alpha=0$ separately.
We are going to show that we are in case (4a) in the former case, and case (4b) in the latter case.

If $\beta\alpha\ne0$ then $\beta\gamma=\gamma\alpha=\alpha\delta=\delta\beta=0$. Thus, a maximal path which extends $\alpha$ has the
form $(\alpha\beta)^p$, for some $p\ge1$, and a maximal path which extends $\gamma$ has the form $\gamma^q$, for some $q\ge1$.
Thus, we may assume $(\alpha\beta)^p=\gamma^q$. On the other hand, a maximal path which extends $\beta$ has the form
$(\beta\alpha)^{p'}$, for some $p'\ge1$, and a maximal path which extends $\delta$ has the form $\delta^r$, for some $r\ge1$, and
we may assume $(\beta\alpha)^{p'}=\delta^r$. Then,
\begin{align*}
(\alpha\beta)^p =\alpha\cdot(\beta\alpha)^{p'}\cdot(\beta\alpha)^{p-p'-1}\beta, \quad
&\hbox{ if $p'\le p-1$, }\\
(\beta\alpha)^{p'} =\beta(\alpha\beta)^{p'-p-1}\cdot(\alpha\beta)^p \cdot\alpha, \quad
&\hbox{ if $p'\ge p+1$, }
\end{align*}
implies $p'=p$. Thus we are in case (4a). We may show that $A$ is symmetric in this case.

If $\beta\alpha=0$, then we extend $\beta$ to a maximal path and obtain $\beta\gamma\ne0$ and $\gamma\alpha\ne0$. Thus, we have
$\beta\alpha=\gamma^2=\alpha\delta=\delta\beta=0$. The maximal path which extends $\beta$ has the form
$(\beta\gamma\alpha)^p$, for some $p\ge1$, and a maximal path which extends $\delta$ has the form $\delta^q$, for some $q\ge1$,
and we may assume $(\beta\gamma\alpha)^p=\delta^q$. We extend $\alpha$ to a maximal path. It has the form
$(\alpha\beta\gamma)^m$, for some $m\ge1$, and we may prove $m=p$. Similarly, We extend $\gamma$ to a maximal path. It has the form
$(\gamma\alpha\beta)^n$, for some $n\ge1$, and we may prove $n=p$. Since $(\alpha\beta\gamma)^p$ is a nonzero scalar multiple of
$(\gamma\alpha\beta)^p$ and $\Tr((\alpha\beta\gamma)^p)=\Tr((\gamma\alpha\beta)^p)\ne0$, $(\gamma\alpha\beta)^p=(\alpha\beta\gamma)^p$ holds,
and we are in case (4b). We may show that $A$ is symmetric in this case.

Finally, we consider the case $\alpha\beta=\beta\alpha=0$.
We extend $\alpha$ to a maximal path and obtain $\alpha\delta\ne0$ and $\delta\beta\ne0$.
We extend $\beta$ to a maximal path and obtain $\beta\gamma\ne0$ and $\gamma\alpha\ne0$. Thus, we have
$\alpha\beta=\beta\alpha=\gamma^2=\delta^2=0$.

Our task is to show that we are in case (4c). Note that the possibilities for maximal paths which extend four arrows are
\begin{align*}
(\alpha\delta\beta\gamma)^a \quad &\hbox{ or } \quad (\alpha\delta\beta\gamma)^{a-1}\alpha\delta\beta, \\
(\beta\gamma\alpha\delta)^b \quad &\hbox{ or } \quad (\beta\gamma\alpha\delta)^{b-1}\beta\gamma\alpha, \\
(\gamma\alpha\delta\beta)^c \quad &\hbox{ or } \quad (\gamma\alpha\delta\beta)^{c-1}\gamma, \\
(\delta\beta\gamma\alpha)^d \quad &\hbox{ or } \quad (\delta\beta\gamma\alpha)^{d-1}\delta.
\end{align*}
Then, we may exclude the elements on the right hand side because their trace values are $0$.
We have $c=a$ by
\begin{align*}
(\gamma\alpha\delta\beta)^c =\gamma\cdot(\alpha\delta\beta\gamma)^a\cdot(\alpha\delta\beta\gamma)^{c-a-1}\alpha\delta\beta, \quad
&\hbox{ if $c\ge a+1$, }\\
(\alpha\delta\beta\gamma)^a
=\alpha\delta\beta\cdot(\gamma\alpha\delta\beta)^c\cdot(\gamma\alpha\delta\beta)^{a-c-1}\gamma, \quad
&\hbox{ if $c\le a-1$.}
\end{align*}
It follows that $(\alpha\delta\beta\gamma)^a=(\gamma\alpha\delta\beta)^a$.
$(\delta\beta\gamma\alpha)^d$ is the maximal path which extends $\delta$, and it is easy to prove $d=a$. Similarly,
$(\beta\gamma\alpha\delta)^b$ is the maximal path which extends $\beta$, and it is easy to prove $b=a$ as well.
It follows that $(\beta\gamma\alpha\delta)^a=(\delta\beta\gamma\alpha)^a$, and if we denote $a=b=c=d$ by $m$, we are in case (4c).

\subsection{Classification}

Not all the algebras in the list of Theorem \ref{first classification} appear as a finite quiver Hecke algebra
$\fqH(\Lambda_0-\kappa+2\delta)$, for $\kappa\in\weyl\Lambda_0$. In this section,
we compute the center and the stable Auslander-Reiten quiver for those algebras from the list
and give classification when $\lambda\ne0$.
In the special case when we consider weight two blocks of the symmetric group or its Hecke algebra,
Scopes' equivalence shows that weight two blocks with $2$-core $(k,k-1,\dots,1)$ with $k\ge1$ are
Morita equivalent to each other, and we have only two Morita equivalent classes among them. We may generalize this fact.
Namely, if $\lambda\ne0$ then any special biserial finite quiver Hecke algebras 
$\fqH(\beta)$ of tame type is Morita equivalent to either the Hecke algebra $\mathcal{H}_4(q)|_{q=-1}$
associated with the symmetric group $S_4$, or the principal block of the
Hecke algebra $\mathcal{H}_5(q)|_{q=-1}$ associated with the symmetric group $S_5$.
If $\lambda=0$ then a special biserial finite quiver Hecke algebra of tame type must be Morita equivalent to one of the algebras from case (4b)
of Theorem \ref{first classification} with $q=2$,
but we do not pursue further to check if they actually occur as finite quiver Hecke algebras.

\begin{lemma} \label{stable AR quiver}
Suppose that $\ell=1$.

\begin{itemize}
\item[(1)]
If $\lambda=0$ then $\fqH(2\delta)$ is not of polynomial growth and its stable
Auslander-Reiten quiver has
\begin{enumerate}
\item[(i)]
the unique component of $\Z A_\infty/\langle \tau^3\rangle$,
\item[(ii)]
infinitely many components of $\Z A_\infty^\infty$,
\item[(iii)]
infinitely many components of homogeneous tubes.
\end{enumerate}
Further, its center is $5$ dimensional commutative local algebra.

\item[(2)]
If $\lambda\ne0$ then $\fqH(2\delta)$ is domestic and its stable
Auslander-Reiten quiver has
\begin{enumerate}
\item[(i)]
the unique component of $\Z\widetilde{A}_{2,2}$,
\item[(ii)]
two components of $\Z A_\infty/\langle \tau^2\rangle$,
\item[(iii)]
infinitely many components of homogeneous tubes.
\end{enumerate}
Further, its center is $5$ dimensional commutative local algebra.
\begin{proof}
(1) We follow the argument in \cite[Prop.5.6]{AP13}. Let $A$ be the basic algebra of $\fqH(2\delta)$.
Recall that if $\lambda=0$ then the quiver presentation of $A$ is
$$
\xy
(0,0) *{\begin{array}{c} 0 \end{array}}="A", (20,0) *{ \begin{array}{c} 1 \end{array} }="B"
\ar @/_/_{\alpha} "A";"B"
\ar @/_/_{\beta} "B";"A"
\ar @(lu,ld)_\gamma "A";"A"
\ar @(rd,ru)_\delta "B";"B"
\endxy
$$
with relations $\alpha\beta = \beta\gamma = \gamma\alpha = \delta^2 = 0$, $\gamma^2 = \alpha\delta\beta$, $\beta\alpha\delta = \delta\beta\alpha$.
Let $a=\alpha\delta^{-1}\beta\gamma^{-1}$ and $b=\alpha\delta^{-1}\beta$. Then, for each prime $q$,
$$ \{ x_1x_2\cdots x_q \mid x_i=a\;\text{or}\; b \}\setminus\{a^q, b^q\} $$
defines $(2^q-2)/q$ equivalence classes of bands, for the string algebra $A/\Rad(A)$. Hence, it suffices to compute the $\tau$-orbits of string modules
$Ae_1/A\alpha$, $Ae_0/A\beta$, $Ae_0/A\gamma$ and $Ae_1/A\delta$, in order to know the components of the stable Auslander-Reiten quiver of $A$, by the
general result \cite[Thm.2.2]{ES92}. Then, explicit computation shows
\begin{gather*}
\tau(Ae_1/A\alpha)=Ae_0/A\beta, \;\; \tau(Ae_0/A\beta)=Ae_0/A\gamma, \;\; \tau(Ae_0/A\gamma)=Ae_1/A\alpha, \\
\tau(Ae_1/A\delta)=Ae_1/A\delta,
\end{gather*}
where $Ae_0=\Span\{e_0, \beta, \gamma, \delta\beta, \gamma^2\}$ and
$Ae_1=\Span\{e_1, \alpha, \delta, \beta\alpha, \alpha\delta, \beta\alpha\delta\}$.

The elements that commutes with $e_0$ and $e_1$ are
$$ e_0Ae_0\oplus e_1Ae_1=\Span\{e_0, \gamma, \gamma^2, e_1, \delta, \beta\alpha, \beta\alpha\delta\}. $$
It is clear that $\Soc(A)=\Span\{ \gamma^2, \beta\alpha\delta\}$ is contained in the center. As the center is a local
algebra, it suffices to find central elements in $\Span\{\gamma, \delta, \beta\alpha\}$. Then, $\gamma$ and $\beta\alpha$ are central
and $\delta$ is not central. We have $Z(A)=\Span\{1, \gamma, \beta\alpha, \gamma^2, \beta\alpha\delta\}$.

\medskip
\noindent
(2) Let $A$ be the basic algebra of $\fqH(2\delta)$ as above and recall that if $\lambda\ne0$ then the quiver presentation of $A$ is
$$
\xy
(0,0) *{\begin{array}{c} 0 \end{array}}="A", (20,0) *{ \begin{array}{c} 1 \end{array} }="B"
\ar @/_/_{\alpha} "A";"B"
\ar @/_/_{\beta} "B";"A"
\ar @(rd,ru)_\gamma "B";"B"
\endxy
$$
with relations $ \alpha\gamma = \gamma\beta = 0 $, $(\beta\alpha)^2 = \gamma^2$. Then $\beta\alpha\gamma^{-1}$ is the unique band,
for the string algebra $A/\Rad(A)$. Hence, \cite[Thm.2.1]{ES92} tells that there are positive integers $m$, $p$, $q$ such that the stable Auslander-Reiten quiver has
$m$ components of tubes $\Z A_\infty/\langle \tau^p\rangle$, $m$ components of tubes $\Z A_\infty/\langle \tau^q\rangle$,
$m$ components of $\Z\widetilde{A}_{p,q}$, and the remaining components are infinitely many homogeneous tubes. Hence, it suffices to
compute the $\tau$-orbits of string modules $Ae_1/A\alpha$, $Ae_0/A\beta$, $Ae_1/A\gamma$, and $Ae_0/\Soc(Ae_0)$ because
$\Rad(Ae_0)/\Soc(Ae_0)$ is indecomposable. $Ae_0/A\beta$ is a simple $A$-module and explicit computation shows
$$ \tau(Ae_1/A\alpha)=Ae_0/A\beta, \quad \tau(Ae_0/A\beta)=Ae_1/A\alpha. $$
On the other hand, $\tau(Ae_1/A\gamma)=Ae_0/\Soc(Ae_0)$ by explicit computation, and the almost split sequence
$$ 0 \to \Rad(Ae_0) \to Ae_0\oplus \Rad(Ae_0)/\Soc(Ae_0) \to Ae_0/\Soc(Ae_0) \to 0 $$
shows that $\tau(Ae_0/\Soc(Ae_0))=\Rad(Ae_0)=Ae_1/A\gamma$. Thus, we may conclude that $m=1$ and $p=q=2$.
The computation of the center is similar to (1) and we obtain $Z(A)=\Span\{1, \alpha\beta+\beta\alpha, (\alpha\beta)^2, \gamma, \gamma^2\}$.
\end{proof}

\end{itemize}
\end{lemma}

\begin{thm} \label{second classification}
Any special biserial finite quiver Hecke algebra $\fqH(\Lambda_0-\kappa+2\delta)$ in affine type $A^{(1)}_1$ is Morita equivalent to

\begin{itemize}
\item[(1)]
one of the algebras from Theorem \ref{first classification}(4b) with $q=2$ if $\lambda=0$.

\item[(2)]
the algebra from Theorem \ref{first classification}(2a) with $p=q=2$, or (4a) with $p=1$ and $q=r=2$ if $\lambda\ne0$.

\end{itemize}
\end{thm}
\begin{proof}
We compute the stable Auslander-Reiten quivers for the algebras from Theorem \ref{first classification}.

Suppose that we are in case (2a) and denote the algebra by $A$. The relations are
$$ \beta\gamma=\gamma\alpha=0, \;\; \gamma^p=(\alpha\beta)^q. $$
Then, as $A$ is symmetric, we may compute $\tau(Ae_1/A\alpha)$ by computing the kernel of $Ae_0\to Ae_1$ given by right multiplication by $\alpha$.
$$
\begin{array}{cc}
\qquad\qquad\qquad e_0 & \\
\gamma & \beta \\
\gamma^2 & \alpha\beta \\
\vdots & \vdots \\
\qquad\qquad\qquad (\alpha\beta)^q &
\end{array}
\quad
\stackrel{\cdot\alpha}{\longrightarrow}
\quad
\begin{array}{c}
e_1 \\
\alpha \\
\beta\alpha \\
\vdots \\
(\beta\alpha)^q
\end{array}
$$
We may compute $\tau(Ae_0/A\beta)$ and $\tau(Ae_0/A\gamma)$ in the similar way. We obtain
\begin{align*}
\tau(Ae_1/A\alpha)&\simeq\Ker(Ae_0\stackrel{\cdot\alpha}{\rightarrow}Ae_1)\simeq Ae_0/A\beta, \\
\tau(Ae_0/A\beta)&\simeq\Ker(Ae_1\stackrel{\cdot\beta}{\rightarrow}Ae_0)\simeq Ae_1/A\alpha, \\
\tau(Ae_0/A\gamma)&\simeq\Ker(Ae_0\stackrel{\cdot\gamma}{\rightarrow}Ae_0)\simeq Ae_1/\Soc(Ae_1),
\end{align*}
and $\tau(Ae_1/\Soc(Ae_1))\simeq Ae_0/A\gamma$. Thus, if it is the basic algebra of a finite quiver Hecke algebra of tame representation type,
we must have $\lambda\ne0$ by Lemma \ref{stable AR quiver} and $A$ must be domestic. However,
if $p\ge3$ then we may use $a=\alpha\beta\gamma^{-1}$ and $b=\alpha\beta\gamma^{-2}$ to show that $A$ is not of polynomial growth.
Similarly, if  $q\ge3$ then we may use $a=\alpha\beta\gamma^{-1}$ and $b=(\alpha\beta)^2\gamma^{-1}$ to show that $A$ is not of polynomial growth.
Thus, we have either $(p,q)=(2,1)$ or $(2,2)$. But if $(p,q)=(2,1)$ then $A$ is a Brauer tree algebra because
\begin{align*}
Ae_0\simeq
\begin{array}{c}  S_0 \\  S_0 \oplus S_1 \\  S_0 \end{array}, \qquad
Ae_1\simeq
\begin{array}{c}  S_1 \\  S_0 \\  S_1 \end{array}.
\end{align*}
Hence, we may exclude this case. The case $p=q=2$ is nothing but the basic algebra of $\fqH(2\delta)$, for $\lambda\ne0$.
We conclude that this is the only possibility in case (2a). Note that it is Morita equivalent to
the Hecke algebra $\mathcal{H}_4(q)|_{q=-1}$ associated with the symmetric group $S_4$.

Suppose that we are in case (2b). The relations are
$$ \beta\alpha=\gamma^2=0, \;\; (\gamma\alpha\beta)^m=(\alpha\beta\gamma)^m. $$
Then we consider the following map $Ae_0 \to Ae_1$, for computing $\tau(Ae_1/A\alpha)$,
$$
\begin{array}{cc}
\qquad\qquad\qquad e_0 & \\
\beta & \gamma \\
\alpha\beta & \beta\gamma \\
\gamma\alpha\beta & \alpha\beta\gamma \\
\vdots & \vdots \\
\qquad\qquad\qquad (\alpha\beta\gamma)^m &
\end{array}
\quad
\stackrel{\cdot\alpha}{\longrightarrow}
\quad
\begin{array}{c}
e_1 \\
\alpha \\
\gamma\alpha \\
\beta\gamma\alpha \\
\vdots \\
(\beta\gamma\alpha)^m
\end{array}
$$
and consider the similar maps for other arrows to obtain
\begin{align*}
\tau(Ae_1/A\alpha)&\simeq Ae_1/\Soc(Ae_1), \;\; \tau(Ae_1/\Soc(Ae_1))\simeq \Rad(Ae_1)\simeq Ae_0/A\beta, \\
\tau(Ae_0/A\beta)&\simeq Ae_1/A\alpha, \;\;
\tau(Ae_0/A\gamma)\simeq Ae_0/A\gamma.
\end{align*}
As there is a tube of period $3$, if $A$ is the basic algebra of a finite quiver Hecke algebra of tame representation type,
we must have $\lambda=0$ by Lemma \ref{stable AR quiver} and $A$ is not of polynomial growth.
Hence its center is $5$ dimensional. However, we have
$$ Z(A)=\Span\{1, \alpha\beta(\gamma\alpha\beta)^{m-1}, (\alpha\beta\gamma)^m, (\beta\gamma\alpha)^m\} $$
by explicit computation as in (1), and we conclude that case (2b) can not occur.

Suppose that we are in case (3a). The relations are
$$ \alpha\beta'=\beta'\alpha=\alpha'\beta=\beta\alpha'=0,\;\;(\alpha\beta)^p=(\alpha'\beta')^q,\;\;(\beta\alpha)^p=(\beta'\alpha')^q. $$
Then, we have
\begin{align*}
\tau(Ae_1/A\alpha)&\simeq Ae_1/A\alpha, \;\; \tau(Ae_1/A\alpha')\simeq Ae_1/A\alpha', \\
\tau(Ae_0/A\beta)&\simeq Ae_0/A\beta, \;\; \tau(Ae_0/A\beta')\simeq Ae_0/A\beta'.
\end{align*}
There is no tube of period greater than $1$, and case (3a) can not occur by Lemma \ref{stable AR quiver}.

Suppose that we are in case (3b). The relations are
$$
\alpha\beta'=\beta\alpha=\alpha'\beta=\beta'\alpha'=0,\;\;(\alpha\beta\alpha'\beta')^m=(\alpha'\beta'\alpha\beta)^m,\;\;
(\beta\alpha'\beta'\alpha)^m=(\beta'\alpha\beta\alpha')^m.
$$
Then, the radical series of $Ae_0$ and $Ae_1$ are as follows.
$$
\begin{array}{cc}
\qquad\qquad\qquad e_0 & \\
\beta & \beta' \\
\alpha\beta & \alpha'\beta' \\
\beta'\alpha\beta & \beta\alpha'\beta' \\
\vdots & \vdots
\end{array}
\qquad
\begin{array}{cc}
\qquad\qquad\qquad e_1 & \\
\alpha & \alpha' \\
\beta'\alpha & \beta\alpha' \\
\alpha'\beta'\alpha & \alpha\beta\alpha' \\
\vdots & \vdots
\end{array}
$$
It follows that
\begin{align*}
\tau(Ae_1/A\alpha)&\simeq Ae_1/A\alpha', \;\; \tau(Ae_1/A\alpha')\simeq Ae_1/A\alpha, \\
\tau(Ae_0/A\beta)&\simeq Ae_0/A\beta', \;\; \tau(Ae_0/A\beta')\simeq Ae_0/A\beta.
\end{align*}
Thus, Lemma \ref{stable AR quiver} implies that we must have $\lambda\ne0$ and $A$ is domestic. However,
we may use $a=\alpha\beta(\alpha'\beta')^{-1}$ and $b=\beta'\alpha(\beta\alpha')^{-1}$ to show that $A$ is not of polynomial growth.
We conclude that case (3b) can not occur.

Suppose that we are in case (4a). The relations are
$$ \beta\gamma=\gamma\alpha=\alpha\delta=\delta\beta=0,\;\; (\alpha\beta)^p=\gamma^q,\;\; (\beta\alpha)^p=\delta^r. $$
As before, we compute that the Auslander-Reiten translate swaps
$Ae_1/A\alpha$ and $Ae_0/A\beta$, $Ae_0/A\gamma$ and $Ae_1/A\delta$, respectively. Thus, Lemma \ref{stable AR quiver} implies that
we must have $\lambda\ne0$ and $A$ is domestic. If $q\ge3$ then we may use $a=\alpha\beta\gamma^{-1}$ and $b=\alpha\beta\gamma^{-2}$
to show that $A$ is not of polynomial growth. If $r\ge3$ then we may use $a=\beta\alpha\delta^{-1}$ and $b=\beta\alpha\delta^{-2}$
to show that $A$ is not of polynomial growth. Thus, $q=r=2$ follows. Similarly, if $p\ge3$ then
we may use $a=\alpha\beta\gamma^{-1}$ and $b=(\alpha\beta)^2\gamma^{-1}$ to show that $A$ is not of polynomial growth. Hence,
we have either $p=1$ or $p=2$. If $p=2$ and $q=r=2$ then
$$
Ae_0=\Span\{e_0, \beta, \alpha\beta, \beta\alpha\beta, \gamma, \gamma^2\}, \quad 
Ae_1=\Span\{e_1, \alpha, \beta\alpha, \alpha\beta\alpha, \delta, \delta^2\},
$$
and the center is $Z(A)=\Span\{1, \alpha\beta+\beta\alpha, \gamma,\gamma^2,\delta,\delta^2\}$, which is not $5$ dimensional. We conclude that
$p=1$ and $q=r=2$ is the only possibility in case (4a).
Note that it is Morita equivalent to the principal block of the
Hecke algebra $\mathcal{H}_5(q)|_{q=-1}$ associated with the symmetric group $S_5$.

Suppose that we are in case (4b). The relations are
$$
\beta\alpha=\gamma^2=\alpha\delta=\delta\beta=0,\;\; (\gamma\alpha\beta)^p=(\alpha\beta\gamma)^p,\;\;
(\beta\gamma\alpha)^p=\delta^q.
$$
Then, the radical series of $Ae_0$ and $Ae_1$ are as follows.
$$
\begin{array}{cc}
\qquad\qquad\qquad e_0 & \\
\beta & \gamma \\
\alpha\beta & \beta\gamma \\
\gamma\alpha\beta & \alpha\beta\gamma \\
\vdots & \vdots
\end{array}
\qquad
\begin{array}{cc}
\qquad\qquad\qquad e_1 & \\
\alpha & \delta \\
\gamma\alpha & \delta^2 \\
\beta\gamma\alpha & \delta^3 \\
\vdots & \vdots
\end{array}
$$
It follows that
\begin{align*}
\tau(Ae_1/A\alpha)&\simeq Ae_1/A\delta, \;\; \tau(Ae_1/A\delta)\simeq Ae_0/A\beta, \\
\tau(Ae_0/A\beta)&\simeq Ae_1/A\alpha, \;\; \tau(Ae_0/A\gamma)\simeq Ae_0/A\gamma.
\end{align*}
Thus, Lemma \ref{stable AR quiver} implies that we must have $\lambda=0$.
Further, we have
$$ Z(A)=\Span\{ 1, \alpha\beta(\gamma\alpha\beta)^{p-1}, (\gamma\alpha\beta)^p, \delta, \delta^2,\dots, \delta^q \}, $$
which forces $q=2$.
Here, we do not pursue further to determine which $p$ actually occur as a finite quiver Hecke algebra.
We have computed the first Hochshild cohomology group and the answer is $\dim HH^1(A)=q$, which does not determine $p$. 
But higher Hochshild cohomology groups might be helpful: use the method in \cite{Erd13} to compute them. 

Finally, we suppose that we are in case (4c). The relations are
$$
\alpha\beta=\beta\alpha=\gamma^2=\delta^2=0,\;\;(\beta\gamma\alpha\delta)^m=(\delta\beta\gamma\alpha)^m,\;\;
(\gamma\alpha\delta\beta)^m=(\alpha\delta\beta\gamma)^m.
$$
Then, the radical series of $Ae_0$ and $Ae_1$ are as follows.
$$
\begin{array}{cc}
\qquad\qquad\qquad e_0 & \\
\beta & \gamma \\
\delta\beta & \beta\gamma \\
\alpha\delta\beta & \delta\beta\gamma \\
\vdots & \vdots
\end{array}
\qquad
\begin{array}{cc}
\qquad\qquad\qquad e_1 & \\
\alpha & \delta \\
\gamma\alpha & \alpha\delta \\
\beta\gamma\alpha & \gamma\alpha\delta \\
\vdots & \vdots
\end{array}
$$
It follows that
\begin{align*}
\tau(Ae_1/A\alpha)&\simeq Ae_1/A\alpha, \;\; \tau(Ae_0/A\beta)\simeq Ae_0/A\beta, \\
\tau(Ae_0/A\gamma)&\simeq Ae_0/A\gamma, \;\; \tau(Ae_1/A\delta)\simeq Ae_1/A\delta.
\end{align*}
There is no tube of period greater than $1$, and
case (4c) can not occur by Lemma \ref{stable AR quiver}.
\end{proof}

\subsection{Tame finite quiver Hecke algebras}
In the previous subsection, we classified special biserial finite quiver Hecke algebras. In this subsection, 
we show that if $\lambda\ne0$ then any tame finite quiver Hecke algebras are biserial algebras. As finite quiver Hecke algebras of tame type is 
derived equivalent to $\fqH(2\delta)$, they are stably equivalent to the special biserial algebra given by the quiver
$$
\xy
(0,0) *{\begin{array}{c} 0 \end{array}}="A", (20,0) *{ \begin{array}{c} 1 \end{array} }="B"
\ar @/_/_{\alpha} "A";"B"
\ar @/_/_{\beta} "B";"A"
\ar @(rd,ru)_\gamma "B";"B"
\endxy
$$
with relations $ \alpha\gamma = \gamma\beta = 0 $, $(\beta\alpha)^2 = \gamma^2$. We denote the algebra by $A$. 
Let $S_0=\bR e_0$ and $S_1=\bR e_1$ be the irreducible $A$-modules and $P_0, P_1$ their projective covers, respectively. 
For a string $C$, we denote the corresponding string module by $M(C)$. See \cite[Def.5.2]{AP13}. 

In the terminology of the appendix, the stable equivalence to $\fqH(2\delta)$ 
defines a maximal system of orthogonal stable bricks for $A$. 
Hence, its classification allows us to determine the quivers of all tame finite quiver Hecke algebras for $\lambda\ne0$.

\begin{lemma}\label{s.o.s.b}
A system of orthogonal stable bricks for $A$ is one of the following pairs of string modules.
\begin{itemize}
\item[(1)]
$X_0=S_0, X_1=S_1$.
\item[(2)]
$X_0=M(\beta), X_1=M(\alpha\beta\alpha)$.
\item[(3)]
$X_0=M(\beta\alpha\beta), X_1=M(\alpha)$.
\item[(4)]
$X_0=M(\beta\alpha\beta), X_1=M(\alpha\beta\alpha)$. 
\end{itemize}
\end{lemma}
\begin{proof}
Maximal paths in the quiver are $\alpha\beta\alpha, \beta\alpha\beta, \gamma$. Suppose that $\gamma$ appears in the string of $X_i$. 
Then, as $\gamma$ is a maximal path, $S_1$ appears in both $\Top(X_i)$ and $\Soc(X_i)$ and we have $r\in \Rad\End(X_i)$ defined by
$$
r: X_i\twoheadrightarrow S_1\hookrightarrow X_i.
$$
If it factors through a projective module, it factors through $P_1$. Let $f:X_i\to P_1$, $p_1:P_1\to S_1$ be such that 
$p: X_i\to \Im(r)\simeq S_1$ is $fp_1$. Then, there exists $g:P_1\to X_i$ with $gp=p_1$, so that $p_1=gfp_1$. It follows that 
$1-gf\in \Rad\End(P_1)$ is nilpotent and $gf$ is invertible. Then $P_1$ is a direct summand of $X_i$, a contradiction. 
Thus, $\underline\End(X_i)$ contains $\underline{1}$ and $\underline{r}$, which are linearly independent. Since 
$\underline\End(X_i)=\bR$, this is impossible. Hence the string does not contain $\gamma$. It implies that 
the string cannot contain a substring of the form $uv^{-1}, u^{-1}v$, for arrows $u, v$. Hence, the string is one of 
$$
e_0, e_1, \alpha, \beta, \alpha\beta, \beta\alpha, \alpha\beta\alpha, \beta\alpha\beta.
$$
We may delete the possibility of $\alpha\beta$ and $\beta\alpha$ because $S_0$ or $S_1$ appears in both $\Top(X_i)$ and $\Soc(X_i)$. 
On the other hand, $r\in\End(M(\alpha\beta\alpha))$ with $\Im(r)=\Rad^2(M(\alpha\beta\alpha))$ factors through $P_1$ and 
$\underline\End(M(\alpha\beta\alpha))=\bR$. Similarly, we have $\underline\End(M(\beta\alpha\beta))=\bR$. 

Suppose that $X_0=S_0$. If $X_1=S_1$, it satisfies $\underline\Hom(X_0,X_1)=\underline\Hom(X_1,X_0)=0$, which gives case (1). 
If $X_1=M(\beta)$ or $M(\beta\alpha\beta)$, we may find a nonzero element in $\underline\Hom(X_1, X_0)$. 
If $X_1=M(\alpha)$ or $M(\alpha\beta\alpha)$, we also have $\underline\Hom(X_0,X_1)\ne0$. 
Suppose that $X_1=S_1$ and we check the possibilities for
$$
X_0=M(\alpha),\; M(\beta),\; M(\alpha\beta\alpha),\; M(\beta\alpha\beta).
$$
However, we have $\underline\Hom(X_1,X_0)\ne0$ for $X_0=M(\beta)$ or $M(\beta\alpha\beta)$, and 
$\underline\Hom(X_0,X_1)\ne0$ for $X_0=M(\alpha)$ or $M(\alpha\beta\alpha)$. Next suppose that $X_0=M(\beta)$ and check 
the possibilities for $X_1=M(\alpha), M(\alpha\beta\alpha), M(\beta\alpha\beta)$. We have 
$\underline\Hom(X_0,X_1)\ne0$ for $M(\alpha)$ and $M(\beta\alpha\beta)$ on the one hand, $X_1= M(\alpha\beta\alpha)$ gives case (2). 
If $X_1=M(\alpha)$, $\underline\Hom(X_1,X_0)\ne0$ for $X_0=M(\alpha\beta\alpha)$, while $X_0=M(\beta\alpha\beta)$ gives case (3). 
Finally, we see that $\underline\Hom(X_0,X_1)=\underline\Hom(X_1,X_0)=0$ for the remaining case (4). 
\end{proof}

The case (4) from Lemma \ref{s.o.s.b} does not occur: the projective resolutions for $X_0$ and $X_1$ are
\begin{align*}
\cdots\cdots\to P_1\to P_0 \to P_0 \to X_0 &\to 0 \\
\cdots\cdots\to P_0\to P_1 \to P_1 \to X_1 &\to 0
\end{align*}
and we have $\Ext^1(X_i,X_i)=0$, for $i=0,1$, which implies that the finite quiver Hecke algebra is of finite type and not of tame type. 

Similarly, the case (3) does not occur: $0 \to M(\gamma^{-1}\beta\alpha) \to P_1 \to X_1\to 0$ gives
$$
\Hom(P_1,X_0) \to \Hom(M(\gamma^{-1}\beta\alpha), X_0) \to \Ext^1(X_1,X_0)\to 0.
$$
Then $\Hom(M(\gamma^{-1}\beta\alpha), X_0)=\bR$ is given by $M(\gamma^{-1}\beta\alpha)\twoheadrightarrow \Soc(X_0)$ and it is the image of 
$P_1\twoheadrightarrow \Rad(X_0)$. Thus, we have $\Ext^1(X_1,X_0)=0$. On the other hand, direct computation using the projective resolution of 
$X_0$ shows that $\Ext^1(X_0,X_1)=\bR$. However, 
as irreducible modules for finite quiver Hecke algebras are self-dual with respect to the anti-involution fixing the generators, 
we must have $\Ext^1(X_0,X_1)\simeq \Ext^1(X_1,X_0)$.

\begin{lemma}
Let $A$ be as above, and we consider maximal systems of orthogonal stable bricks corresponding to tame finite quiver Hecke algebras. 
Then one of the following holds. 
\begin{itemize}
\item[(1)]
$X_0\simeq S_0$, $X_1\simeq S_1$ and 
$$
\Ext^1(X_0,X_0)=0,\;\; \Ext^1(X_1,X_1)=\bR, \;\; \Ext^1(X_0,X_1)=\Ext^1(X_1,X_0)=\bR.
$$ 
\item[(2)]
$X_0\simeq M(\beta), X_1\simeq M(\alpha\beta\alpha)$ and 
$$
\Ext^1(X_0,X_0)=\bR,\;\; \Ext^1(X_1,X_1)=0, \;\; \Ext^1(X_0,X_1)=\Ext^1(X_1,X_0)=\bR.
$$
\end{itemize}
\end{lemma}
\begin{proof}
(1) is clear. (2) follows from direct computation using resolutions of $X_0$ and $X_1$. 
\end{proof}

\begin{prop}
Let $A$ be as above. Suppose that the pair $X_0=M(\alpha), X_1=M(\beta\alpha\beta)$ is a maximal s.o.s.b. and let 
$M_0$ and $M_1$ be the corresponding s-projective $A$-modules. Then
\begin{itemize}
\item[(1)]
$M_0\simeq M(\beta\alpha\gamma^{-1})$ and 
$M(\beta)\oplus M(\alpha\beta\alpha\gamma^{-1}) \to M_0$ is minimal right almost split. 
\item[(2)]
$M_1\simeq S_0$ and $M(\beta^{-1}\gamma) \to M_1$ is minimal right almost split.
\end{itemize}
\end{prop}
\begin{proof}
As $M_i\simeq \tau^{-1}\Omega(X_i)$ by Proposition \ref{s-projectives}, we use the combinatorial rule to give 
almost split sequences for string modules to obtain the result. 
\end{proof}

\begin{prop}
Let $\weyl$ be the Weyl group of type $A^{(1)}_1$. 
If $\lambda\ne0$ then tame finite quiver Hecke algebras $\fqH(\Lambda_0-w\Lambda_0+2\delta)$, for $w\in\weyl$, are biserial. 
\end{prop}
\begin{proof}
We check the conditions (a) (b) (c) from Proposition \ref{characterization of stably biserial} and use Corollary \ref{biserialness}. 
By Proposition \ref{sosb}, we may check them by using a maximal system of orthogonal stable bricks $\{X_0, X_1\}$. If $\{X_0, X_1\}$ is 
$\{S_0, S_1\}$, they are clearly satisfied. Thus, we may assume that $X_0=M(\beta), X_1=M(\alpha\beta\alpha)$. 
First, the following (i) (ii) (iii) shows that the condition (a) holds. 
\begin{itemize}
\item[(i)]
If $N=M(\beta)$ then 
$$
\underline\Hom(N,X_0)=\underline\Hom(X_0,N)=\bR, \quad \underline\Hom(N,X_1)=\underline\Hom(X_1,N)=0.
$$ 
\item[(ii)]
If $N=M(\alpha\beta\alpha\gamma^{-1})$ then 
$$
\underline\Hom(N,X_0)=\underline\Hom(X_0,N)=0, \quad \underline\Hom(N,X_1)=\underline\Hom(X_1,N)=\bR.
$$
\item[(iii)]
If $N=M(\beta^{-1}\gamma)$ then 
$$
\underline\Hom(N,X_0)=\underline\Hom(X_0,N)=\bR, \quad \underline\Hom(N,X_1)=\underline\Hom(X_1,N)=0.
$$
\end{itemize}

We set $Y_1=M(\beta)$, $Y_2=M(\alpha\beta\alpha\gamma^{-1})$ and denote 
$w_1:Y_1\hookrightarrow M_0$ and $w_2:Y_2\twoheadrightarrow M_0$. Let $p$ be either 
$M_0=M(\beta\alpha\gamma^{-1})\to M(\beta)$ or $M(\beta^{-1}\gamma)$. We have to show that 
$w_1p$ or $w_2p$ factors through a projective module. By inspection, we see that $w_1p=0$. Hence the condition (b) holds. 
To prove the condition (c), we set $Y_0=M(\beta)$ and $Y_1=M(\beta^{-1}\gamma)$, because the pair is the unique pair 
which satisfies $\underline\Hom(Y_1,X_i)=\bR$ and $\underline\Hom(Y_2,X_i)=\bR$, for some $i$. We choose  
$p_i:M_0\to Y_i$ in such a way that its composition with $Y_i\to X_0$ is nonzero in $\underline\Hom(M_0,X_0)$. 
We have to show that $wp_0$ or $wp_1$ factors through a projective module, for  
$w: M(\beta)\hookrightarrow M_0$ and $w: M(\alpha\beta\alpha\gamma^{-1})\twoheadrightarrow M_0$. 
If $w:M(\beta)\hookrightarrow M_0$ then $wp_1=0$. If $w: M(\alpha\beta\alpha\gamma^{-1})\twoheadrightarrow M_0$ then 
$wp_0$ factors through $P_0$, proving the condition (c). 
\end{proof}

\section{Appendix}

This section is for explaining some results from \cite{Po94}. As the main results in \cite{Po94} are incorrect and the proofs for 
many parts in the paper are left to the reader, we explain the proofs of necessary materials.

\subsection{Stable biserial algebras}
We start by introducing stably biserial algebras. Our goal is Proposition \ref{characterization of stably biserial}.

\begin{defn}
Let $Q$ be a quiver, $I$ an admissible ideal of $\bR Q$. The algebra $A=\bR Q/I$ is called \emph{stably biserial}
if the following conditions are satisfied.
\begin{itemize}
\item[(a)]
$A$ is a self-injective $\bR$-algebra. In paricular, the socle of the right regular representation and the left regular representation coincide,
which we denote by $\Soc(A)$.
\item[(b)]
For each vertex $i\in Q_0$, the number of outgoing arrows and the number of incoming arrows are less than or equal to $2$.
\item[(c)]
For each arrow $\alpha\in Q_1$, there is at most one arrow $\beta$ that satisfies
$$ \alpha\beta\not\in \alpha\Rad(A)\beta+\Soc(A). $$
\item[(d)]
For each arrow $\alpha\in Q_1$, there is at most one arrow $\beta$ that satisfies
$$ \beta\alpha\not\in \beta\Rad(A)\alpha+\Soc(A). $$
\end{itemize}
\end{defn}

The following is clear.

\begin{lemma}\label{at most one arrow}
Suppose that $A=\bR Q/I$ is stably biserial. Then the following hold.
\begin{itemize}
\item[(1)]
If arrows $\alpha, \beta, \gamma$ satisfy $\alpha\beta\ne0, \alpha\gamma\ne0, \beta\ne\gamma$ then either
$$
\alpha\beta\in \alpha\Rad(A)\beta+\Soc(A) \quad\text{or}\quad
\alpha\gamma\in \alpha\Rad(A)\gamma+\Soc(A).
$$
\item[(2)]
If arrows $\alpha, \beta, \gamma$ satisfy $\beta\alpha\ne0, \gamma\alpha\ne0, \beta\ne\gamma$ then either
$$
\beta\alpha\in \beta\Rad(A)\alpha+\Soc(A)\quad \text{or}\quad
\gamma\alpha\in \gamma\Rad(A)\alpha+\Soc(A).
$$
\end{itemize}
\end{lemma}

\begin{lemma}
Suppose that $A=\bR Q/I$ is stably biserial.
If two arrows $\beta$ and $\gamma$ start from the endpoint of an arrow $\alpha$ such that
\begin{itemize}
\item[(a)]
$\alpha\beta\not\in \Soc(A)$,
\item[(b)]
there is $r\in \Rad(A)$ such that $\alpha(1-r)\beta\in \Soc(A)$,
\end{itemize}
then the following hold.
\begin{itemize}
\item[(1)]
We have $\alpha(1-r')\gamma\not\in \Soc(A)$, for all $r'\in \Rad(A)$.
\item[(2)]
$\alpha\gamma\not\in \Soc(A)$.
\end{itemize}
\end{lemma}
\begin{proof}
(1) The assumption (a) implies that there exists an arrow $\delta$ such that $\alpha\beta\delta\ne0$. On the
other hand, (b) implies $\alpha\beta\delta=\alpha r\beta\delta$, and we may assume that $r$ is a linear combination
of loops and cycles that starts and ends at the endpoint of $\alpha$. As the number of outgoing arrows is at most two,
we may write
$$ r=\beta a_1+ \gamma a_2 \quad(a_1, a_2\in A). $$
Suppose that $\alpha(1-r')\gamma\in \Soc(A)$, for some $r'\in \Rad(A)$. Then
$\alpha\gamma\delta=\alpha{r'}\gamma\delta$ and
$$ r'=\beta a'_1+ \gamma a'_2 \quad(a'_1, a'_2\in A). $$
Now, the following repeated use of $r=\beta a_1+ \gamma a_2$ and $r'=\beta a'_1+ \gamma a'_2$ makes the
length of paths that appear on the right hand side longer and longer, so that we may conclude that
$\alpha\beta\delta=0$, which is a contradiction.
\begin{align*}
\alpha\beta\delta&=\alpha{r}\beta\delta\\
&=\alpha\beta a_1\beta\delta+\alpha\gamma a_2\beta\delta\\
&=\alpha{r}\beta a_1\beta\delta+\alpha{r'}\gamma a_2\beta\delta\\
&=\cdots\cdots
\end{align*}
Hence, $\alpha(1-r')\gamma\not\in \Soc(A)$, for all $r'\in \Rad(A)$.

\bigskip
(2) We assume that $\alpha\gamma\in \Soc(A)$. Then $\alpha\gamma\Rad(A)=0$ and the similar argument shows
\begin{align*}
\alpha(1-r)\gamma\delta&=-\alpha{r}\gamma\delta=-\alpha(\beta a_1)\gamma\delta\\
&=-\alpha{r}\beta a_1\gamma\delta=-\alpha(\beta a_1)\beta a_1\gamma\delta \\
&=-\alpha{r}\beta a_1\beta a_1\gamma\delta=\cdots\cdots\\
&=0.
\end{align*}
Thus, we have $\alpha(1-r)\gamma\in \Soc(A)$. But (1) says that $\alpha(1-r)\gamma\not\in \Soc(A)$ and
we conclude that $\alpha\gamma\not\in \Soc(A)$.
\end{proof}

We may prove the following lemma by the same proof.

\begin{lemma}
Suppose that $A=\bR Q/I$ is stably biserial.
If two arrows $\beta$ and $\gamma$ end at the initial point of an arrow $\alpha$ such that
\begin{itemize}
\item[(a)]
$\beta\alpha\not\in \Soc(A)$,
\item[(b)]
there is $r\in \Rad(A)$ such that $\beta(1-r)\alpha\in \Soc(A)$,
\end{itemize}
then the following hold.
\begin{itemize}
\item[(1)]
We have $\gamma(1-r')\alpha\not\in \Soc(A)$, for all $r'\in \Rad(A)$.
\item[(2)]
$\gamma\alpha\not\in \Soc(A)$.
\end{itemize}
\end{lemma}

\begin{prop}\label{stably biserial}
If $A=\bR Q/I$ is stably biserial then we may choose the presentation of $A$ in such a way that
the following (1) and (2) hold.
\begin{itemize}
\item[(1)]
If $\alpha\beta\ne0, \alpha\gamma\ne0, \beta\ne\gamma$, for arrows $\alpha, \beta, \gamma$, then either
$\alpha\beta\in \Soc(A)$ or $\alpha\gamma\in \Soc(A)$.
\item[(2)]
If $\beta\alpha\ne0, \gamma\alpha\ne0, \beta\ne\gamma$, for arrows $\alpha, \beta, \gamma$, then either
$\beta\alpha\in \Soc(A)$ or $\gamma\alpha\in \Soc(A)$.
\end{itemize}
\end{prop}
\begin{proof}
Suppose that arrows $\alpha, \beta, \gamma$ are such that $\beta\ne\gamma$,
$\alpha\beta\not\in \Soc(A), \alpha\gamma\not\in \Soc(A)$.
Then, Lemma \ref{at most one arrow}(1) shows that there is $r\in\Rad(A)$ such that
$\alpha(1-r)\beta$ or $\alpha(1-r)\gamma$ belongs to $\Soc(A)$. As the argument is the same,
we assume that $\alpha(1-r)\beta\in \Soc(A)$. We may also assume that
$r$ is a linear combination of loops and cycles which start at the endpoint of $\alpha$. Thus, if
$i$ is the initial point of $\alpha$ and $j$ is the endpoint of $\alpha$, then we have
$$
e_k\alpha(1-r)=\delta_{ik}\alpha(1-r), \;\; \alpha(1-r)e_k=\delta_{jk}\alpha(1-r).
$$
It implies that we have a well-defined algebra homomorphism
$p: \bR Q \longrightarrow A$ defined by $\eta\mapsto \eta$, for arrows $\eta\ne\alpha$, and $\alpha\mapsto \alpha(1-r)$. Let
$I'=\Ker(p)$ and $A'=\bR Q/I'$. If arrows $\rho, \kappa, \eta$ are such that
$\rho\ne\alpha, \rho\kappa\in \Soc(A)$ or $\rho\eta\in \Soc(A)$ then $\rho\kappa\in \Soc(A')$ or $\rho\eta\in \Soc(A')$ holds because
$\Soc(A)(1-r)=\Soc(A)$. In this way, we may decrease
$$
\sharp\{(\alpha,\beta,\gamma) \mid \alpha\beta, \alpha\gamma\not\in\Soc(A), \beta\ne\gamma\}+
\sharp\{(\alpha,\beta,\gamma) \mid \beta\alpha, \gamma\alpha\not\in\Soc(A), \beta\ne\gamma\}
$$
to zero.
\end{proof}

\begin{cor}\label{biserialness}
Suppose that $A$ has an anti-involution which makes irreducible modules self-dual. 
If $A$ is stably biserial then $A$ is biserial and $A/\Soc(A)$ is special biserial.
\end{cor}
\begin{proof}
$A/\Soc(A)$ is special biserial by Proposition \ref{stably biserial}. Let $P$ be
an indecomposable projective $A$-module. Since $\Rad(P)/\Soc(P)$ is union of
two uniserial submodules, so is $\Rad(P)$ and the assumption implies that $A$ is biserial.
\end{proof}

In \cite[Thm.2.6]{Po94}, the author asserts that stably biserial algebras are special biserial. As we show in the next example, 
this assertion fails even for symmetric two-point stably biserial algebras which has an anti-involution making two 
irreducible modules self-dual. We thank the referee for pointing out the failure of  \cite[Thm.2.6]{Po94} by providing us 
with the local algebra with two loops $\alpha, \beta$ obeying the relations 
$\alpha^2=(\alpha\beta)^2=(\beta\alpha)^2$ and $\beta^2=0$.

\begin{Ex}
We consider the quiver
$$
Q=
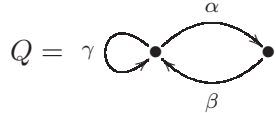
\begin{xy} (0,0) *{\bullet}="A", (15,0) *{\bullet}="B", \ar @(lu,ld) "A";"A"_{\gamma} \ar @/^4mm/ "A";"B"^\alpha \ar @/^4mm/ "B";"A"^\beta 
\end{xy}
$$
with relations $\gamma^2=\gamma\alpha\beta=\alpha\beta\gamma$, $\beta\gamma\alpha=\beta\alpha$, $\alpha\beta\alpha=\beta\alpha\beta=0$. 
It has the following basis. 
$$ \{ e_1, e_2, \alpha, \beta, \gamma, \alpha\beta, \beta\alpha, \gamma^2, \gamma\alpha, \beta\gamma\} $$
Then, it is stably biserial but not special biserial. Defining the trace map by 
\begin{align*}
\Tr(x)&=1 \quad \text{for } x\in\{e_1, e_2, \alpha\beta, \beta\alpha, \gamma^2\}, \\
\Tr(x)&=0 \quad \text{for } x\in\{ \alpha, \beta, \gamma, \gamma\alpha, \beta\gamma\}.
\end{align*}
we know that it is symmetric. It has the anti-involution which fixes $e_1, e_2, \gamma$ elementwise and swaps $\alpha$ and $\beta$. 
The anti-involution makes two irreducible modules $\bR e_1$ and $\bR e_2$ self-dual. 
\end{Ex}

The following proposition is one of the key observations by Pogorza{\l}y in \cite{Po94}, and we have used this result to show that 
tame finite quiver Hecke algebras in affine type A, for parameter values $\lambda\ne0$, are biserial algebras.

\begin{prop}\label{characterization of stably biserial}
If a self-injective algebra $B$ satisfies the following three conditions, then $B$ is Morita equivalent to a stably biserial algebra.
\begin{itemize}
\item[(a)]
For each indecomposable projective module $P$, we have $\Rad(P)/\Soc(P)=X'\oplus X''$, where $X'\ne0$, such that
$\Top(X'), \Top(X''), \Soc(X'), \Soc(X'')$ are simple modules.
\item[(b)]
Let $X=X'$ or $X''$, and let $Q$ be the projective cover of $X$. Then $X$ is non-projective and we denote
$p: Q/\Soc(Q) \to X$. Suppose that $\Rad(Q)/\Soc(Q)=Y_1\oplus Y_2$, where $Y_1$ and $Y_2$ are indecomposable modules.
Then, for irreducible homomorphisms
$$
w_1: Y_1\to Q/\Soc(Q),\quad w_2: Y_2\to Q/\Soc(Q),
$$
$w_1p$ or $w_2p$ factors through a projective module.
\item[(c)]
Let $X=X'$ or $X''$, and let $Y_1$ and $Y_2$ be an indecomposable direct summand of
$\Rad(Q_1)/\Soc(Q_1)$ and $\Rad(Q_2)/\Soc(Q_2)$, for indecomposable projective modules $Q_1$ and $Q_2$, respectively.
Suppose that both $Y_1$ and $Y_2$ have $P$ as their projective covers and we denote
$p_i: P/\Soc(P)\to Y_i$, for $i=1,2$.
Then, for an irreducible homomorphism $w: X\to P/\Soc(P)$, $wp_1$ or $wp_2$ factors through a projective module.
\end{itemize}
\end{prop}
\begin{proof}
For each vertex $i\in Q_0$, (a) implies that the number of incoming arrows is
$$ \dim \Soc^2(I_i)/\Soc(I_i)=\dim \Soc(\Rad(I_i)/\Soc(I_i))\le 2. $$
Similarly, the number of outgoing arrows is
$$ \dim \Rad(P_i)/\Rad^2(P_i)=\dim \Top(\Rad(P_i)/\Soc(P_i))\le 2. $$

For each arrow $\alpha\in Q_1$, we have to show that there is at most one arrow $\beta$ such that
$$ \alpha\beta\not\in \alpha\Rad(B)\beta+\Soc(B). $$
We shall prove that if $\alpha\beta_1\ne0$ and $\alpha\beta_2\ne0$ in $B/\Soc(B)$, then
either $\alpha\beta_1\in \alpha\Rad(B/\Soc(B))\beta_1$ or $\alpha\beta_2\in \alpha\Rad(B/\Soc(B))\beta_2$.

Let $i$ and $j$ be the initial point and the endpoint of $\alpha$, respectively, and let $P_i$ and $P_j$ be the
corresponding indecomposable projective modules. We write
$$ X'\oplus X''=\Rad(P_j)/\Soc(P_j) $$
as in (a). Then, we have $\alpha=pw$, for
$p: P_i/\Soc(P_i)\to X$, where $X=X'$ or $X''$, and $w:X\hookrightarrow P_j/\Soc(P_j)$. Note that $w$ is an irreducible
homomorphism as it is a direct summand of the right minimal almost split homomorphism
$\Rad(P_j)/\Soc(P_j)\oplus P_j\to P_j/\Soc(P_j)$.

Similarly, we may write $\beta_1$ and $\beta_2$ as
\begin{align*}
\beta_1&: P_j/\Soc(P_j) \stackrel{p_1}{\rightarrow} Y_1 \hookrightarrow Q_1/\Soc(Q_1), \\
\beta_2&: P_j/\Soc(P_j) \stackrel{p_2}{\rightarrow} Y_2 \hookrightarrow Q_2/\Soc(Q_2),
\end{align*}
where $Q_i$ is indecomposable projective and $Y_i$ is an indecomposable direct summand of $\Rad(Q_i)/\Soc(Q_i)$, for $i=1,2$.
Then, by (c), we may assume that $wp_1$ factors through a projective-injective $B$-module, say $Q$. Thus, we have
$q: X\to Q$ and $r: Q\to Y_1$ such that $qr=wp_1$. Then we have $q=wt$ and $r=sp_1$ as follows.

\medskip
\hspace{4.5cm}
\begin{xy}
(0,20) *{X}="A", (25,20) *{P_j/\Soc(P_j)}="B", (25,5) *{Q}="C", (50,20) *{Y_1}="D",

\ar "A";"B"^{w}
\ar "A";"C"_{q}
\ar @/_ 4mm/ "B";"C"_{\exists t}
\ar "B";"D"^{p_1}
\ar "C";"D"_{r}
\ar @/_ 4mm/ "C";"B"_{\exists s}
\end{xy}

\medskip
\noindent
If $ts$ was an isomorphism, then $P_j/\Soc(P_j)$ would be a direct summand of $Q$, a contradiction. Thus,
$ts\in \Rad\End_B(P_j/\Soc(P_j))$, and it follows that $\alpha\beta_1\in\alpha\Rad(B/\Soc(B))\beta_1$.

To show that there is at most one arrow $\beta$ such that
$$ \beta\alpha\not\in \beta\Rad(B)\alpha+\Soc(B), $$
we prove that if $\beta_1\alpha\ne0$ and $\beta_2\alpha\ne0$ in $B/\Soc(B)$, then
either $\beta_1\alpha\in \beta_1\Rad(B/\Soc(B))\alpha$ or $\beta_2\alpha\in \beta_2\Rad(B/\Soc(B))\alpha$.
Let
$$ \alpha: P_i/\Soc(P_i) \stackrel{p}{\rightarrow} X \hookrightarrow P_j/\Soc(P_j) $$
as before, and we write $\beta_1$ and $\beta_2$ as
\begin{align*}
\beta_1&: Q_1/\Soc(Q_1) \twoheadrightarrow Y_1 \hookrightarrow P_i/\Soc(P_i), \\
\beta_2&: Q_1/\Soc(Q_2) \twoheadrightarrow Y_2 \hookrightarrow P_i/\Soc(P_i),
\end{align*}
where $Q_1$ and $Q_2$ are indecomposable projective modules, and $\Rad(P_i)/\Soc(P_i)=Y_1\oplus Y_2$.

We denote $w_1: Y_1\hookrightarrow P_i/\Soc(P_i)$ and $w_2: Y_2\hookrightarrow P_i/\Soc(P_i)$. Then,
they are irreducible homomorphisms, and by (b), we may assume that $w_1p$ factors through a projective-injective module, say $Q$ again.
Thus,

\medskip
\hspace{4.5cm}
\begin{xy}
(0,20) *{Y_1}="A", (25,20) *{P_i/\Soc(P_i)}="B", (25,5) *{Q}="C", (50,20) *{X}="D",

\ar "A";"B"^{w_1}
\ar "A";"C"
\ar @/_ 4mm/ "B";"C"_{\exists t}
\ar "B";"D"^{p}
\ar "C";"D"
\ar @/_ 4mm/ "C";"B"_{\exists s}
\end{xy}

\medskip
\noindent
and $ts\in \Rad\End_B(P_i/\Soc(P_i))$. It follows that $\beta_1\alpha\in \beta_1\Rad(B/\Soc(B))\alpha$.
\end{proof}

\subsection{Notions for stable module categories}
We introduce several useful notions for stable module categories of self-injective algebras. They are, systems of orthogonal stable bricks,
s-top and s-socle of non-projective modules, s-projective and
s-injective modules.

\begin{defn}
Let $A$ be a self-injective algebra. A collection of indecomposable $A$-modules
${\mathcal M}=\{ M_i\}_{i\in I}$ is a \emph{system of orthogonal stable bricks}, or \emph{s.o.s.b.} for short, if
$\tau(M_i)\not\simeq M_i$, for all $i$, and the following holds in the stable category of $A$-modules.
$$
\underline\Hom_B(M_i,M_j)=
\begin{cases} \bR \quad&(i=j)\\ 0 \quad&(i\ne j)\end{cases}
$$
We call the cardinality of $I$ the \emph{stable rank} of the s.o.s.b. ${\mathcal M}=\{ M_i\}_{i\in I}$, and we denote it by
$\rank({\mathcal M})$. A s.o.s.b. ${\mathcal M}$ is called a \emph{maximal s.o.s.b.} if we have
$$
\underline\Hom_A(\oplus_{i\in I} M_i, N)\ne0 \quad \text{and}\quad
\underline\Hom_A(N, \oplus_{i\in I} M_i)\ne0,
$$
for all indecomposable non-projective $A$-modules $N$ such that $\tau(N)\not\simeq N$.
\end{defn}
%
%We also define
%$$ {\rm sr}(A)=\sup\{ \rank{\mathcal M} \mid \text{${\mathcal M}$ is a maximal s.o.s.b.}\}. $$
%

Let ${\mathcal M}=\{ M_i\}_{i\in I}$ be a maximal s.o.s.b..
If $\underline\Hom_A(N, \oplus_{i\in I} M_i)=\bR$, there is a unique $i\in I$ such that $\underline\Hom_A(N,M_i)\ne0$.
If this is the case, we write $s\text{-}\Top(N)=M_i$. Similarly,
if $\underline\Hom_A(\oplus_{i\in I} M_i, N)=\bR$, there is a unique $i\in I$ such that $\underline\Hom_A(M_i, N)\ne0$.
If this is the case, we write $s\text{-}\Soc(N)=M_i$.

\begin{defn}
Let ${\mathcal M}=\{ M_i\}_{i\in I}$ be a maximal s.o.s.b.. An indecomposable non-projective $A$-module $N$ is
\emph{s-projective with respect to ${\mathcal M}$} if
\begin{itemize}
\item[(i)]
$\tau(N)\not\simeq N$.
\item[(ii)]
$\underline\Hom_A(N, \oplus_{i\in I} M_i)=\bR$.
\item[(iii)]
For any indecomposable non-projective $A$-module $X$ and $0\ne \underline f\in\underline\Hom_A(X, s\text{-}\Top(N))$, there exists
$0\ne \underline{g}\in \underline\Hom_A(N,X)$ such that $\underline{g}\underline{f}\ne0$.
\end{itemize}

Dually, an indecomposable non-projective $A$-module $N$ is \emph{s-injective with respect to ${\mathcal M}$} if
\begin{itemize}
\item[(i)]
$\tau(N)\not\simeq N$.
\item[(ii)]
$\underline\Hom_A(\oplus_{i\in I} M_i, N)=\bR$.
\item[(iii)]
For any indecomposable non-projective $A$-module $X$ and $0\ne \underline f\in\underline\Hom_A(s\text{-}\Soc(N), X)$, there exists
$0\ne \underline{g}\in \underline\Hom_A(X,N)$ such that $\underline{f}\underline{g}\ne0$.
\end{itemize}
\end{defn}

\begin{prop}\label{tau-orbit}
Let $B$ be an indecomposable self-injective algebra which is not a local Nakayama algebra. Then, we have the following.
\begin{itemize}
\item[(1)]
If $P$ is indecomposable projective, then $\tau(P/\Soc(P))\not\simeq P/\Soc(P)$.
\item[(2)]
If $S$ is irreducible, then $S$ is non-projective and $\tau(S)\not\simeq S$.
\end{itemize}
\end{prop}
\begin{proof}
(1) If $\tau(P/\Soc(P))\simeq P/\Soc(P)$ then the almost split sequence
$$
0 \to \Rad(P) \to \Rad(P)/\Soc(P)\oplus P \to P/\Soc(P) \to 0
$$
tells $P/\Soc(P)\simeq \Rad(P)$ and we have a surjective homomorphism $P\to \Rad(P)$.
Hence we have surjective homomorphisms $\Rad^i(P)\to \Rad^{i+1}(P)$, for $i\ge1$.
As a result, $P$ is uniserial and all of the composition factors are isomorphic to $S=\Top(P)$.
If there is another indecomposable projective module $Q$, then the indecomposability of $B$
implies that we have a uniserial module of length two with composition factors $S$ and $T=\Top(Q)$.
But then it is either a submodule of $P$ or a quotient module of $P$, which contradicts
$[P:T]=0$. Therefore, $B$ is a local Nakayama algebra, which we have excluded in the assumption.

\medskip
(2) If $S$ was projective, $B$ would be a local Nakayama algebra. Thus, $S$ is non-projective.
Suppose that $\tau(S)\simeq S$. We set $X_0=S$. Then we have an almost split sequence
$$
0 \to X_0 \to X_1 \to X_0 \to 0.
$$
$X_1$ is indecomposable as it is uniserial.
If $X_1$ was projective, then $B$ would be a local Nakayama algebra. Thus, $X_1$ is non-projective. Let
$X_1\to M_1$ be a left minimal almost split homomorphism. Then, the irreducible homomorphism
$X_1\to X_0$ is a direct summand and we may write $M_1=X_0\oplus X_2$, for some module $X_2$.
Thus, $X_0\to \tau^{-1}(X_1)$ is an irreducible homomorphism and it is a direct summand of
the left minimal almost split homomorphism $X_0\to X_1$. We conclude that $\tau^{-1}(X_1)\simeq X_1$ and
we have the following almost split sequence
$$
0 \to X_1 \to X_0\oplus X_2 \to X_1 \to 0.
$$

Suppose that we have $B$-modules $X_0,\dots,X_{i+1}$ such that
\begin{itemize}
\item[(i)]
$\Top(X_k)\simeq S$, for $0\le k\le i$.
\item[(ii)]
We have almost split sequences
$$ 0 \to X_k \to X_{k-1}\oplus X_{k+1} \to X_k \to 0, $$
for $0\le k\le i$, where we understand $X_{-1}=0$.
\item[(iii)]
All the composition factors of $X_k$ are $S$ and $[X_k: S]=k+1$, for $0\le k\le i$.
\end{itemize}
Note that (i)-(iii) hold if $i=1$. We show that (i)-(iii) imply $\Top(X_{i+1})\simeq S$.
Consider
$$
0 \to \Hom_B(X_i,S)\to \Hom_B(X_{i-1},S)\oplus \Hom_B(X_{i+1},S)\to \Hom_B(X_i,S)=\bR.
$$
Then, $\Hom_B(X_{i+1},S)=\bR$ and, noting that (iii) holds for $k=i+1$, $\Top(X_{i+1})\simeq S$ follows.
Therefore, (i) and (iii) hold for $k=i+1$. Next we show that if $X_{i+1}$ is non-projective then we may increment $i$.
Indeed, if $X_{i+1}$ is non-projective then
we may take a left minimal almost split homomorphism $X_{i+1}\to M_{i+1}$, and the irreducible
homomorphism $X_{i+1}\to X_i$ is a direct summand. Thus, we may write $M_{i+1}=X_i\oplus X_{i+2}$, for some $B$-module
$X_{i+2}$, and we have the almost split sequence
$$
0 \to X_{i+1} \to X_i\oplus X_{i+2} \to \tau^{-1}(X_{i+1}) \to 0.
$$
Then, the irreducible homomorphism $X_i\to \tau^{-1}(X_{i+1})$ is a direct summand of the left minimal almost split
homomorphism $X_i\to X_{i-1}\oplus X_{i+1}$, and we have either $\tau^{-1}(X_{i+1})\simeq X_{i-1}$ or
$\tau^{-1}(X_{i+1})\simeq X_{i+1}$. But $\tau^{-1}(X_{i+1})\simeq X_{i-1}$ implies that
$X_{i+1}\simeq \tau(X_{i-1})\simeq X_{i-1}$, which contradicts $[X_k:S]=k+1$, for $k=i\pm1$. Thus, we have
$\tau(X_{i+1})\simeq X_{i+1}$ and (ii) for $k=i+1$ holds. As $\dim X_i$ grows, $X_{i+1}$ becomes a projective $B$-module at some $i$,
and we conclude that the projective cover $P$ of $S$ is uniserial and all the composition factors of $P$ are $S$.
It follows that $B$ is a local Nakayama algebra. Therefore, $\tau(S)\not\simeq S$ as desired.
\end{proof}

\begin{prop}\label{sosb}
Let $A$ and $B$ be indecomposable self-injective $\bR$-algebras which are not Nakayama algebras. Let $P_1,\dots,P_n$ be
a complete set of indecomposable projective $B$-modules, and we denote $S_i=\Top(P_i)$, for $1\le i\le n$.
Suppose that a functor $\Psi:B\text{-}\underline{\rm mod}\to A\text{-}\underline{\rm mod}$ gives stable equivalence.
Then, we have the following.
\begin{itemize}
\item[(1)]
Let $M_i=\Psi(S_i)$, for $1\le i\le n$. Then ${\mathcal M}=\{ M_i\}_{1\le i\le n}$ is a maximal s.o.s.b..
\item[(2)]
Let $N_i=\Psi(P_i/\Soc(P_i))$. Then, $N_i$ is s-projective and $s\text{-}\Top(N_i)\simeq M_i$, for $1\le i\le n$.
\end{itemize}
\end{prop}
\begin{proof}
(1) The stable Auslander-Reiten quivers of $A$ and $B$ coincide by \cite[X. Cor.1.9]{ARS}
and it follows from  Proposition \ref{tau-orbit}(2) that
$\tau(M_i)\not\simeq M_i$, for $1\le i\le n$. Since
$$ \underline\Hom_A(M_i,M_j)\simeq \underline\Hom_B(S_i,S_j)=\begin{cases} \bR \quad&(i=j) \\ 0 \quad&(i\ne j) \end{cases} $$
${\mathcal M}$ is a system of orthogonal stable bricks. As it is not difficult to prove
$$ \underline\Hom_A(\oplus_{i=1}^n M_i, N)\ne0,\quad \underline\Hom_A(N, \oplus_{i=1}^n M_i)\ne0, $$
for indecomposable non-projective $A$-modules $N$ such that $\tau(N)\not\simeq N$, ${\mathcal M}$ is maximal.

\medskip
(2) We check the conditions (i)(ii)(iii) from the definition of s-projectivity. (i) follows from Proposition \ref{tau-orbit}(1). (ii) is clear and
$s\text{-}\Top(N_i)\simeq M_i$.
Let $0\ne \underline{f}:X\to s\text{-}\Top(N_i)$, for an indecomposable non-projective $A$-module $X$. Then, we have a
surjective homomorphism $\Psi^{-1}(X)\to S_i$, and $g:P_i\to \Psi^{-1}(X)$ such that their composition equals
the surjective homomorphism $p_i: P_i\to S_i$.

\medskip
\hspace{4.5cm}
\begin{xy}
(0,20) *{\Psi^{-1}(X)}="A", (20,20) *{S_i}="B", (20,0) *{P_i}="C",
\ar "A";"B"
\ar "C";"A"^{g}
\ar "C";"B"_{p_i}
\end{xy}

\medskip
\noindent
If $g(\Soc(P_i))\ne0$ then $g(P_i)\simeq P_i$ and we obtain
$P_i\simeq \Psi^{-1}(X)$, a contradiction. Thus, $g$ induces $P_i/\Soc(P_i)\to X$, and it follows (iii).
\end{proof}

\begin{prop}\label{s-projectives}
Let $A$ be a self-injective $\bR$-algebra, ${\mathcal M}=\{ M_i\}_{i\in I}$ a maximal s.o.s.b..
Then, we have the following.
\begin{itemize}
\item[(1)]
$\tau^{-1}\Omega(M_i)$ is s-projective, for all $i\in I$.
\item[(2)]
If $N$ is s-projective such that $s\text{-}\Top(N)\simeq M_i$, then $N\simeq \tau^{-1}\Omega(M_i)$.
\end{itemize}
\end{prop}
\begin{proof}
(1) Note that $\tau^{-1}\Omega(M_i)\simeq \Omega(M_i)$ if and only if $\tau(M_i)\simeq M_i$. Thus, $\tau(N_i)\not\simeq N_i$, for
$N_i=\tau^{-1}\Omega(M_i)$, and the condition (i) is satisfied. Let $P$ be a projective $A$-module such that
$$ 0\to \Omega(M_i) \to P \to M_i \to 0. $$
If $M=\oplus_{i\in I} M_i$ or $M=M_i$, then we have
$$
0\to \Hom_A(M,\Omega(M_i)) \to \Hom_A(M,P) \to \Hom_A(M,M_i) \to \Ext_A^1(M,\Omega(M_i)) \to 0.
$$
Thus, $\bR=\underline\Hom_A(M,M_i)\simeq \Ext_A^1(M,\Omega(M_i))$ and it follows that
$$
\underline\Hom_A(N_i,M)=
\underline\Hom_A(\tau^{-1}\Omega(M_i),M)\simeq D\Ext_A^1(M,\Omega(M_i))=\bR,
$$
where $D=\Hom_\bR(-,\bR)$. Hence, the condition (ii) is satisfied and $s\text{-}\Top(N_i)\simeq M_i$.

For $0\ne \underline{f}\in \underline\Hom_A(X, M_i)$, for an indecomposable non-projective $A$-module $X$, we find
$g: N_i=\tau^{-1}\Omega(M_i)\to X$ such that $\underline{g}\underline{f}\ne0$. Let $w:P\to X\oplus P$ be the natural
inclusion, and we define a homomorphism $j: \Omega(M_i)\to Y$ by the following commutative diagram.

\medskip
\hspace{4.5cm}
\begin{xy}
(0,15) *{0}="A", (15,15) *{\Omega(M_i)}="B", (30,15) *{P}="C", (45,15) *{M_i}="D", (60,15) *{0}="E",
(0,0) *{0}="AA", (15,0) *{Y}="BB", (30,0) *{X\oplus P}="CC", (45,0) *{M_i}="DD", (60,0) *{0}="EE",

\ar "A";"B"
\ar "B";"C"
\ar "C";"D"^{\ell}
\ar "D";"E"

\ar "AA";"BB"
\ar "BB";"CC"
\ar "CC";"DD"_{(f,\ell)}
\ar "DD";"EE"

\ar "B";"BB"^{j}
\ar "C";"CC"^{w}
\ar @{=} "D";"DD"
\end{xy}

\medskip
\noindent
If $j$ is split mono, then $Y=\Omega(M_i)\oplus Y'$, for an $A$-submodule $Y'$ of $Y$. Then
$$
0 \rightarrow Y/\Omega(M_i) \stackrel{\iota}{\rightarrow} X\oplus M_i \stackrel{(f,\id_{M_i})}{\rightarrow} M_i \rightarrow 0
$$
gives $\iota: Y'\simeq Y/\Omega(M_i)\simeq X'=\{(x,-f(x)) \mid x\in M_i\}$. Therefore, we have the following commutative diagram
where $X\oplus M_i\to M_i$ is the projection to the second factor.

\medskip
\hspace{4.5cm}
\begin{xy}
(0,15) *{X}="A", (15,15) *{Y'}="B", (40,15) *{X\oplus P}="C",
(15,0) *{X'}="BB", (40,0) *{X\oplus M_i}="CC",
(40,-15) *{M_i}="CCC",

\ar "A";"B"
\ar "B";"C"^{\rm incl}
\ar "B";"BB"^{\iota}
\ar "BB";"CC"^{\rm incl}
\ar "C";"CC"
\ar "CC";"CCC"
\end{xy}

\medskip
\noindent
In the diagram, $X\to M_i$ is $-f$ and the vertical homomorphism $X\oplus P\to M_i$ factors through $P$. Thus,
$f$ factors through a projective module and it contradicts $\underline{f}\ne0$. We conclude that $j$ is not split mono.
On the other hand, the snake lemma implies
$$
0=\Ker(\id_{M_i}) \to \Coker(j) \to \Coker(w)\to \Coker(\id_{M_i})=0
$$
Hence, $\Coker(j)\simeq X$ and we have the exact sequence
$$ 0 \to \Omega(M_i) \stackrel{j}{\to} Y \to X \to 0. $$
We consider the almost split sequence $0 \to \Omega(M_i) \to Z \to \tau^{-1}\Omega(M_i)\to 0$. Then we may define
$t: Z\to Y$ and $g:N_i=\tau^{-1}\Omega(M_i)\to X$ as follows, because $j$ is not split mono.

\medskip
\hspace{4.5cm}
\begin{xy}
(0,15) *{0}="A", (15,15) *{\Omega(M_i)}="B", (30,15) *{Z}="C", (45,15) *{N_i}="D", (60,15) *{0}="E",
(0,0) *{0}="AA", (15,0) *{\Omega(M_i)}="BB", (30,0) *{Y}="CC", (45,0) *{X}="DD", (60,0) *{0}="EE",

\ar "A";"B"
\ar "B";"C"^{i}
\ar "C";"D"^{p}
\ar "D";"E"

\ar "AA";"BB"
\ar "BB";"CC"_{j}
\ar "CC";"DD"_{s}
\ar "DD";"EE"

\ar "C";"CC"^{t}
\ar "D";"DD"^{g}
\ar @{=} "B";"BB"
\end{xy}

\medskip
\noindent
We shall prove that $\underline{g}\underline{f}\ne0$. Suppose that $gf$ factors through a projective $A$-module. Then, it factors
through $\ell: P\to M_i$ and we may write $gf=h\ell$, for some  $h: N_i\to P$. Thus,
$(g,-h): N_i \to X\oplus P$ factors through $Y\to X\oplus P$, because $(g,-h)(f,\ell)=gf-h\ell=0$.

\medskip
\hspace{4.5cm}
\begin{xy}
(30,30) *{N_i}="X",
(0,15) *{0}="A", (15,15) *{Y}="B", (30,15) *{X\oplus P}="C", (45,15) *{M_i}="D", (60,15) *{0}="E",
(-15,0) *{0}="AA", (0,0) *{\Omega(M_i)}="BB", (15,0) *{Y}="CC", (30,0) *{X}="DD", (45,0) *{0}="EE",

\ar "X";"C"^{(g,-h)}
\ar @{.>} "X";"B"_{\exists h'}
\ar "A";"B"
\ar "B";"C"
\ar "C";"D"_{(f,\ell)}
\ar "D";"E"

\ar "AA";"BB"
\ar "BB";"CC"^{j}
\ar "CC";"DD"^{s}
\ar "DD";"EE"

\ar @{=} "B";"CC"
\ar "DD";"C"_{\rm incl}
\end{xy}

\medskip
\noindent
It follows that $g=h's$ and $(t-ph')s=ts-pg=0$. In particular,  $\Im(t-ph')\subseteq \Omega(M_i)$. On the other hand,
$ip=0$ implies $i(t-ph')=it=j$ and $\Im(t-ph')=\Omega(M_i)$. Thus, $i(t-ph')$ is an isomorphism of $\Omega(M_i)$ and
it implies that $0\to \Omega(M_i) \to Z \to N_i \to 0$ splits, which is a contradiction.
Therefore, $\underline{g}\underline{f}\ne0$ and we have proved that $N_i$ is s-projective.

\medskip
(2) Let $N$ be s-projective such that $s\text{-}\Top(N)=M_i$. Then, for the homomorphism
$$ f: \tau^{-1}\Omega(M_i)\to M_i $$
such that $\underline{f}\ne0$,
there exists $g:N\to \tau^{-1}\Omega(M_i)$ such that $\underline{g}\underline{f}\ne0$. Suppose that $g$ is not an
isomorphism. Then, we may define $h: Z\to P$ and $f': \tau^{-1}\Omega(M_i)\to M_i$ as follows.

\medskip
\hspace{4.5cm}
\begin{xy}
(45,30) *{N}="X",
(0,15) *{0}="A", (15,15) *{\Omega(M_i)}="B", (30,15) *{Z}="C", (45,15) *{\tau^{-1}\Omega(M_i)}="D", (60,15) *{0}="E",
(0,0) *{0}="AA", (15,0) *{\Omega(M_i)}="BB", (30,0) *{P}="CC", (45,0) *{M_i}="DD", (60,0) *{0}="EE",

\ar "X";"D"^{g}
\ar @{.>} "X";"C"
\ar "A";"B"
\ar "B";"C"^{i}
\ar "C";"D"^{p}
\ar "D";"E"

\ar "AA";"BB"
\ar "BB";"CC"
\ar "CC";"DD"_{\ell}
\ar "DD";"EE"

\ar @{=} "B";"BB"
\ar "C";"CC"^{h}
\ar "D";"DD"^{f'}
\end{xy}

\medskip
\noindent
If $\underline{f}'=0$ then it factors through $\ell:P\to M_i$ and we have $f'':\tau^{-1}\Omega(M_i)\to P$ such that
$f'=f''\ell$. Then, $(pf''-h)\ell=pf'-h\ell=0$ implies $pf''-h: Z\to \Omega(M_i)$ and
$i(pf''-h)=-ih$ implies $\Im(pf''-h)=\Omega(M_i)$. Thus, $0\to \Omega(M_i) \to Z \to N_i \to 0$ splits, a contradiction.
Therefore, $\underline{f}'\ne0$ and it is a scalar multiple of $\underline{f}$. But the above diagram shows that
$\underline{g}\underline{f}'=0$ and we have $\underline{g}\underline{f}=0$, which contradicts
$\underline{g}\underline{f}\ne0$. We have proved $g:N\simeq \tau^{-1}\Omega(M_i)$.
\end{proof}

%%%%%%%%%%%%%%%%%%%%%%%%%%%%%%%%%%%%%%%%%%%%%%%%%%%%%%%%%%%%%%%%%%%

\bibliographystyle{amsplain}

%%%%%%%%%%%%%%%%%%%%%%%%%%%%%%%%%%%%%%%%%%%%%%%%%%%%%%%%%%%%%%%%%%%

\end{document}